\documentclass[10pt]{article}

\usepackage{etex}
\usepackage[utf8]{inputenc}
\usepackage{a4wide}
\usepackage{graphicx}
\usepackage{wrapfig}
\usepackage[hmarginratio={1:1},     
vmarginratio={1:1},     
textwidth=16cm,        
textheight=22cm,
heightrounded,]{geometry}
\usepackage{amsmath}
\usepackage{amsthm}
\usepackage{amssymb}
\usepackage{mathrsfs, mathtools}
\usepackage{overpic}
\usepackage{subcaption}
\usepackage{enumitem}
\usepackage[all,cmtip]{xy}

\let\emph\undefined
\newcommand{\emph}[1]{\textsl{#1}}

\usepackage{hyperref}
\usepackage{mathtools}
\mathtoolsset{showonlyrefs}
\numberwithin{equation}{section}

\newtheoremstyle{style1}
{13pt}
{13pt}
{}
{}
{\normalfont\bfseries}
{.}
{.5em}
{}

\theoremstyle{style1}

\newtheorem{definition}{Definition}[section]
\newtheorem{example}[definition]{Example}
\newtheorem{remark}[definition]{Remark}

\newcommand{\catf}[1]{{\mathsf{#1}}}

\usepackage{tocloft}

\newtheoremstyle{style2}
{13pt}
{13pt}
{\slshape}
{}
{\normalfont\bfseries}
{.}
{.5em}
{}

\theoremstyle{style2}

\newtheorem{lemma}[definition]{Lemma}
\newtheorem{theorem}[definition]{Theorem}
\newtheorem{proposition}[definition]{Proposition}
\newtheorem{corollary}[definition]{Corollary}

\usepackage{tikz}
\usetikzlibrary{matrix,arrows,decorations.pathmorphing,shapes.geometric}
\usepackage{tikz-cd}

\newcommand{\hooklongrightarrow}{\lhook\joinrel\longrightarrow}
\newcommand\myrot[1]{\mathrel{\rotatebox[origin=c]{#1}{$\Longrightarrow$}}}
\newcommand\SWarrow{\myrot{-135}}

\newcommand{\R}{\mathbb{R}}
\newcommand{\C}{\mathbb{C}}
\newcommand{\Z}{\mathbb{Z}}

\newcommand{\Ca}{\mathcal{C}}
\newcommand{\Fa}{\mathcal{F}}
\newcommand{\Ga}{\mathcal{G}}

\newcommand{\Aa}{\mathcal{A}}

\newcommand{\Ma}{\mathcal{M}}

\newcommand{\Va}{\mathcal{V}}

\usepackage{needspace}

\newcommand{\Bun}{\operatorname{Bun}}

\newcommand{\Man}{\catf{Man}}

\newcommand{\Aut}{\operatorname{Aut}}

\newcommand{\End}{\operatorname{End}}
\newcommand{\act}{\operatorname{act}}
\newcommand{\reg}{\operatorname{reg}}
\newcommand{\Hom}{\operatorname{Hom}}

\newcommand{\id}{\operatorname{id}}
\newcommand{\Out}{\operatorname{Out}}
\newcommand{\pr}{\operatorname{pr}}

\newcommand{\fd}{\operatorname{fd}}
\newcommand{\ad}{\operatorname{ad}}
\newcommand{\Ad}{\operatorname{Ad}}
\newcommand{\Tr}{\operatorname{Tr}}

\newcommand{\rev}{\operatorname{rev}}

\newcommand{\SC}{\mathsf{SC}}
\newcommand{\fSC}{\mathsf{fSC}}
\newcommand{\QCoh}{\operatorname{QCoh}}

\newcommand{\Cat}{\catf{Cat}}

\newcommand{\DS}{\text{/\hspace{-0.1cm}/}}

\let\to\undefined
\newcommand{\to}{\longrightarrow}
\let\mapsto\undefined
\newcommand{\mapsto}{\longmapsto}

\newcommand{\Spaces}{\catf{Spaces}}
\newcommand{\Disk}{\catf{Disk}}

\newcommand{\Alg}{\catf{Alg}}

\newcommand{\Prc}{\catf{Pr}_c}
\newcommand{\DMan}{D\text{-}\catf{Man}_2}

\newcommand{\bBr}{\catf{bBr}}
\newcommand{\Rep}{\catf{Rep}}

\newcommand{\comp}{\mathsf{C}}

\DeclareMathSymbol{\Phiit}{\mathalpha}{letters}{"08} 
\DeclareMathSymbol{\Psiit}{\mathalpha}{letters}{"09}
\DeclareMathSymbol{\Sigmait}{\mathalpha}{letters}{"06}
\DeclareMathSymbol{\Xiit}{\mathalpha}{letters}{"04}
\DeclareMathSymbol{\Piit}{\mathalpha}{letters}{"05}\let\Pi\undefined\newcommand{\Pi}{\Piit}
\DeclareMathSymbol{\Gammait}{\mathalpha}{letters}{"00}
\DeclareMathSymbol{\Omegait}{\mathalpha}{letters}{"0A}\let\Omega\undefined\newcommand{\Omega}{\Omegait}
\DeclareMathSymbol{\Upsilonit}{\mathalpha}{letters}{"07}
\DeclareMathSymbol{\Thetait}{\mathalpha}{letters}{"02}
\DeclareMathSymbol{\Lambdait}{\mathalpha}{letters}{"03}\let\Lambda\undefined\newcommand{\Lambda}{\Lambdait}

\let\Phi\undefined\newcommand{\Phi}{\Phiit}
\let\Sigma\undefined\newcommand{\Sigma}{\Sigmait}
\let\Psi\undefined\newcommand{\Psi}{\Psiit}
\let\Gamma\undefined\newcommand{\Gamma}{\Gammait}

\newcommand{\G}{D}
\newcommand{\g}{d}

\usepackage{float}

\usepackage{tikzit}

\tikzstyle{black dot}=[fill=black, draw=black, shape=circle, minimum size=3pt, inner sep=0pt]
\tikzstyle{black dot small}=[fill=black, draw=black, shape=circle, minimum size=3pt, inner sep=0pt]
\tikzstyle{big white circle}=[fill=white, draw=black, shape=circle, minimum width=0.75cm]
\tikzstyle{white dot big}=[fill=white, draw=black, shape=circle, inner sep=1pt]
\tikzstyle{white dot}=[fill=white, draw=black, shape=circle, minimum size=3pt, inner sep=0pt]
\tikzstyle{flat box}=[fill=white, draw=black, shape=rectangle, minimum width=2.5cm, minimum height=0.5cm]
\tikzstyle{square}=[fill=white, draw=black, shape=rectangle]
\tikzstyle{flat box 2}=[fill=white, draw=black, shape=rectangle, minimum height=0.5cm, minimum width=1.0cm]
\tikzstyle{over }=[front]
\tikzstyle{theta}=[fill=blue, draw=blue, shape=ellipse, minimum height=6pt, minimum width=6pt, inner sep=0pt]
\tikzstyle{thetabig}=[fill=blue, draw=blue, shape=ellipse, minimum width=1cm, minimum height=0.01cm]
\tikzstyle{thetainv}=[fill=white, draw=blue, shape=ellipse, minimum height=6pt, minimum width=6pt, inner sep=0pt]
\tikzstyle{thetabinv}=[fill=white, draw=blue, shape=ellipse, minimum width=1cm, minimum height=0.01cm]
\tikzstyle{rec m}=[fill=white, draw=black, shape=rectangle, minimum width=0.75cm, minimum height=0.5cm]
\tikzstyle{rec s}=[fill=white, draw=black, shape=rectangle, minimum width=0.2cm, minimum height=0.2cm]

\tikzstyle{mid arrow}=[-, postaction={on each segment={mid arrow}}]
\tikzstyle{end arrow}=[->]
\tikzstyle{red mid arrow}=[-, draw={rgb,255: red,214; green,42; blue,51}, postaction={on each segment={mid arrow}}, line width=1pt]
\tikzstyle{blue}=[-, draw=black]
\tikzstyle{over}=[-, link]
\tikzstyle{mapsto}=[{|->}]

\usetikzlibrary{decorations.pathreplacing,decorations.markings}
\tikzset{
	on each segment/.style={
		decorate,
		decoration={
			show path construction,
			moveto code={},
			lineto code={
				\path [#1]
				(\tikzinputsegmentfirst) -- (\tikzinputsegmentlast);
			},
			curveto code={
				\path [#1] (\tikzinputsegmentfirst)
				.. controls
				(\tikzinputsegmentsupporta) and (\tikzinputsegmentsupportb)
				..
				(\tikzinputsegmentlast);
			},
			closepath code={
				\path [#1]
				(\tikzinputsegmentfirst) -- (\tikzinputsegmentlast);
			},
		},
	},
	mid arrow/.style={postaction={decorate,decoration={
				markings,
				mark=at position .5 with {\arrow[#1]{stealth}}
	}}},
}
\tikzset{%
	link/.style    = { white, double = black, line width = 1.8pt,
		double distance = 0.41pt },
	channel/.style = { white, double = black, line width = 0.8pt,
		double distance = 0.61pt },
}
\usepackage{multicol}

\begin{document}
		
	\vspace*{-1.5cm}

	\vspace{5mm}
	
	\begin{center}
		\textbf{\LARGE{Finite symmetries of quantum character stacks}}\\
		\vspace{1cm}
	{\large Corina Keller $^{a}$} \ \ and \ \ {\large Lukas Müller $^{b}$}

\vspace{5mm}

{\em $^a$ Institut Montpelliérain Alexander Grothendieck \\ 
 Université de Montpellier \\
     Place Eugène Bataillon, 34090 Montpellier}\\
    {\tt corina.keller@umontpellier.fr }
\\[7pt]
{\em $^b$ Max-Planck-Institut f\"ur Mathematik\\
	Vivatsgasse 7, D\,--\,53111 Bonn}\\
{\tt lmueller4@mpim-bonn.mpg.de\ }
		
		\vspace{5mm}

	\end{center}

	\begin{abstract}\noindent 
For a finite group $D$, we study categorical factorisation homology on oriented surfaces equipped with principal $D$-bundles, which `integrates' a (linear) balanced braided category $\mathcal{A}$ with $D$-action over those surfaces. For surfaces with at least one boundary component, we identify the value of factorisation homology with the category of modules over an explicit algebra in $\mathcal{A}$, extending the work of Ben-Zvi, Brochier and Jordan to surfaces with $\G$-bundles. Furthermore, we show that the value of factorisation homology on annuli, boundary conditions, and point defects can be described in terms of equivariant representation theory.  \par 
Our main example comes from an action of Dynkin diagram automorphisms on representation categories of quantum groups. We show that in this case factorisation homology gives rise to a quantisation of the moduli space of flat twisted bundles.  
\end{abstract}
	
	\tableofcontents

\section{Introduction}
In this paper we extend the work on categorical factorisation homology by Ben-Zvi, Brochier and Jordan~\cite{bzbj, bzbj2} to (framed) $\mathsf{E}_2$-algebras with an action of a finite group $\G$. This leads to functorial invariants for manifolds equipped with an 
$SO(2)\times \G$ tangential structure, or in more geometric terms oriented 2-dimensional manifolds equipped with 
principal $\G$-bundles. \par 
Factorisation homology~\cite{AF,Lurie} is a local-to-global invariant
which `integrates' higher algebraic quantities, namely disk algebras in a symmetric monoidal higher category $\mathcal{C}$, over manifolds. We will work with $\mathcal{C}=\Prc$, the 2-category of $k$-linear compactly generated presentable categories for $k$ an algebraically closed field of characteristic 0. In the $\G$-decorated setting, the coefficients $\mathcal{A}$ for factorisation homology are given by balanced braided monoidal categories equipped with an additional $\G$-action through balanced braided monoidal automorphisms. Factorisation homology then assigns to every oriented 2-dimensional manifold $\Sigma$
equipped with a principal $\G$-bundle, described by its classifying map $\varphi \colon \Sigma \longrightarrow B\G$, a linear category 
\begin{align}
\int\displaylimits_{(\Sigma,\varphi)} \mathcal{A} \in \Prc \ \ .
\end{align} 
This construction is functorial in the pair $(\Sigma,\varphi)$. \par 
Our main example will be $\mathcal{A}=\Rep_q(G)$, the (locally finite) representation category of a quantum group associated to a reductive group $G$ and $q \in \mathbb{C}^\times$ (we assume $q$ is not a root of unity), which admits a natural action of the group of outer automorphisms $\Out(G)$ of $G$. We use these coefficients to construct a functorial quantisation of the moduli space of flat twisted bundles related to finite $\Out(G)$-symmetries in gauge theories. Before addressing the role of symmetries, we give a brief overview on the factorisation homology approach to the quantisation of moduli spaces of flat bundles. \par 
For a reductive algebraic group $G$, the moduli space $\mathcal{M}(\Sigma)$ of flat principal $G$-bundles over a Riemann surfaces $\Sigma$ is ubiquitous in mathematical physics and symplectic geometry: For example, the symplectic volume of $\mathcal{M}(\Sigma)$ computes the 
topological limit of the partition function of two dimensional Yang-Mills theory on $\Sigma$~\cite{Witten:1991gt}, the state space of 3-dimensional Chern-Simons theory on $\Sigma$ can be constructed by applying geometric quantisation to $\mathcal{M}(\Sigma)$~\cite{Hitchin,geomquant}, and deformations of the category of quasi-coherent sheaves on $\mathcal{M}(\Sigma)$ describe boundary condition in the 4-dimensional Kapustin-Witten theory~\cite{KW,BZN2}. \par 
In~\cite{BZN,bzbj,bzbj2} it was shown that quasi-coherent sheaves on the classical moduli space can be understood in terms of the factorisation homology of $\Rep(G)$:
\begin{align}
\operatorname{QCoh}(\mathcal{M}(\Sigma)) \cong \int_{\Sigma} \Rep(G) \ \ .
\end{align}
The category $\Rep(G)$ admits a well-studied deformation by the category $\Rep_q(G)$ of (locally finite) $U_q(\mathfrak{g})$-modules. Thus, using the local-to-global property of factorisation homology, the quantum analog of the category of quasi-coherent sheaves on $\mathcal{M}(\Sigma)$ is defined in~\cite{bzbj, bzbj2} as the \emph{quantum character stack} $\int_{\Sigma} \Rep_q(G)$. This is a mathematical construction of the 2-dimensional part of the 4-dimensional Kapustin-Witten theory as a topological quantum field theory, which assigns to an oriented surfaces $\Sigma$ a quantisation of the moduli space of flat $G$-bundles on $\Sigma$, where `quantisation' is understood as a deformation of the category of quasi-coherent sheaves on the moduli space. \par
To explain the physical role of the $\G$-action, we turn our attention to symmetries of quantum field theories, in particular to symmetries of moduli spaces of $G$-local systems. One source for symmetries are automorphisms of the classical space of fields preserving the classical action functional. In gauge theories the space of fields is most naturally understood as a higher differential geometric object, namely a smooth stack, and automorphisms should take this higher geometric structure into account. Concretely, this means that the action of a symmetry group $\G$ only needs to close up to gauge 
transformations.   
In the physics literature these are known as \emph{fractionalised symmetries}~\cite{Wang:2017loc} and can be described by group extensions
\begin{align}
	1\longrightarrow G\longrightarrow \widehat{G} \longrightarrow D \longrightarrow 1 \ \ .
\end{align}
where the group $\widehat{G}$ encodes the non-trivial interaction of gauge transformations and the symmetry 
group $D$. We refer to \cite{BrauerGroup, MSGauge} for a detailed discussion of these symmetries in the case of 
discrete gauge theories. \par 
We will restrict our attention to extension of the form
\begin{align}
	1\longrightarrow G \longrightarrow G \rtimes \Out(G) \longrightarrow \Out(G) \longrightarrow 1
\end{align}
with $\G = \Out(G)$.\footnote{To handle arbitrary extensions one could use non-abelian 2-cocycles.} An element $\kappa \in \Out(G)$ acts on a gauge field described by a principal $G$-bundle with connection by forming the associated bundle along the
group homomorphism $\kappa \colon G\longrightarrow G$. In~\cite{2DYM} these symmetries have been 
studied in the context of 2-dimensional Yang-Mills theory. They restrict to an action of $\Out(G)$ on the moduli space $\mathcal{M}(\Sigma)$. 
One motivation for developing the general 
framework presented in this paper was to study these symmetries for quantum character stacks. \par 
On the level of the local coefficients, i.e.~for $\mathsf{fE}_2$-algebras, the symmetry is realised through the $\Out(G)$-action on $\Rep(G)$ by pullbacks. In Section~\ref{Ex: Aut-extension} we show that this action extends to $\Rep_q(G)$ and hence we can compute the value of factorisation homology for $\Rep_q(G)$ on oriented surfaces with principal $\Out(G)$-bundles. By evaluation on surfaces with trivial bundles we get an action of $\Out(G)$ on the quantum character stack associated to an arbitrary surface. This implements the action of the symmetry on the quantum character stack. \par 
Factorisation homology on surfaces equipped with non-trivial $\Out(G)$-bundles has also a natural field theoretical interpretation: The value of factorisation homology describes the coupling of the quantum character field theory to non-trivial $\Out(G)$-background fields. In~\cite{2DYM} the topological limit of the partition function of 2-dimensional Yang-Mills theory coupled to an $\Out(G)$-background field $\varphi \colon \Sigma \longrightarrow B \Out(G)$ 
was related to the symplectic volume of the moduli space of flat $\varphi$-twisted $G$-bundles $\mathcal{M}_\rho(\Sigma)$~\cite{WildChar, Meinrenken:Convexity,Zerouali}. We will show the analogous statement for quantum character stacks, i.e.~that they provided a quantisation of the category of quasi-coherent sheaves on $\mathcal{M}_\rho(\Sigma)$.  

\subparagraph{Summary of results and outline.}
In Section~\ref{Sec: Setup} we review factorisation homology following~\cite{AF}, with a focus on categorical factorisation homology on oriented 2-dimensional surfaces with $\G$-bundles for a finite group $\G$. We will also allow for certain stratifications along the lines of~\cite{AFT}, namely boundary conditions and point defects. The section concludes with some details related to the algebraic quantities appearing in this paper. 
In particular, we introduce the representation category of a quantum group $\Rep_q(G)$, and show that it is naturally endowed with an $\Out(G)$-action. \par 
After the setup is established, we compute in Section \ref{Sec:FHD-mfds} the factorisation homology with coefficients in a rigid braided tensor category $\mathcal{A}$ with $\G$-action $\vartheta$ of an oriented punctured surface $\Sigma$ equipped with a $\G$-bundle. To that end we apply reconstruction techniques for module categories, following ideas presented in \cite[Section 5]{bzbj}. We use a combinatorial description of the surface with decoration $\varphi \colon \Sigma \to B\G$, namely a decorated fat graph model $(P, \g_1, \dots, \g_n)$, see Definition~\ref{Def: Gluing pattern}. From $(P, \g_1, \dots, \g_n)$ we can define an algebra $a_P^{\g_1, \dots, \g_n} \coloneqq \bigotimes_{i = 1}^n \mathcal{F}_\mathcal{A}^{\g_i}$ in $\mathcal{A}$, where each $\mathcal{F}_{\mathcal{A}}^{\g_i}=\int^{V \in \text{comp}(\mathcal{A})} V^\vee \boxtimes \vartheta(\g_i^{-1}).V$ is a twisted version of Lyubashenko's coend \cite{LyubCoend} in $\mathcal{A}$. We show in Theorem \ref{Thm: value surface} that there is an equivalence of categories 
$$
\int\displaylimits_{(\Sigma, \varphi)} \mathcal{A} \cong a_P^{\g_1, \dots, \g_n}\text{-mod}_\mathcal{A} \ \ ,
$$ 
identifying factorisation homology with the category of modules over an algebra which can be described in purely combinatorial terms. This result is an extension of \cite[Theorem 5.14]{bzbj} to surfaces with $\G$-bundles.\par 
In Section \ref{Sec: LB} we explore the algebraic structure that arises on the collection of the factorisation homologies
$$
\int\displaylimits_{\varphi \colon \mathbb{S}^1\times \R \to \G} \mathcal{A}
$$
for varying decoration $\varphi$, which turn out to assemble into an algebra over the little bundles operad \cite{littlebundles}. It was shown in \cite{littlebundles} that categorical little bundles algebras can be identified with braided $\G$-crossed categories, as defined by Turaev \cite{turaevgcrossed,turaevhqft}. We
compute the resulting $\G$-crossed categories concretely in terms of bimodule traces introduced in \cite{fss}. \par 
The goal of Section \ref{Sec: Defects} is to give an explicit description of the algebraic data describing boundary conditions and point defects in $\G$-structured factorisation homology. It is well-known that for oriented 2-manifolds without $\G$-bundles, boundary conditions are incorporated by algebras over the Swiss-cheese operad, and point defects by $\mathsf{E}_2$-modules~\cite{AFT,ginot}. For algebras in linear categories, \cite[Theorem 3.11]{bzbj2} shows that the latter coincides with the notion of a braided module category as introduced in \cite{BMod3, BMod1, BMod2}. In order to extend these algebraic structures to the $\G$-decorated setting, we will work with combinatorial models for the decorated Swiss-cheese operad and the operad of decorated disks with marked points respectively. If we let $\mathcal{A}$ be a balanced braided tensor category with $\G$-action, we find:
\begin{itemize}
\item Boundary conditions are given by a monoidal category $\mathcal{C}$ with $\G$-action and a $\G$-equivariant braided functor $\mathcal{A} \longrightarrow \mathcal{Z}(\mathcal{C})$ into the Drinfeld centre of $\mathcal{C}$ (see Proposition \ref{prp:D-bdrycondition}).
\item Point defects are equivariant balanced right modules over $\mathcal{A}$ as given in Definition \ref{df:eqbalancedmodule} (see Proposition \ref{prp:D-mkdpoints}).
\end{itemize}
In Section \ref{Sec:ClsdMfd}, we treat the case of closed manifolds. \par 
Lastly, Section \ref{Sec: QCV} is devoted to our main application, the quantisation of the moduli space of twisted flat bundles via $\Out(G)$-structured factorisation homology with coefficients in $\Rep_q(G)$:
For a connected surface $\Sigma = \Sigma_{g,r}$ of genus $g$ and with $r > 0$ boundary components, together with a chosen point $p$ on the boundary, recall that the $G$-representation variety is the affine variety $\Hom(\pi_1(\Sigma), G)$ of group homomorphisms. Since the fundamental group of $\Sigma$ is free on $n = 2g + r -1 $ generators we have $\Hom(\pi_1(\Sigma), G) \cong G^{n}$. Via the holonomy map, the $G$-representation variety is identified with the moduli space $\mathcal{M}^\circ(\Sigma)$ of flat $G$-bundles on $\Sigma$ with a trivialisation over $p \in \partial \Sigma$, and there is an action of $G$ on $\mathcal{M}^\circ(\Sigma)$ changing the trivialisation. \par 
Now, given an $\Out(G)$-bundle $\rho \colon \pi_1(\Sigma_{g,r}) \to \Out(G)$ described by a tuple $(\kappa_1, \dots, \kappa_n)$ of elements $\kappa_i \in \Out(G)$, we can define the \emph{$\Out(G)$-twisted representation variety} $\mathcal{M}^\circ_\rho(\Sigma) = \Hom_\rho(\pi_1(\Sigma), G)$, where now the maps $\pi_1(\Sigma)\to G$ are no longer group homomorphisms, but twisted by the elements $\kappa_i$, see Section \ref{Sec:ModuliBunTwisted} for the formal definitions. The \emph{moduli space of flat $\rho$-twisted bundles} is the stacky quotient $\mathcal{M}^\circ_\rho(\Sigma) /^\rho G$, with respect to the $\rho$-twisted conjugation action, and we show that the category of quasi-coherent sheaves on this moduli space can be computed via $\Out(G)$-structured factorisation homology
$$
\int\limits_{\rho \colon \Sigma \to B\Out(G)} \Rep(G) \cong \bigotimes_{i = 1}^n \mathcal{O}^{\kappa_i}(G)\text{-mod}_{\Rep(G)} \ \ ,
$$
where on the right hand side $\otimes_{i = 1}^n \mathcal{O}^{\kappa_i}(G)$ is the algebra of functions on $\mathcal{M}^\circ_\rho(\Sigma) \cong G^n$ with the induced $\rho$-twisted action by $G$. \par 
We then follow the approach of Ben-Zvi, Brochier and Jordan \cite{bzbj} to quantise these moduli spaces by locally choosing coefficients in the representation category of the corresponding quantum group $\Rep_q(G)$ and subsequently gluing this local data together via factorisation homology over the surface $\Sigma$ decorated with an $\Out(G)$-bundle:
$$
\int\limits_{\rho \colon \Sigma \to B\Out(G)} \Rep_q(G) \cong a_P^{\kappa_1,\dots,\kappa_n}\text{-mod}_{\Rep_q(G)} \ \ .
$$ 
We then show by means of a direct computation that the above provides a quantisation of the moduli space of flat twisted bundles. To that end, we present in Proposition \ref{prp:twistedFR} a novel combinatorial formula for the Poisson structure on $\mathcal{M}_\rho^\circ(\Sigma)$ and in Theorem \ref{Thm:QuantisationTwistedFR} we prove that the algebra\footnote{The algebra $a_\hbar^{\kappa_1,\dots,\kappa_n}$ is the combinatorial algebra $a_P^{\kappa_1,\dots,\kappa_n}$ in $\mathcal{A} = \Rep_\hbar(G)$.} $a_\hbar^{\kappa_1,\dots,\kappa_n}$ is a deformation quantisation of the algebra of functions on $\mathcal{M}_\rho^\circ(\Sigma)$. 

\subparagraph{Relation to topological field theories.}
We conclude the introduction by briefly commenting on the relation to topological field
theories. We restrict our discussion to \emph{framed} field theories since we want to highlight the additional structure coming from the 
bundle decorations and because this is the case most studied in the literature on fully extended field theories. In the undecorated setting, i.e.~for manifolds without $\G$-bundles, factorisation homology gives 
rise to fully extended topological field theories.
More precisely, for an $\mathsf{E_n}$-algebra $\mathcal{E}$ in a (nice) symmetric monoidal $(\infty,1)$-category $\mathcal{C}$, Scheimbauer~\cite{claudiathesis} explicitly constructed a fully extended framed topological field theory taking values in the higher Morita category of 
$\mathsf{E_n}$-algebras~\cite{claudiathesis,Hau17, JFS} in $\mathcal{C}$ via factorisation homology for framed manifolds, assigning $\mathcal{E}$ to the framed point. For $n=2$ and $\mathcal{C}=\Prc$ the Morita 
category is the 4-category $\catf{BrTens}$ of braided tensor categories with central algebras\footnote{For $\mathcal{A},\mathcal{B}\in \catf{BrTens}$ a \emph{central algebra} is an $\mathsf{E}_1$-algebra in $\mathcal{A}$-$\mathcal{B}$-bimodules.} as 
1-morphisms, central bimodules as 2-morphisms, and functors and natural transformations as 3- and 4-morphisms, respectively~\cite{dualizabilityBrTens}. Every object of $\catf{BrTens}$ is 2-dualisable~\cite{GS,dualizabilityBrTens} and 
hence by the cobordism hypothesis~\cite{BD,CH} defines a 2-dimensional framed topological field theory, namely the one explicitly constructed by Scheimbauer. \par 
If one adds decorations with principal $\G$-bundles, the corresponding topological field theories are known as $\G$-equivariant field theories~\cite{turaevhqft}. Factorisation homology for 
$\mathsf{E}_n$-algebras with $\G$-action is expected to provide examples of $\G$-equivariant 
field theories with values in the Morita category of $\mathsf{E}_n$-algebras. Our work can be understood as exploring this (expected) equivariant field theory in the oriented setting and dimension $n=2$ with values in $\Prc$. As a complementary example 
it was shown in~\cite{MuellerWoikeHH} that equivariant higher Hochschild homology, that is 
factorisation homology for $\mathsf{E_\infty}$-algebras with $\G$-action in chain complexes, gives 
examples of equivariant field theories in any dimension $n$. \par 
$\G$-equivariant field theories can also be studied through the cobordism hypothesis, which implies 
that 2-dimensional framed fully extended $\G$-equivariant field theories with values in $\catf{BrTens}$
are described by functors $B\G \longrightarrow \catf{BrTens}$. Such a functor is described by picking out an object $\mathcal{A} \in 
\catf{BrTens}$, together with a central algebra $\mathcal{M}_\g$ for every $\g \in \G $, a central $\mathcal{M}_{\g_2} \circ \mathcal{M}_{\g_1}$-$\mathcal{M}_{\g_2\g_1} $-bimodule for every pair
$\g_1,\g_2 \in \G$ and furthermore 3- and 4-morphisms for all triples and quadruples of group elements, respectively, satisfying a coherence condition involving five group elements. This data 
can be constructed from an $\mathsf{E}_2$-algebra in $\Prc$ with $\G$-action by setting $\mathcal{M}_\g = 
\mathcal{A}$, seen as an $\mathsf{E}_1$-algebra in bimodules over itself, where the left action is
twisted by acting with $\g$. The coherence isomorphisms for the $\G$-action induce the additional
data. However, this is only a special case for the data classifying equivariant framed field 
theories according to the cobordism hypothesis and situations outside this class do not seem to be 
accessible using factorisation homology with values in $\Prc$.\par 
The type of factorisation homology we compute in this article is a special case of equivariant factorisation homology for global quotient orbifolds~\cite{EFH}; namely the case of free actions. The general case, which requires additional input data, should give rise to field 
theories defined as functors out of the bordism category introduced in~\cite{EBord}. Hence, our results provide a first steps towards computing this field theory.   

\subparagraph{Acknowledgements.}
We thank Bart Vlaar for helpful discussions on the action of Dynkin diagram automorphisms on $\Rep_q(G)$. We thank  Adrien Brochier, Damien Calaque, David Jordan, Christoph Schweigert, and Lukas Woike for helpful discussions and correspondence. CK has received funding from the European Research Council (ERC) under the European Union's Horizon 2020 research and innovation programme (grant agreement No.~768679). LM gratefully acknowledges support from the Max-Planck-Institute for Mathematics in Bonn.

\section{Setup}\label{Sec: Setup}
In this section we review some of the necessary mathematical background and introduce the main example of 
$\Out(G)$-actions on the representation category of a quantum group $\Rep_q(G)$ leading to a coherent 
quantisation of moduli spaces of twisted flat bundles in Section~\ref{Sec: QCV}. 

\subsection{Review of factorisation homology for manifolds with $\mathcal{G}$-structures}\label{Sec:ReviewFH}

Let $\Man_n$ be the topological category of $n$-dimensional manifolds which admit a finite good open cover 
with embeddings as morphisms. The morphism spaces are equipped with the compact-open topology. The disjoint union of
manifolds equips $\Man_n$ with the structure of a symmetric monoidal 
category. Let $\mathcal{G}$ be a topological group
and $\rho \colon \mathcal{G}\longrightarrow GL(n)$ a continuous group homomorphism.  
A \emph{$\mathcal{G}$-structure} on a manifold $M$ is a homotopy lift
\begin{equation}
\begin{tikzcd}
 & B\mathcal{G} \ar[d,"B\rho"] \\
 M \ar[r]\ar[ru] & BGL(n)
\end{tikzcd}
\end{equation} 
of the classifying map for the frame bundle. 
These homotopy lifts correspond to a reduction of the structure group of the frame bundle to $\mathcal{G}$. 
There is a space of tangential $\mathcal{G}$-structures on
$M$, given by the mapping space $\Spaces_{/BGL(n)}(M, B \mathcal{G})$. This space is a model for the $\infty$-groupoid of tangential $\mathcal{G}$-structures on $M$. Homotopies in this space lead to a natural notion of morphisms of tangential structures. 

\begin{example}\label{Ex: G-structure}
We list some important examples of $\mathcal{G}$-structures.
\begin{itemize}
	\item For $\mathcal{G}=\star$, a $\mathcal{G}$-structure is the same as the choice of a framing on $M$.
	\item For $\mathcal{G}=SO(n) \longrightarrow GL(n)$ the canonical embedding, a $\mathcal{G}$-structure is the same as the choice of an orientation.
	\item For $\mathcal{G}=SO(n) \times \G$ and $\rho \colon SO(n) \times \G \xrightarrow{\pr_{SO(n)}} SO(n) \longrightarrow GL(n)$, a $\mathcal{G}$-structure is the choice of an orientation on $M$ together with a map $M\longrightarrow B\G$, i.e.~a principal $\G$-bundle. This is the example considered in this paper.
\end{itemize}  
\end{example}

To construct the $\infty$-category $\Man_n^\mathcal{G}$ of manifolds with $\mathcal{G}$-structure we proceed as follows: there is a symmetric monoidal functor $\tau \colon \Man_n \longrightarrow \Spaces_{/BGL(n)}$ of $\infty$-categories sending a manifold $M$ to the classifying map $M\longrightarrow BGL(n)$ of its frame bundle~\cite[Section
2.1]{AF}. The category $\Man_n^\mathcal{G}$ of manifolds with tangential $\mathcal{G}$-structure is defined as the pullback
\begin{equation}
\begin{tikzcd}
\Man_n^\mathcal{G} \ar[r] \ar[d] & \Spaces_{/B\mathcal{G}} \ar[d] \\
\Man_n \ar[r,"\tau"] & \Spaces_{/BGL(n)}
\end{tikzcd}
\end{equation}
Denote by $\Disk_n^\mathcal{G} \subset \Man_n^\mathcal{G}$ the full symmetric monoidal subcategory whose objects are disjoint unions of Euclidean spaces. Let $\Va$ be a symmetric monoidal $\infty$-category. A \emph{$\Disk_n^\mathcal{G}$-algebra} in $\Va$ is a symmetric monoidal functor $\mathcal{A} \colon \Disk_n^\mathcal{G} \longrightarrow \Va$. 

\begin{remark}
For the tangential structures of Example~\ref{Ex: G-structure}, disk algebras have a description in terms of more classical objects:
\begin{itemize}
\item A $\Disk^\star_n$-algebra is an $\mathsf{E_n}$-algebra, see for example~\cite{AF}.
\item A $\Disk^{SO(n)}_n$-algebra is a framed $\mathsf{E_n}$-algebra, see for example~\cite{AF}. 
\item A $\Disk^{SO(n)\times \G}_n$-algebra is a framed $\mathsf{E_n}$-algebra 
equipped with a $\G$-action, see for example~\cite[Proposition 4.6]{EFH}.
\end{itemize} 
\end{remark}

\begin{example}
In Figure \ref{fig:diskalgebras} we give a sketch for $n = 2$ of the disk operations in $\Disk^{\mathcal{G}}_2$, for the tangential structures of the previous remark, and the corresponding algebraic structures on $\mathcal{A} \colon \Disk^{\mathcal{G}}_2 \to (\mathcal{V},\otimes)$.
\end{example}

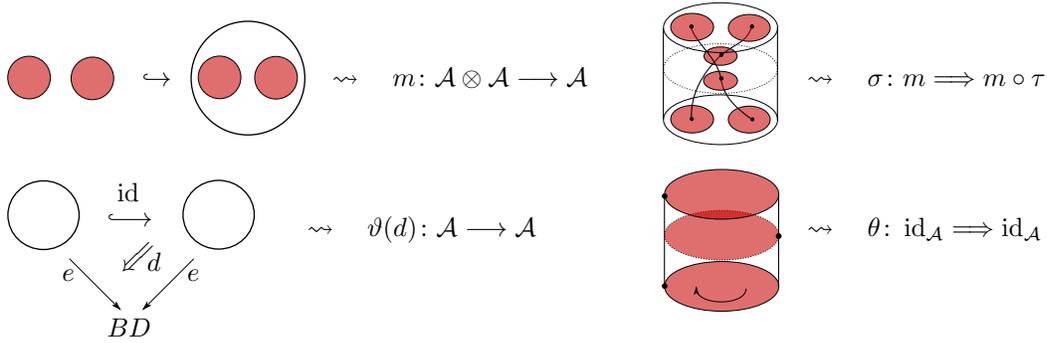
\begin{figure}[h!]
\vspace{0.25cm}
	\begin{overpic}[scale=0.5
		,tics=10]
		{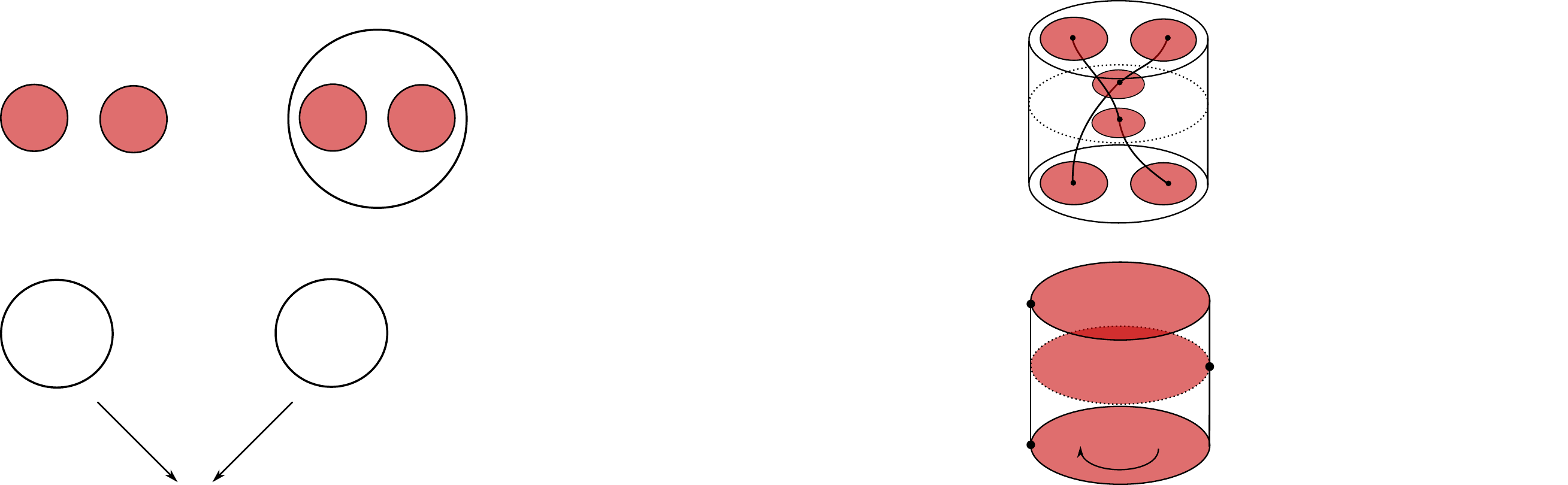}
		\put(13.5,22.5){$\hookrightarrow$}
		\put(32.5,22.5){$\rightsquigarrow \quad m \colon \mathcal{A} \otimes \mathcal{A} \to \mathcal{A}$}
		\put(5.5,3){$e$}
		\put(10,-2.5){$B\G$}
		\put(18,3){$e$}
		\put(10,8.5){$\hooklongrightarrow$}
		\put(11,11){$\id$}
		\put(11,5){$\SWarrow$}
		\put(14,4){$\g$}
		\put(30,7.5){$\rightsquigarrow \quad \vartheta(\g) \colon \mathcal{A} \to \mathcal{A}$}
		\put(80,22.5){$\rightsquigarrow \quad \sigma \colon m \Longrightarrow m\circ \tau$}
		\put(80,7.5){$\rightsquigarrow \quad \theta \colon \id_{\mathcal{A}} \Longrightarrow \id_{\mathcal{A}}$}
	\end{overpic}  
	\vspace{0.25cm}
	\caption{\emph{First row}: Disk embeddings (or isotopies thereof) in $\Disk^{\ast}_2$ that give rise to the multiplication $m$ and the braiding $\sigma$ in the $\mathsf{E_2}$-algebra $\mathcal{A}$. Here $\tau \colon A \otimes A \to A \otimes A$ denotes the braiding in $\mathcal{V}$. \emph{Second row}: On the right, the additional operation in the oriented case given by a loop in the space of disk embeddings in $\Disk^{SO(2)}_2$, rotating the disk by $2\pi$. Together with the operations in the first row, this endows $\mathcal{A}$ with the structure of a framed $\mathsf{E_2}$-algebra. On the left, the additional operation in the $D$-decorated oriented case, given by the identity disk embedding in $\Disk^{SO(2) \times \G}_2$ with homotopy $d \colon \id^*(e) \Rightarrow e$, inducing an automorphism of $\mathcal{A}$ for each $\g \in \G$, i.e.~a $\G$-action on $\mathcal{A}$.}
\label{fig:diskalgebras}
\end{figure}

Let $(\mathcal{V},\otimes)$ be a symmetric monoidal $\infty$-category. We assume that $\mathcal{V}$ admits sifted colimits and that $\otimes$ preserves them in each component. \emph{Factorisation homology} $\int_\bullet \Aa$ with coefficients in the $\Disk_n^\mathcal{G}$-algebra $\mathcal{A}$ is the left Kan-extension~\cite{AF}:
\begin{equation}
\begin{tikzcd}
\Disk_n^\mathcal{G} \ar[r,"\Aa"]  \ar[d] & \Va \\
\Man_n^\mathcal{G} \ar[ru,"\int_\bullet \Aa",swap] & 
\end{tikzcd}
\end{equation}
The condition that $\otimes$ preserves sifted colimits makes factorisation homology into a symmetric monoidal functor. Hence, the value of factorisation homology on any manifold $M$ is naturally pointed by the inclusion $\emptyset \hookrightarrow M$ of the empty manifold:
\begin{align}
\int_\emptyset \mathcal{A} \cong 1_\mathcal{V} \longrightarrow \int_M \mathcal{A} \ \ . 
\end{align} 

\subsubsection{Excision}\label{Sec:Excision}

The main tool for computing factorisation homology will be $\otimes$-excision. Excision allows one to reconstruct 
the value of factorisation homology from a certain decomposition of $M$, namely from a {collar-gluing}~\cite[Section 3.3]{AF}. We recall that a \emph{collar-gluing} of a $\mathcal{G}$-structured manifold $M$ is given by a smooth map 
$$
f\colon M \longrightarrow [-1,1] \ \ ,
$$ 
such that $f^{-1}(-1,1) \longrightarrow (-1,1)$ is a manifold bundle. If we define $M_-\coloneqq f^{-1}[-1,1)$, $M_+\coloneqq f^{-1}(-1,1]$ and $M_0\coloneqq f^{-1}(-1,1)$, we will often denote the collar-gluing by $M = M_-\bigcup_{M_0} M_+$. 

\begin{figure}[t]
\begin{center}
\vspace{0.25cm}
	\begin{overpic}[scale=0.3
		,tics=10]
		{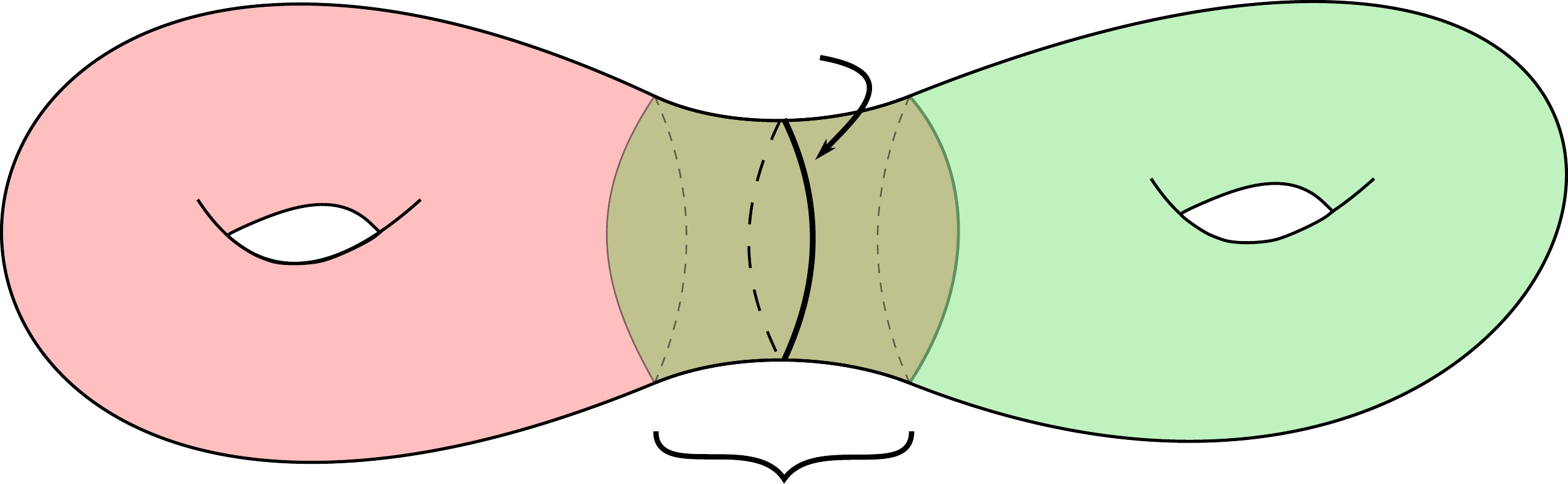}
		\put(47.5,-5){$M_0$}
		\put(47.5,27.5){$N$}
		\put(7.5,22.5){$M_-$}
		\put(90,22.5){$M_+$}
	\end{overpic}  
	\vspace{0.5cm}
	\caption{An example of a collar-gluing.}
\label{fig:collar-gluing}
\end{center}
\end{figure}
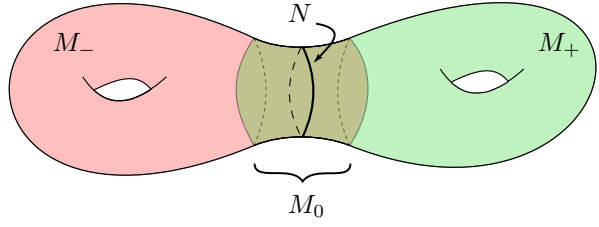

We can choose an equivalence $\theta \colon M_0 \xrightarrow{\ \cong \ } N\times (-1,1)$ in the $\infty$-category of $\mathcal{G}$-structured manifolds, where $N$ is the fibre over an arbitrary point in $(-1,1)$, as illustrated in Figure \ref{fig:collar-gluing}. The object $\int_{N\times (-1,1)} \mathcal{A}$ has a natural 
$E_1$-algebra structure in $\Va$, which gives rise to an $E_1$-algebra structure on $\int_{M_0} \mathcal{A}$. We fix oriented embeddings 
\begin{align}
\mu_+ \colon (-1,1) \sqcup (-1,1] \longrightarrow (-1,1] \  \text{ and } \  \mu_- \colon [-1,1) \sqcup (-1,1) \longrightarrow [-1,1) \ \ ,
\end{align}      
which are the identity in a neighbourhood of the boundary. Using the equivalence $\theta$, we lift these embeddings to maps $\act_- \colon M_- \sqcup M_0 \longrightarrow M_- $ and $\act_+ \colon M_0 \sqcup M_+ \longrightarrow M_+$ of $\mathcal{G}$-structured manifolds, see Figure~\ref{fig:modulestructure} below for a sketch.

\begin{figure}[b]
\begin{center}
\vspace{0.25cm}
	\begin{overpic}[scale=0.3
		,tics=10]
		{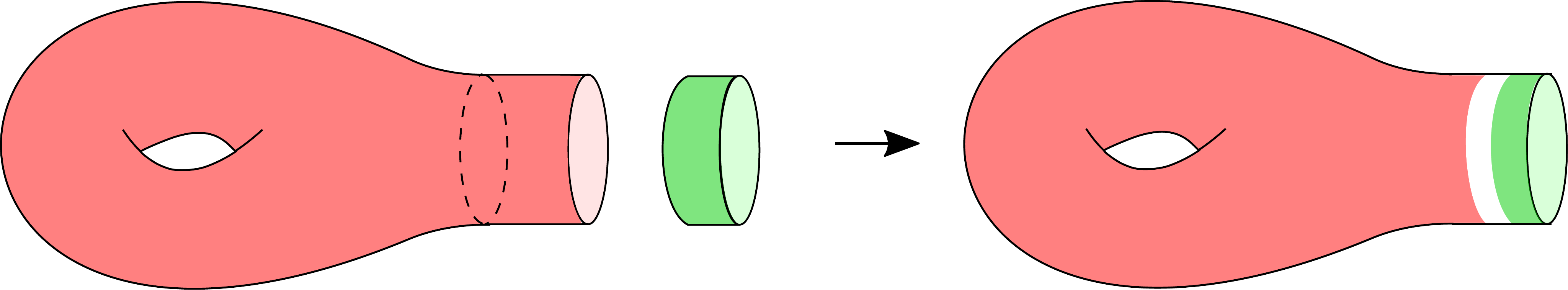}
	\end{overpic}  
	\vspace{0.5cm}
	\caption{The map which induces the right $\int_{M_0}\mathcal{A}$-module structure on $\int_{M_-} \mathcal{A}$. Here, the green collar depicts the manifold $N \times (-1,1)$.}
\label{fig:modulestructure}
\end{center}
\end{figure}

Evaluation of factorisation homology on $\act_-$ and $\act_+$ equips $\int_{M_-} \mathcal{A}$ and $\int_{M_+}\mathcal{A}$ with the structure of a right and left module over $\int_{M_0}\mathcal{A}$, respectively. At this point we want to highlight that the module structures depend on the chosen trivialisation $\theta$; see Section \ref{sec: excision D-mfd} for an example. The value of factorisation homology on $M$ can be computed as the relative tensor product~\cite[Lemma 3.18]{AF}
\begin{align}\label{Eq: Excision}
\int_M \mathcal{A} \cong  \int_{M_-} \mathcal{A}~\underset{{\int_{M_0}\mathcal{A}}}{\bigotimes}~\int_{M_+} \mathcal{A} \ \ ,
\end{align} 
defined through the bar construction in $\Va$.

\subsubsection{Point defects and boundary conditions}\label{Sec: PD and BC}

Factorisation homology admits a natural extension to stratified manifolds~\cite{AFT}, which in more physics oriented language 
corresponds to incorporating defects in the field theory that we wish to study via factorisation homology. For us, only two types of defects will be relevant; namely point defects and boundary 
conditions. Instead of going through the heavy machinery of stratified manifolds, we only mention the concrete examples studied in 
this paper following~\cite{bzbj2}. \par 
We fix $\mathcal{G} = SO(2) \times \G$ and define the $\infty$-category $\Man_{2,\ast}^{\mathcal{G}}$ whose objects are oriented 2-dimensional manifolds $\Sigma$, together with a collection of marked points $p_1,\dots ,p_n\in \Sigma$ and a continuous map $\varphi \colon \Sigma \setminus \{ p_1,\dots ,p_n \}\longrightarrow B\G$. Morphisms are embeddings of manifolds, mapping marked 
points bijectively onto marked points, which are compatible with the morphisms into $B\G$. We denote by $\Disk_{2, \ast}^{\mathcal{G}}$ the full subcategory whose objects are disjoint unions of disks with one or zero marked points. Notice that we do not require the $\G$-bundles to extend to the whole of $\Sigma$. As for smooth manifolds, factorisation homology can again be defined by left Kan extension:\footnote{The slice categories appearing in the coend formula for the left Kan extension are not sifted. Hence, here we need to assume that $\Va$ is tensor cocomplete.} 
\begin{equation}
	\begin{tikzcd}
		\Disk_{2, \ast}^{\mathcal{G}} \ar[r,"\mathcal{F}"]  \ar[d] & \Va \\
		\Man_{2,\ast}^{\mathcal{G}} \ar[ru,"\int_\bullet \mathcal{F}",swap] & 
	\end{tikzcd}
\end{equation}   
      
The second type of defects we want to study are boundary conditions. To that end, we define the category $\Man_{2,\partial}^{\Ga}$ of oriented 2-dimensional manifolds $\Sigma$ with boundary $\partial \Sigma$ and continuous maps $\Sigma \longrightarrow B\G$. We denote by $\Disk_{2,\partial}^{\Ga}$ the full subcategory with objects disjoint unions of disks and half disks, by the latter we mean manifolds
diffeomorphic to $\R\times \R_{\geq 0}$. \par 
We will adopt the following terminology: 
\begin{definition}
By \emph{point defects} in $\mathcal{G} = SO(2) \times \G$-structured factorisation homology we mean a symmetric monoidal functor $\mathcal{F} \colon \Disk_{2, \ast}^{\mathcal{G}} \to \mathcal{V}$. Similarly, by a \emph{boundary condition} we mean a symmetric monoidal functor $\mathcal{F} \colon \Disk_{2, \partial}^{\mathcal{G}} \to \mathcal{V}$.
\end{definition}
In Section~\ref{Sec: Defects} we will give an algebraic characterisation of point defects and boundary conditions.

\begin{remark}
Unless otherwise stated, we will usually work with trivial boundary conditions in this paper, meaning that we use the same disk algebra for a disk with empty boundary, as for a disk with non-empty boundary.
\end{remark}

\subsection{The categorical case}\label{Sec: Categorical}

From now on we specialise to 2-dimensional manifolds and tangential structures of the form $\G \times SO(2)$, where $\G$ is a finite group. Throughout this paper we will work with factorisation homology with values in the $(2,1)$-category $\Prc$ of $k$-linear compactly generated presentable categories with compact and cocontinuous functors and natural isomorphisms between them, meaning that we will not use any non-invertible 2-morphisms. For us $k$ will always be an algebraically closed field of characteristic 0, usually $k = \mathbb{C}$. Recall that an object $c$ in a $k$-linear category $\Ca$ is \emph{compact} if the functor $\Hom(c,-)$ preserves filtered colimits. A category $\Ca$ is \emph{compactly generated} if every object can be written as a filtered colimit of compact objects and a functor is \emph{compact} if it preserves compact objects. We refer the reader to~\cite[Section 3]{bzbj} for more details on $\Prc$. \par 
Every $\infty$-functor from $\Man_2^{\G \times SO(2)}$ to $\Prc$ will factor through its homotopy 2-category  
which admits the following concrete description.
\begin{definition}
We denote by $\G\text{-}\catf{Man}_2$ the $(2,1)$-category with
\begin{itemize}
\item \emph{Objects}: Oriented 2-dimensional manifolds $\Sigma$ equipped
with a continuous map $\varphi\colon \Sigma \longrightarrow B\G$.
\item \emph{1-Morphisms}: Smooth embeddings $f\colon \Sigma_1 \longrightarrow \Sigma_2$ together with the choice of a homotopy $h \colon \varphi_1 \longrightarrow f^*\varphi_2$.
\item \emph{2-Morphisms}: A 2-morphism $(f_1,h_1)\longrightarrow (f_2,h_2)$ is given by an equivalence class of isotopies $\chi \colon f_1 \longrightarrow f_2$, together with a map $\gamma \colon \Sigma_1 \times \Delta^2 \longrightarrow B\G $ filling
\begin{equation}
\begin{tikzcd}
 & \ \  f_2^* \varphi_2 & \\
\varphi_1 \ar[ru, "h_2" ] \ar[rr, "h_1",swap] &  & f_1^* \varphi_2 \ar[lu, "\chi^*\varphi_2" ,swap]
\end{tikzcd}
\end{equation}
Two such pairs $(\chi,\gamma)$ and $(\chi',\gamma')$ are equivalent if 
there exists an isotopy of isotopies from $\chi$ to $\chi'$ (i.e. a map $\Omega \colon \Sigma_1 \times \Delta^2 \longrightarrow \Sigma_2$ filling the bottom in Diagram~\eqref{Eq: simplex}) and a map $\Gamma \colon \Sigma \times \Delta^3 \longrightarrow B\G $ filling
\begin{equation}\label{Eq: simplex}
\begin{tikzcd}
 & & & \varphi_1 & & \\ 
 & & &  & &  \\
 & & & & & \ar[lluu, leftarrow, "h_2" description] f_2^*\varphi_2  \\ 
f_1^*\varphi_2 \ar[rrruuu,"h_1" description, leftarrow] \ar[rrd,"\chi^*\varphi_2" description] \ar[rrrrru,"{\chi'}^*\varphi_2" description, near start] & & & & & \\
 & & f_2^*\varphi_2 \ar[uuuur, crossing over, leftarrow, "h_2" description, near end] \ar[rrruu]& & & \\
\end{tikzcd}
\end{equation}
where the faces are labeled with the various maps which are part of the morphisms.  
\end{itemize}
We denote the corresponding disk category by $\G\text{-}\Disk_2$.
\end{definition} 

\begin{remark}
Similarly, there are truncated versions $\G\text{-}\catf{Man}_{2,\ast}$ and $\G\text{-}\catf{Man}_{2,\partial}$ of the categories $\Man_{2,\ast}^{\G\times SO(2)}$
and $\Man_{2,\partial}^{\G \times SO(2)}$ introduced above. 
\end{remark}

One reason to work with $\Prc$ is that it is a closed symmetric monoidal $(2,1)$-category under the 
Deligne-Kelly tensor product $\boxtimes$. In particular, the tensor product $\boxtimes$ preserves sifted colimits in each variable, see~\cite[Proposition 3.5]{bzbj}. For any two objects $\mathcal{C}, \mathcal{D} \in \Prc$, the 
Deligne-Kelly tensor product $\mathcal{C} \boxtimes \mathcal{D} \in \Prc$ of $\mathcal{C}$ and $\mathcal{D}$ 
is characterised via the natural equivalence
$$
\Prc[\mathcal{C} \boxtimes \mathcal{D},\mathcal{E}] \cong \catf{Bil}_c[\mathcal{C},\mathcal{D};\mathcal{E}]  \ \ ,
$$
where $\catf{Bil}_c[\mathcal{C},\mathcal{D};\mathcal{E}]$ is the category of $k$-bilinear functors from $\mathcal{C} \times \mathcal{D}$ to $\mathcal{E}$, preserving colimits  in each variable separately. 

\begin{definition}
A tensor category $\mathcal{A}$ in $\Prc$ is \emph{rigid} if all compact objects of $\mathcal{A}$ are left and right dualisable.
\end{definition}

\begin{definition}
A \emph{balancing} is a family of natural isomorphisms $(\theta_V: V \to V)_{V \in \mathcal{A}}$, such that $\theta_{1_\mathcal{A}} = \id_{1_\mathcal{A}}$, and so that it is compatible with the braiding $\sigma$ of $\mathcal{A}$:
$\theta_{V \otimes W} = \sigma_{W,V} \circ \theta_W \otimes \theta_V \circ \sigma_{V,W}: V \otimes W \to V \otimes W$, graphically we depict this compatibility as follows:
\begin{equation}\label{Eq: balancing}
	\begin{tikzpicture}
	\begin{pgfonlayer}{nodelayer}
		\node [style=none] (22) at (-12, 1) {};
		\node [style=none] (23) at (-12, 2) {};
		\node [style=none] (24) at (-13, 1) {};
		\node [style=none] (25) at (-13, 2) {};
		\node [style=none] (26) at (-12, 3) {};
		\node [style=none] (27) at (-12, 4) {};
		\node [style=none] (28) at (-13, 3) {};
		\node [style=none] (29) at (-13, 4) {};
		\node [style=none] (33) at (-11, 2.5) {=};
		\node [style=rec m] (47) at (-12.5, 2.5) {};
		\node [style=rec m] (48) at (-12.5, 2.5) {$\theta$};
		\node [style=rec s] (49) at (-9.5, 2.5) {$\theta$};
		\node [style=rec s] (50) at (-8.5, 2.5) {$\theta$};
		\node [style=none] (51) at (-13, 0) {};
		\node [style=none] (52) at (-12.5, 0) {$V \otimes W$};
		\node [style=none] (53) at (-9, 0) {$V \otimes W$};
		\node [style=none] (54) at (-9.5, 1) {};
		\node [style=none] (55) at (-8.5, 1) {};
		\node [style=none] (56) at (-9.5, 2) {};
		\node [style=none] (57) at (-8.5, 2) {};
		\node [style=none] (58) at (-9.5, 3) {};
		\node [style=none] (59) at (-8.5, 3) {};
		\node [style=none] (60) at (-9.5, 4) {};
		\node [style=none] (61) at (-8.5, 4) {};
	\end{pgfonlayer}
	\begin{pgfonlayer}{edgelayer}
		\draw (22.center) to (23.center);
		\draw (24.center) to (25.center);
		\draw (26.center) to (27.center);
		\draw (28.center) to (29.center);
		\draw [in=-90, out=90] (54.center) to (57.center);
		\draw [style=over, in=-90, out=90] (55.center) to (56.center);
		\draw [in=-90, out=90] (58.center) to (61.center);
		\draw [style=over, in=-90, out=90] (59.center) to (60.center);
	\end{pgfonlayer}
\end{tikzpicture}

\end{equation}
A \emph{balanced tensor category} is then a braided tensor category equipped with a balancing.
\end{definition}

By a result of Salvatore and Wahl~\cite{salvatorewahl}, the 2-category of framed $\mathsf{E}_2$-algebras (or equivalently $\Disk_2^{SO(2)}$-algebras) in $\Prc$ can be identified with the 2-category of balanced braided tensor categories $\bBr$.  We also recall that a ribbon category in $\Prc$ is a rigid balanced braided tensor category so that the balancing maps satisfy $\theta_{V^\vee} = (\theta_V)^\vee$. One can show that in this case giving a balancing is equivalent to giving a pivotal structure, see e.g.~\cite[Appendix A.2]{Ctrace}. Finally, a $\G\text{-}\Disk_2$-algebra is described by a balanced braided tensor category with $\G$-action. 

\begin{definition}
Let $\mathcal{A}$ be a balanced tensor category. A \emph{$\G$-action on 
$\mathcal{A}$} is a (2-)functor $\vartheta\colon \star \DS \G \longrightarrow \star\DS \Aut_{\bBr}(\mathcal{A})$ from the category with one object and $\G$ as automorphisms to the 2-category with one object, balanced braided automorphisms\footnote{A braided automorphism is \emph{balanced} if it preserves $\theta$.} of
$\mathcal{A}$ as 1-morphisms and natural transformations as 2-morphisms. In more details, the action consists of an auto-equivalence $\vartheta(\g)\colon \mathcal{A} \to \mathcal{A}$ for each $\g \in \G$, and for each composable pair $\g_i,\g_j \in \G$ we have a natural isomorphism $c_{ij}\colon \vartheta(\g_i\g_j) \xrightarrow{\cong} \vartheta(\g_1)\vartheta(\g_2)$ satisfying the usual associativity axiom. 
\end{definition}

Our main example will be constructed from Dynkin diagram automorphisms acting on the representation categories of quantum groups, see Section~\ref{Ex: Aut-extension}. 

\subsubsection{Excision for manifolds with $\G$-bundles}\label{sec: excision D-mfd} 

Consider an object $(\Sigma,\varphi)$ in $\G\text{-}\catf{Man}_2$, where $\Sigma$ is an oriented 2-manifold and $\varphi \colon \Sigma \to B\G$ a continuous map. Let $\Sigma = \Sigma_- \cup_{\Sigma_0} \Sigma_+$ be a collar-gluing and $\theta \colon \Sigma_0 \cong N \times (-1,1)$ a diffeomorphism of oriented manifolds. Notice that when using excision to compute factorisation homology on $(\Sigma,\varphi)$, the restriction $\varphi|_{N \times (-1,1)}$ is \emph{not} required to be constant along the interval $(-1,1)$, though it will be homotopic to the constant map. For the cases of interest to us, making this homotopy compatible with the collar-gluing will introduce a $\G$-twist in the action featuring in excision. We illustrate this last point with an example which will be relevant later on:

\begin{example}\label{Ex:Excision D-mfd}
Assume that the map $\varphi$ is such that its restriction $\varphi|_{\Sigma_- \setminus \Sigma_0}$ as well as $\varphi|_{\Sigma_+ \setminus \Sigma_0}$ agree with the constant map to the base point $\ast$ of $B\G$. Furthermore, let us fix a diffeomorphism $\theta \colon \Sigma_0 \xrightarrow{\cong}  N \times (-1,1)$ of oriented manifolds. Here, $N$ is the codimension 1 submanifold determined by the given collar-gluing, see Figure \ref{fig:collar-gluing}. We choose $\varphi$ such that its pullback to $N \times (-1,1)$ is given by
\begin{align*}
(\theta^{-1})^*\varphi(n,s) = 
\begin{cases}
\ast, & \text{for}~s \notin (-\tfrac{1}{2},\tfrac{1}{2}) \\
\gamma_{\g^{-1}}(s + \frac{1}{2}), & \text{for}~s \in (-\tfrac{1}{2},\tfrac{1}{2})
\end{cases}
\end{align*}
for all $n \in N$, as illustrated in Figure \ref{fig:excisionDmfd}. Here, $\gamma_{\g^{-1}}\colon [0,1] \longrightarrow B\G$ is the loop corresponding to the inverse of a given group element $\g \in \G$.
\begin{figure}[b]
\vspace{0.25cm}
\begin{subfigure}[t]{0.5\textwidth}
\centering
	\begin{overpic}[scale=0.4
		,tics=10]
		{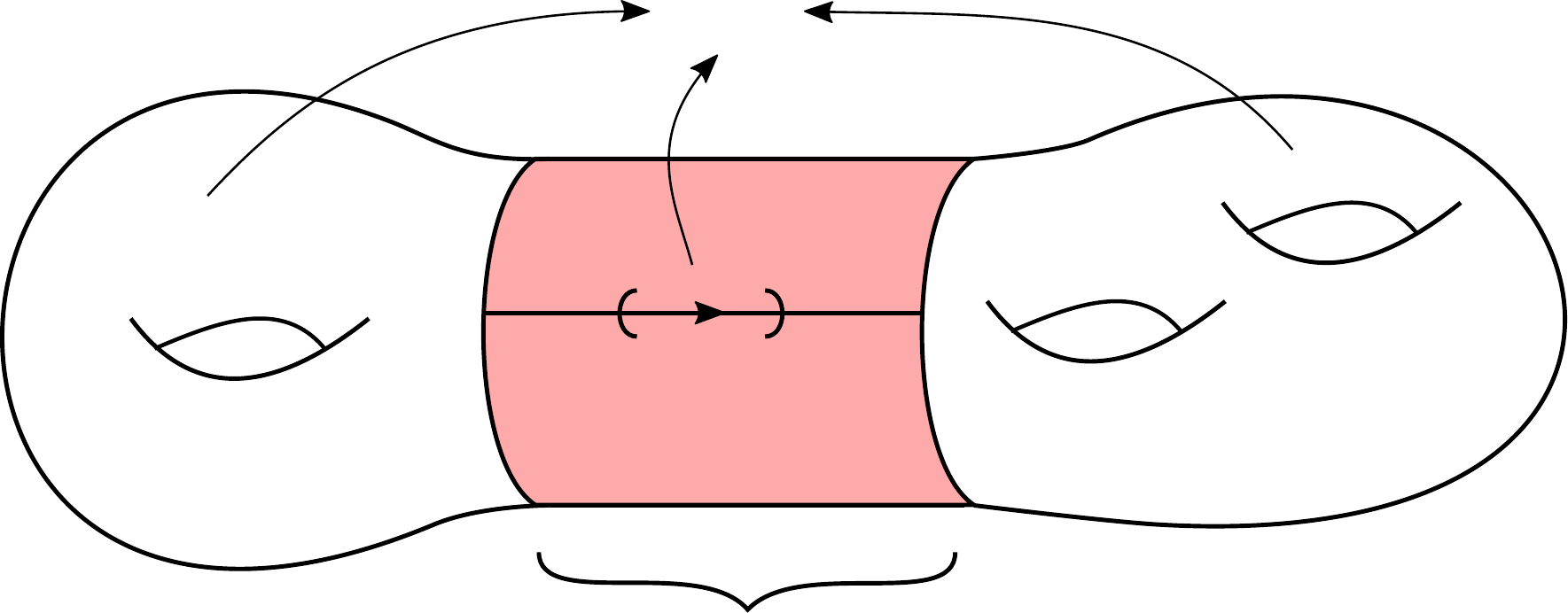}
		\put(35,-7){$N \times (-1,1)$}
		\put(7.5,7.5){$\Sigma_- \setminus \Sigma_0$}
		\put(77.5,12.5){$\Sigma_+ \setminus \Sigma_0$}
		\put(42.5,40){$B\G$}
		\put(20,35){$\ast$}
		\put(77.5,35){$\ast$}
		\put(45,25){$\gamma_{\g^{-1}}$}
		\put(35,12.5){$-\frac{1}{2}$}
		\put(48,12.5){$\frac{1}{2}$}
	\end{overpic} 
	\vspace{0.5cm} 
	\caption{The map $\varphi$ on a collar-gluing.}
\label{fig:excisionDmfd}
\end{subfigure}
\begin{subfigure}[t]{0.5\textwidth}
\vspace{-1.5cm}
\centering
	\begin{overpic}[scale=0.5
		,tics=10]
		{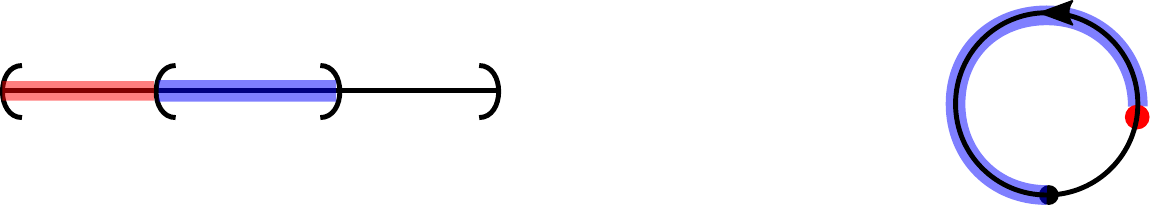}
		\put(-15,8.5){$N~\times$}
		\put(-5,0){$-1$}
		\put(10,0){$-\frac{1}{2}$}
		\put(27.5,0){$\frac{1}{2}$}
		\put(41,0){$1$}
		\put(90,-5){$\ast$}
		\put(102.5,7.5){\textcolor{red}{$\gamma_{\g^{-1}}(t_0)$}}
		\put(75,22.5){\textcolor{blue}{$\gamma_{\g^{-1}}(t \geq t_0)$}}
		\put(47,8.5){$,~~H(s,t_0) = $}
	\end{overpic}  
	\vspace{1cm}
	\caption{Sketch of the homotopy $H$ for some fixed $t_0 \in [0,1]$.}
\label{fig:htpyExcisionDmfd}
\end{subfigure}
\caption{ }
\end{figure}
We then extend $\theta$ to an equivalence of $\G$-manifolds, where the collar $N \times (-1,1)$ is equipped with the constant map to $B\G$. The equivalence is established using a homotopy $H \colon  (\theta^{-1})^*\varphi|_{\Sigma_0} \Rightarrow \ast$, which is given by $\gamma_{\g^{-1}}$ for every point in $N \times (-1,-\tfrac{1}{2}]$, continues the loop $\gamma_{\g^{-1}}$ to its end point on $N \times (-\tfrac{1}{2},\tfrac{1}{2})$ and is constant on $N \times [\tfrac{1}{2},1)$, as sketched in Figure \ref{fig:htpyExcisionDmfd}. Hence, we get an equivalence 
$$
\int_{(\Sigma_0,\varphi|_{\Sigma_0})} \mathcal{A} \cong \int_{N \times (-1,1)} \mathcal{A} \eqqcolon \mathcal{C} \ \ .
$$  
Given a balanced tensor category $\mathcal{A}$ with a $\G$-action, we now want to deduce the module structures featuring in the excision formulae for 
$\int_{(\Sigma,\varphi)} \mathcal{A}$. Denote by $\Sigma_-^\ast$ and $\Sigma_+^\ast$ two objects in $\G\text{-}\catf{Man}_2$, whose underlying manifolds agree with $\Sigma_-$ and $\Sigma_+$, but whose maps to $B\G$ are assumed to be constant. The value of factorisation homology on these manifolds naturally defines module categories $\mathcal{M}_-$ and $\mathcal{M}_+$ over the $E_1$-algebra $\mathcal{C}$. In order to obtain an explicit description of the module structures obtained by excision, note that the homotopy $H$ from above can be used to construct an equivalence $\theta_+ \colon \Sigma_+ \xrightarrow{\cong} \Sigma_+^\ast$ with homotopy $H_+ \colon (\theta^{-1})^*\varphi|_{\Sigma_+}\Rightarrow \ast$ in $\G\text{-}\catf{Man}_2$. We use this equivalence to identify 
$\int_{(\Sigma_+,\varphi|_{\Sigma_+})} \mathcal{A} \cong \int_{\Sigma_+^*} \mathcal{A}$ as categories. This equivalence can be promoted to an equivalence of module categories, i.e.~the following diagram commutes:
\begin{center}
\begin{tikzcd}
N \times (-1,1) \sqcup \Sigma_+ \arrow[r] \arrow[d,"\id \sqcup (\theta_+{,} H_+)",swap] & \Sigma_+ \arrow[d,"(\theta_+{,}H_+)"] \\
N \times (-1,1) \sqcup \Sigma_+^\ast \arrow[r] & \Sigma_+^\ast 
\end{tikzcd}
\end{center}
We can see that the action of $\mathcal{C}$ on $\int_{(\Sigma_+^*,\varphi|_{\Sigma_+})}\mathcal{A}$ is precisely given by the $\mathcal{C}$-module structure of $\mathcal{M}_+$. On the contrary, the situation is a bit more involved for the $\mathcal{C}$-module structure of $\int_{(\Sigma_-, \varphi|_{\Sigma_-})} \mathcal{A}$: We cannot simply identify $\Sigma_- \cong \Sigma_-^\ast$ via $H$ since the homotopy is not constant near $N \times \{-1\}$. However, we can construct an equivalence $\theta_-\colon \Sigma_- \xrightarrow{\cong} \Sigma_-^\ast$ from a homotopy $H_-$, which is defined by using the loop $\gamma_\g$ similarly to how we used $\gamma_{\g^{-1}}$ above. This gives rise to an identification $\int_{(\Sigma_-,\varphi|_{\Sigma_-})} \mathcal{A} \cong \int_{\Sigma_-^\ast}\mathcal{A}$ together with a weakly commuting diagram 
\begin{equation}
	\begin{tikzcd}
	\Sigma_- \sqcup N \times (-1,1) \ar[d, "(\theta_-{,} H_-) \sqcup (\id{,}\gamma_{\g})",swap] \ar[r] & \Sigma_- \ar[d,"(\theta_-{,} H_-)"] \\ 
	\Sigma_-^\ast \sqcup N \times (-1,1) \ar[r] & \Sigma_-^\ast 
	\end{tikzcd}
\end{equation}  
in $\G\text{-}\Man_2$. From the horizontal maps we deduce that the module structure relevant for excision is obtained by twisting by the $\G$-action on $\Ca$:
\begin{align}
\act_-^\g \colon \mathcal{M}_- \boxtimes \mathcal{C} \xrightarrow{\id_{\mathcal{M}_-}\boxtimes \vartheta(\g)} \mathcal{M}_- \boxtimes \mathcal{C} \xrightarrow{\act_-} \mathcal{M}_-  \ \ .
\end{align}
We write $\mathcal{M}_{-,\g}$ for this module category. Combining everything we arrive at 
\begin{align}\label{Eq: twisted excision}
\int_{(\Sigma,\varphi)} \mathcal{A} \cong \mathcal{M}_{-,\g} \underset{\mathcal{C}}{\boxtimes} \mathcal{M}_+ \ \ .
\end{align} 
\end{example}

\begin{remark}
Notice that alternatively we could have chosen a trivialisation of $\Sigma_0$ which extends to $\Sigma_-$, instead of $\Sigma_+$, which would have resulted in a twisting of $\mathcal{M}_+$ by $\g^{-1}$. In this sense the module structures featuring in excision for $\G$-structured oriented 2-manifolds are not unique, though the value of the relative tensor product is. 
\end{remark}

\subsection{Actions of diagram automorphisms and their quantisation}\label{Ex: Aut-extension}

For applications to quantum physics, we will be mostly interested in factorisation homology for the ribbon category $\Rep_q(G)$. In this section we will show that $\Rep_q(G)$ admits an $\Out(G)$-action, which can be seen as a quantisation of the $\Out(G)$-symmetry in gauge theory. \par 
The outer automorphism group $\Out(G)$ of $G$ is finite and can be identified with the group of Dynkin diagram automorphisms. Concretely, one finds for the non-trivial outer automorphism groups 
\begin{center}
	\begin{tabular}{|c||c|c|c|c|}
		\hline {Type} & {$A_n$ , $n\geq 2$}& {$D_n$ , $n>4$} & {$D_4$} &
		{$E_6$}
		\\[4pt] \hline
		$\Out(G)$ & $\Z_2$ & $\Z_2$ & $S_3$ & $\Z_2$
		\\\hline
	\end{tabular}
\end{center} The identification of outer automorphisms and Dynkin digram automorphisms provides an explicit splitting $\Out(G) \longrightarrow \Aut(G)$ and allows us to write down the short exact sequence 
\begin{align}
1\longrightarrow G \longrightarrow G\rtimes \Out(G) \longrightarrow \Out(G) \longrightarrow 1
\end{align}
containing the semi-direct product. \par 
The category $\Rep(G)$ of $G$-representations is a symmetric monoidal ribbon category and hence in particular a framed $\mathsf{E}_2$-algebra. The finite group $\Out(G)$ acts naturally on the category $\Rep(G)$ by pulling back representations along the inverse and this symmetry extends to the representation category of the corresponding quantum group, see Proposition \ref{prp:OutGactionRephG} below. \par 
We will use the following notation and conventions. We consider a finite-dimensional simple complex Lie algebra $\mathfrak{g}$ with Cartan matrix $(a_{ij})_{1 \leq i,j \leq n}$. We fix a Cartan subalgebra $\mathfrak{h} \subset \mathfrak{g}$ and select a set of simple roots $\Pi = \{\alpha_1, \dots, \alpha_n\}$. We write $\Lambda$ for the weight lattice and we choose a symmetric bilinear form $(\cdot,\cdot)$ on $\Lambda$ such that $(\alpha_i,\alpha_j) = a_{ij}$. For the rest of this paragraph we will restrict our attention to Lie algebras with Dynkin diagrams of type $A_n$ ($n \geq 2$), $D_n$ ($n \geq 4$), or $E_6$, since these are the only cases for which we have non-trivial Dynkin diagram automorphisms. \par 
The formal quantum group $U_\hbar(\mathfrak{g})$ is a Hopf algebra deformation of the universal enveloping algebra $U(\mathfrak{g})$ over $\C[[\hbar]]$ with generators $\{H_{\alpha_i}, X^{\pm}_{\alpha_i} \}_{\alpha_i \in \Pi}$, subjected to certain relations, see for example \cite[Section 6.5]{quantumgroups} for details. In order to define positive and negative root vectors, we fix a reduced decomposition $\omega_0 = s_{i_1} s_{i_2} \dots s_{i_N}$ of the longest element $\omega_0$ in the Weyl group of $\mathfrak{g}$. The positive and negative root vectors are then defined as
$$
X_{\beta_r}^{\pm} = T_{i_1}T_{i_2} \dots T_{i_{r-1}} X^{\pm}_{\alpha_{i_r}}
$$
in $U_\hbar(\mathfrak{g})$ by acting on the generators with elements $T_i \in \mathfrak{B}_\mathfrak{g}$ of the braid group associated to $\mathfrak{g}$~\cite[Section 8.1 ]{quantumgroups}. The formal quantum group $U_\hbar(\mathfrak{g})$ is quasi-triangular with universal R-matrix given by the multiplicative formula \cite[Theorem 8.3.9]{quantumgroups}
$$
\mathcal{R} = \Omega \widehat{\mathcal{R}}, \quad \Omega = \prod_{\alpha_i \in \Pi} e^{\hbar ( {a_{ij}^{-1} H_{\alpha_i} \otimes H_{\alpha_j} )}}, \quad \widehat{\mathcal{R}} = \prod_{\beta_r} \widehat{\mathcal{R}}_{\beta_r} \ \ ,
$$
where the order in the second product is such that the $\beta_r$-term is to the left of the $\beta_s$-term if $r > s$, and $\widehat{\mathcal{R}}_{\beta_r} = \exp_q((1-q^{-2}) X^+_{\beta_r} \otimes X^-_{\beta_r})$ for $q=\exp(\hbar)$. It is shown in \cite[Corollary 8.3.12]{quantumgroups} that $\mathcal{R}$ is independent of the chosen reduced decomposition of $\omega_0$. We denote by $\Rep_\hbar(G)$ the category of topologically free left modules over $U_\hbar(\mathfrak{g})$ of finite rank. This tensor category comes with a braiding defined via the universal R-matrix $\mathcal{R}$ of $U_\hbar(\mathfrak{g})$.

\begin{proposition}\label{prp:OutGactionRephG}
The braided tensor category $\Rep_\hbar(G)$ admits a left action of $\Out(G)$.
\end{proposition}

\begin{proof}
The outer automorphisms $\Out(G)$ can be identified with the automorphism group $\Aut(\Pi)$ of the Dynkin diagram of $\mathfrak{g}$. An element $ \kappa \in \Aut(\Pi)$ acts on the generators of $U_\hbar(\mathfrak{g})$ via
$$
H_{\alpha_i} \mapsto H_{\alpha_{\kappa(i)}}, \quad X^{\pm}_{\alpha_i} \mapsto X^{\pm}_{\alpha_{\kappa(i)}}.
$$
We thus get an action $\rho$ of $\Out(G)$ on the tensor category $\Rep_\hbar(G)$ defined by pulling back a representation along the inverse automorphism, i.e.~$\rho(\kappa)(X) = (\kappa^{-1})^*X$, for any $X \in \Rep_\hbar(G)$. It is left to show that the action preserves the braiding. The action of $\kappa$ on a positive, respectively negative, root vector is given by 
$$
\kappa.X^\pm_{\beta_r} = T_{\kappa(i_r)} \dots T_{\kappa(i_{r-1})} X^\pm_{\alpha_{\kappa(i_r)}} \ \ .
$$
We now make use of the following explicit expressions for $\omega_0$, details can be found for example in \cite[Section 3.19]{Humphreys}. First, divide the vertices of the Dynkin diagram into two nonempty disjoint subsets $S$ and $S'$, so that in each subset the corresponding simple reflections commute. Let $a$ and $b$ be the products of the simple reflections in $S$ and $S'$, respectively. For $A_n$ ($n$ odd), $D_n$ ($n \geq 4$) and $E_6$ we can set $\omega_0 = (ab)^{h}$, where $h$ is the respective Coxeter number. Whereas for $A_n$ ($n$ even), $\omega_0$ can be represented either as $\omega_0 = (ab)^{\frac{n}{2}}a$ or as $\omega_0 = b(ab)^{\frac{n}{2}}$. We thus see that $\kappa$ sends a given reduced decomposition of the longest Weyl group element $\omega_0$ to another reduced decomposition of $\omega_0$. But since the R-matrix is independent of the chosen reduced decomposition the result follows.
\end{proof}

\begin{proposition}
The action of $\Out(G)$ on $\Rep_\hbar(G)$ is compatible with the balancing automorphism of $\Rep_\hbar(G)$.
\end{proposition}

\begin{proof}
The balancing in $\Rep_\hbar(G)$ is given by acting with the ribbon element $c_\hbar = \exp(\hbar H_{\rho}) u_\hbar$ of $U_\hbar(\mathfrak{g})$, see \cite[Section 8.3.F]{quantumgroups}. Here, $H_\rho = \sum_{i = 1}^n \mu_i H_{\alpha_i}$ with coefficients $\mu_i = \sum_{j = 1}^n a^{-1}_{ij}$ and $u_\hbar = m_\hbar(S_\hbar \otimes \id)\mathcal{R}_{2,1}$ with $m_\hbar$ and $S_\hbar$ the multiplication and antipode in $U_\hbar(\mathfrak{g})$ respectively. It follows from Proposition \ref{prp:OutGactionRephG} that a Dynkin diagram automorphism $\kappa \in \Aut(\Pi)$ preserves the element $u_\hbar$. So it is left to show that $\kappa$ preserves the element $H_{\rho}$. Since the Cartan matrix is invariant under the Dynking diagram automorphism, we have $\mu_i = \sum_{j = 1}^n a^{-1}_{i,j} = \sum_{j = 1}^n a^{-1}_{\kappa(i),\kappa(j)} = \sum_{j = 1}^n a^{-1}_{\kappa(i),j} = \mu_{\kappa(i)}$ and thus $\kappa.H_{\rho} = H_\rho$.
\end{proof}

Let $q \in \mathbb{C}^\times$ be a non-zero complex number which is not a root of unity and let $U_q(\mathfrak{g})$ be the corresponding  specialisation of the rational form of $U_\hbar(\mathfrak{g})$. A precise definition of $U_q(\mathfrak{g})$ can be found e.g.~in \cite[Section 9]{quantumgroups}. We denote by $\Rep_q(G)$ the category of locally finite $U_q(\mathfrak{g})$-modules of type 1. Strictly speaking, $U_q(\mathfrak{g})$ is not quasi-triangular. However, it's representation category admits a braiding \cite[Section 10.1.D]{quantumgroups}. On a representation $V \otimes V' \in \Rep_q(G)$, the braiding is defined by the so-called quasi R-matrix $\Theta_{V,V'} = \tau \circ E_{V, V'} \widehat{\mathcal{R}}_{V,V'}$, where $\tau$ is the map swapping the tensor factors and $E_{V,V'}$ is an invertible operator on $V \otimes V'$ acting on the subspace $V_\lambda \otimes V'_\mu$ by the scalar $q^{(\lambda,\mu)}$, for $\lambda,\mu \in \Lambda$. Moreover, the standard ribbon element for $U_q(\mathfrak{g})$ acts on $V_\lambda$ as the constant $q^{-(\lambda,\lambda)-2(\lambda,\rho)}$ with $\rho$ the half-sum of positive roots, giving rise to the balancing in $U_q(\mathfrak{g})$. Hence, we get the $q$-analog of Proposition \ref{prp:OutGactionRephG}:

\begin{proposition}
The braided balanced tensor category $\Rep_q(G)$ admits a left action of $\Out(G)$.
\end{proposition}

\subsection{Reconstruction theorems for module categories}
The following is a brief recollection of \cite[Section 4]{bzbj} which will allow us to compute the value of
factorisation homology explicitly in terms of module categories over certain algebras in the next section. We start by recalling that the inclusion $\emptyset \hookrightarrow \Sigma$ of the empty manifold into a surface $\Sigma$ induces a canonical functor $1_{\Prc} \cong \mathsf{Vect}_k \to \int_{\Sigma} \mathcal{A}$ on the level of factorisation homology, see Section \ref{Sec:ReviewFH}. We thus have a distinguished object $\operatorname{Dist}_\Sigma \in \int_{\Sigma} \mathcal{A}$, given as the image of $k$ under this functor. If we assume that $\Sigma$ is not closed and we choose a marked interval in its boundary, there is a natural $\mathcal{A}$-module structure on $\int_\Sigma \mathcal{A}$, induced by embedding a disk along the marked interval. In order to study the factorisation homology of the surface $\Sigma$, we wish to describe the entire category $\int_\Sigma \mathcal{A}$ internally in terms of $\mathcal{A}$. To that end, following \cite{bzbj}, we will apply techniques from Barr-Beck monadic reconstruction to monads arising from adjunctions of module functors of the form $\act_{\operatorname{Dist}_{\Sigma}} \colon \mathcal{A} \to \int_\Sigma \mathcal{A}$. \par 
Applying monadic reconstruction techniques to module categories was first done for fusion categories in the work of Ostrik \cite{ostrikmodulecats}, and later in the setting of finite abelian categories in \cite{dsps}. Here, we will recall its further generalisation to categories in $\Prc$, as developed in~\cite[Section 4]{bzbj}. For the remainder of this section, let $\Aa$ be an abelian rigid tensor category in $\Prc$ and $\Ma$ an abelian right $\Aa$-module category with action functor $\act \colon \Ma \boxtimes \Aa \to \Ma$. For each $m \in \Ma$, the induced functor 
$$
\act_m \colon \Aa \to \Ma, \quad \act_m(a) \coloneqq m \otimes a
$$
admits a right adjoint which we denote $\act^R_m$. For any pair of objects $m,n \in \Ma$, define the \emph{internal morphisms} from $m$ to $n$ as the object $\underline{\Hom}_\Aa(m,n)=\act_m^R(n) \in \Aa$ representing the functor $a \mapsto \Hom_{\mathcal{M}}(m \otimes a, n)$. Then, there is a natural algebra internal to $\Aa$ given by $\underline{\End}_\Aa(m) \coloneqq \underline{\Hom}_\Aa(m,m)$, which is called the \emph{internal endomorphism algebra} of $m$. For each $m \in \mathcal{M}$, we get a functor 
$$
\widetilde{\act^R_m} \colon \mathcal{M} \to (\act^R_m \circ \act_m)\operatorname{-mod}_\Aa
$$
sending an object $n \in \mathcal{M}$ to $\underline{\Hom}_\mathcal{A}(m,n)$ with canonical action $\act^R_m \circ \act_m \circ \act^R_m (n) \to \act^R_m(n)$ given by the counit of the adjunction. The monadicity theorem (see Theorem \ref{thm:reconstruction} below) then tells us when this functor is an equivalence. In order to state the theorem, we adopt the following terminology.

\begin{definition}
\normalfont
An object $m \in \Ma$ is called 
\begin{itemize}
\item an \emph{$\Aa$-generator} if $\act^R_m$ is faithful,
\item \emph{$\Aa$-projective} if $\act^R_m$ is colimit-preserving,
\item an \emph{$\Aa$-progenerator} if it is both $\Aa$-projective and an $\Aa$-generator.
\end{itemize}
\end{definition}

\begin{theorem}[{\cite[Theorem 4.6]{bzbj}}]
Let $m \in \Ma$ be an $\Aa$-progenerator. Then the functor 
$$
\widetilde{\act^R_m} \colon \Ma \xrightarrow{\cong} \underline{\End}_\Aa(m)\operatorname{-mod}_\Aa, 
$$
is an equivalence of $\Aa$-module categories, where $\Aa$ acts on the right by the tensor product.
\end{theorem}\label{thm:reconstruction}

When computing factorisation homology of a surface, we will make extensive use of $\otimes$-excision, as explained in Section \ref{Sec:Excision}. On a categorical level this means that we wish to apply monadic reconstruction to the relative Deligne-Kelly tensor product of two module categories. For this, notice that if $\Ma$ is a left $\Aa$-module category and $a$ an algebra in $\Aa$, one can use the $\Aa$-action on $\mathcal{M}$ to define the category of $a$-modules in $\Ma$, which we denote by $a\text{-mod}_\Ma$. 

\begin{theorem}[{\cite[Theorem 4.12]{bzbj}}]
Let $\Ma_-$ and $\Ma_+$ be right-, respectively left $\Aa$-module categories. Assume that $m \in \Ma_-$ and $n \in \Ma_+$ are both $\Aa$-progenerators. Then there are equivalences
$$
\Ma_- \boxtimes_\mathcal{A} \Ma_+ \cong \underline{\End}_\Aa(m)\operatorname{-mod}_{\Ma_+} \cong (\underline{\End}_\Aa(m),\underline{\End}_\Aa(n))\operatorname{-bimod}_\mathcal{A}
$$
of categories.
\end{theorem}\label{thm:reconstructionreltensorprod}

The following special case will be of particular interest for us: We assume that $\mathcal{M}_+$ is itself a tensor category and that the $\mathcal{A}$-module structure on $\mathcal{M}_+$ is induced by a tensor functor $F\colon \Aa \to \Ma_+$, which is such that every object in $\Ma_+$ appears as a subobject, or equivalently a quotient, of an object in the image of $F$. Tensor functors with this property are called \emph{dominant}. When in this setting, we have the following base-change formula:

\begin{corollary}[{\cite[Corollary 4.13]{bzbj}}]\label{crl:basechange}
Let $F\colon \Aa \to \Ma_+$ be a dominant tensor functor and $m \in \Ma_-$ an $\Aa$-progenerator. Then there is an equivalence of $\Ma_+$-module categories 
$$
\Ma_- \boxtimes_\Aa \Ma_+ \cong F(\underline{\End}_\Aa(m))\operatorname{-mod}_{\Ma_+}.
$$
\end{corollary}

\section{Factorisation homology for surfaces with $\G$-bundles}\label{Sec:FHD-mfds}

In this section we use excision and reconstruction theorems to compute factorisation homology of an abelian rigid balanced braided tensor category $\mathcal{A}$ equipped with $\G$-action, for $\G$ a finite group, over a surface $\Sigma$ with principal $\G$-bundles and at least one boundary component.\par 
Furthermore, we study the algebraic structure corresponding to the evaluation on annuli, boundary conditions 
and point defects.

\subsection{Reconstruction for rigid braided tensor categories with group action}\label{Sec: G-Reconstruction}

For $d \in D$, consider the right $\Aa^{\boxtimes 2}$-module category $\Ma_\g$, whose underlying category is $\Aa$ and the action is
\begin{equation}\label{eq:twsited regular action}
\reg^{\g} \colon \Ma_{\g} \boxtimes \Aa \boxtimes \Aa \xrightarrow{\id \boxtimes \id \boxtimes \vartheta(\g)} \Ma_{\g} \boxtimes \Aa \boxtimes \Aa \xrightarrow{T^3} \Ma_{\g} \ \ ,
\end{equation}
where $T^3$ is the iterated tensor product functor $x \boxtimes y \boxtimes z \mapsto x \otimes y \otimes z$.
\begin{lemma}\label{lma:pg}
$1_\mathcal{A}$ is a progenerator for the twisted regular action $\reg^{\g}$.
\end{lemma} 
\begin{proof}
The unit $1_\mathcal{A}$ is a progenerator for the right regular action (see \cite[Proposition 4.15]{bzbj}). Since $\vartheta(\g)$ is an automorphism of $\mathcal{A}$, it is also a progenerator for $\reg^{\g}$.
\end{proof}

The internal endomorphism algebra $\underline{\End}_{\Aa^{\boxtimes 2}}(1_\mathcal{A})$ can 
be explicitly described by the coend 
\begin{align}
\int^{V \in \text{comp}(\mathcal{A})} V^\vee \boxtimes \vartheta(\g^{-1}).V \ \ ,
\end{align}
where $V^\vee$ is the dual of $V$ and the colimit is taken over compact objects in $\Aa$. To derive the above expression it is enough to note that the action is given by pre-composition of the regular action with the automorphism $\id \boxtimes \vartheta(\g)$ with adjoint $\id \boxtimes \vartheta(\g^{-1})$ and use Remark 4.16 of \cite{bzbj}. Applying the tensor product functor $T\colon \mathcal{A} \boxtimes \mathcal{A}
\longrightarrow \mathcal{A}$ to $\underline{\End}_{\Aa^{\boxtimes 2}}(1_\mathcal{A})$ we get the object 
\begin{equation}\label{eq:twistedcoend}
\mathcal{F}_{\mathcal{A}}^{\g} \coloneqq \int^{V \in \text{comp}(\mathcal{A})} V^\vee \otimes \vartheta(\g^{-1}).V \ \ .
\end{equation}
Notice that for the identity element $e \in \G$, this is Lyubashenko's coend $\int V^\vee \otimes V$ \cite{LyubCoend}, which in particular is a braided Hopf algebra in $\Aa$.

\begin{example}\label{ex:CoendRep(H)}
Let $H$ be a ribbon Hopf algebra with $\G$-action, meaning that an element $\g \in \G$ acts on $H$ by Hopf algebra automorphisms, the universal R-matrix is $\G$-invariant, i.e.~$\mathcal{R} \in (H \otimes H)^{\G}$ and the ribbon element is preserved by the $H$-action. Let $\Rep(H)$ be the braided tensor category of locally finite left modules over $H$ on which the elements $\g \in \G$ act through the pullback of representations along $\g^{-1}$. It is a well-known result that at the identity element $e \in \G$, the algebra $\Fa_{\Rep(H)}^{e}$ is identified with the braided dual of $H$, also known as the reflection equation algebra (REA), equipped with the coadjoint action. Its underlying vector space is given by the matrix coefficients $H^\circ$ of finite dimensional $H$-representations. As an algebra, the REA can be obtained from the so-called Faddeev-Reshetikhin-Takhtajan (FRT) algebra via twisting by a cocycle given in terms of the universal R-matrix \cite{twistREalg}. In more detail, the FRT algebra is identified with the coend 
$$
\mathcal{F}_{\text{FRT}} = \int^{V \in \Rep^{\fd}(H)} V^\vee \boxtimes V \in \Rep(H)^{\rev} \boxtimes \Rep(H) \ \ ,
$$
where $\Rep(H)^{\rev}$ is the category with the opposite monoidal product, with multiplication $m_{\text{FRT}}$ induced by the canonical maps 
$$
(V^\vee \boxtimes V) \otimes (W^\vee \boxtimes W) = (V^\vee \otimes^{\rev} W^\vee) \boxtimes (V \otimes W) \cong (W \otimes V)^\vee \boxtimes (W \otimes V) \xrightarrow{\iota_{V \otimes W}} \mathcal{F}_{\text{FRT}} \ \ .
$$
The REA is then given by the image of the FRT algebra under the composite functor 
\begin{equation}\label{eq:FRTandREA}
\Rep(H)^{\rev} \boxtimes \Rep(H) \xrightarrow{(\id,\sigma) \boxtimes \id} \Rep(H) \boxtimes \Rep(H) \xrightarrow{T} \Rep(H) \ \ ,
\end{equation}
where $(\id, \sigma)$ denotes the identity functor, equipped with a non-trivial tensor structure given by the braiding $\sigma$ in $\Rep(H)$. \par 
In the decorated case, we precompose the functor in \eqref{eq:FRTandREA} with the automorphism $1 \boxtimes \vartheta(d)$. Then, for any $\g \in \G$, the underlying vector space of $\Fa_{\Rep(H)}^{\g}$ is identified again with $H^\circ$ via 
$$
V^\vee \otimes {\g}^*V \xrightarrow{\iota_V} H^\circ, \quad \iota_V(\phi \otimes v) = \phi( - \triangleright ({\g}^{-1})^*v),
$$
for any $V \in \Rep^{\fd}(H)$, but $H^\circ$ is now equipped with the twisted coadjoint action $\ad^*_d(h \otimes \phi)(v)=\phi(S(h_{(1)})(-){\g}.h_{(2)}\triangleright v)$. The multiplication on the coend algebra is defined in terms of its universal property. Concretely, consider the following dinatural map 
$$
f_{V,W} \colon  V^\vee \otimes {\g}^*V  \otimes  W^\vee \otimes {\g}^*W  \xrightarrow{\sigma_{{\g}^*V,W^\vee \otimes {\g}^*W}}  V^\vee \otimes W^\vee \otimes {\g}^*W \otimes {\g}^*V  \xrightarrow{\cong}(W \otimes V)^\vee \otimes {\g}^*(W \otimes V) \xrightarrow{\iota_{W \otimes V}} \Fa_{\Rep(H)}^{\g} 
$$
Then there exists a unique multiplication map $\Fa_{\Rep(H)}^{\g} \otimes \Fa_{\Rep(H)}^{\g} \xrightarrow{m} \Fa_{\Rep(H)}^{\g}$ such that $f_{V,W} = m \circ (\iota_V \otimes \iota_W)$. Explicitly, the product of $\phi,\psi \in \Fa_{\Rep(H)}^{\g}$ is given by
$$
m_{\text{REA}}^{\g}(\phi \otimes \psi) = m_{\text{FRT}}(\phi(\mathcal{R}_1(-){\g}.\mathcal{R}'_1) \otimes \psi(S(\mathcal{R}'_2)\mathcal{R}_2(-)) \ \ .
$$
In the language of \cite{twistREalg}, we thus find that $\Fa^{\g}_{\Rep(H)}$ is obtained by twisting the module algebra $(H^\circ, \ad^*_\g)$ by the cocycle $\mathcal{R}_1 \otimes {\g}.\mathcal{R}'_1 \otimes \mathcal{R}_2 \mathcal{R}'_2 \otimes 1$, where we write $\mathcal{R} = \mathcal{R}_1 \otimes \mathcal{R}_2$ and we use primes to distinguish different copies of the R-matrix.
\end{example}

\begin{example}
The category of finite-dimensional $U_q(\mathfrak{g})$-modules of type 1 is a semisimple braided tensor category via the quasi R-matrix $\Theta$. The quantised coordinate algebra $\mathcal{O}_q(G)$ is then defined as the algebra of matrix coefficients of objects in this category. Given an automorphism $\kappa \in \Out(G)$, the twisted coend algebra \eqref{eq:twistedcoend} takes the form $T(\underline{\End}_{\Rep_q(G)^{\boxtimes 2}}(\mathbb{C})) \cong \bigoplus_{V} V^\vee \otimes \kappa^*V$, where the sum runs over the simple objects. By a quantum version of the Peter-Weyl theorem (see for example \cite[Proposition 4.1]{QPeterWeyl}) we get an identification $\bigoplus_{V} V^\vee \otimes \kappa^*V \cong \mathcal{O}_q(G)$ as vector spaces, and by the previous example, we thus find that the coend algebra is isomorphic to $\mathcal{O}_q(G)$ with $\kappa$-twisted multiplication.
\end{example}

\subsection{Computation on punctured surfaces}\label{sec:compFH}

Throughout this section we consider connected oriented surfaces with at least one boundary component. We can pick a ciliated fat graph model to describe the surface $\Sigma$ we want to work with, which in~\cite{bzbj} is conveniently defined via a \emph{gluing-pattern}, that is a bijection $P \colon \{1,1',\dots,n,n'\} \to \{1,\dots,2n\}$, such that $P(i) < P(i')$. Here, $n$ is the number of edges of the fat-graph model of $\Sigma$. Given a gluing pattern $P$, we can reconstruct $\Sigma$ as depicted in Figure \ref{fig:gluingpattern}, namely by gluing $n$ disks $\mathbb{D}_{\bullet \bullet}$ with two marked intervals each to a disk ${_{\bullet^{2n}}}\mathbb{D}_\bullet$ with $2n+1$ marked intervals, thereby gluing the intervals $i$ and ${i}'$ to $P(i)$ and $P(i')$, respectively. 

\begin{definition}\label{Def: Gluing pattern}
A \emph{$\G$-labeled gluing pattern} is a gluing pattern $P\colon \{ 1,1',\dots ,n ,n' \} \longrightarrow \{ 1, \dots , 2n \}$ together with $n$ elements $\g_1,\dots , \g_n \in \G$.  
\end{definition}

Notice that the fundamental group of a genus $g$ surface with $r+1$ boundary components is free on $n = 2g + r$ generators. This implies that a $\G$-labeled gluing pattern determines a principal $\G$-bundle on the surface constructed from the gluing pattern. Furthermore, up to equivalence all principal $\G$-bundles on surfaces with at least one boundary arise in this way. 

\begin{figure}[h!]
\begin{subfigure}{0.5\textwidth}
\centering
\vspace{0.25cm}
	\begin{overpic}[scale=0.5
		,tics=10]
		{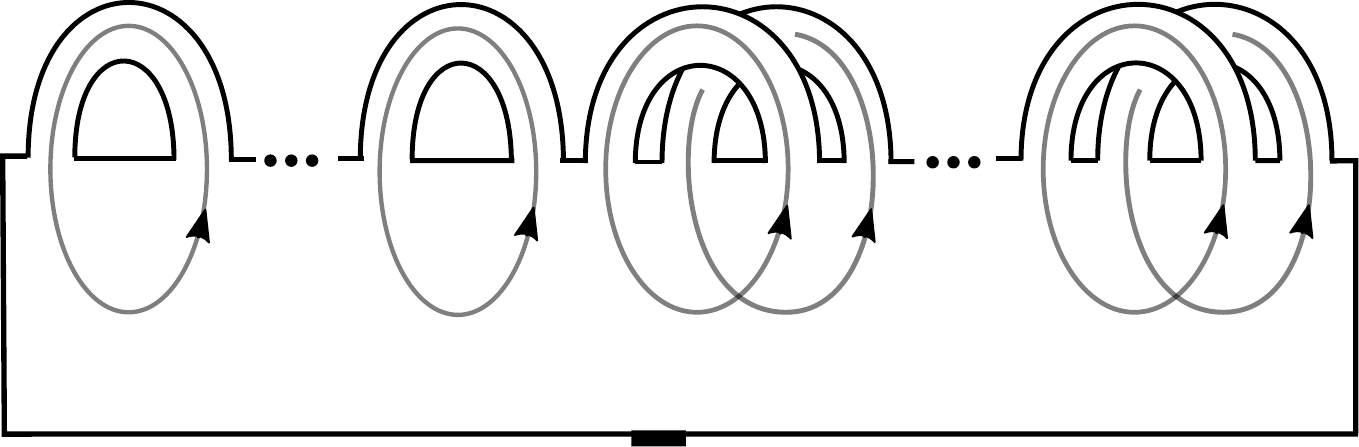}
		\put(5.25,4.25){\footnotesize $[\gamma_{\g_1}]$}
		\put(29.5,4.25){\footnotesize $[\gamma_{\g_r}]$}
		\put(42.5,4.25){\footnotesize $[\gamma_{\g_{r+1}}]$}
		\put(56,4.25){\footnotesize $[\gamma_{\g_{r+2}}]$}
		\put(73.5,4.25){\footnotesize $[\gamma_{\g_{n-1}}]$}
		\put(87.5,4.25){\footnotesize $[\gamma_{\g_{n}}]$}
	\end{overpic}  
	\vspace{0.5cm}
	\caption{Generators of the homotopy group $\pi_1(\Sigma)$.}
\end{subfigure}
\begin{subfigure}{.5\textwidth}
  \centering
  \vspace{0.25cm}
  \begin{overpic}[scale=0.5
		,tics=10]
		{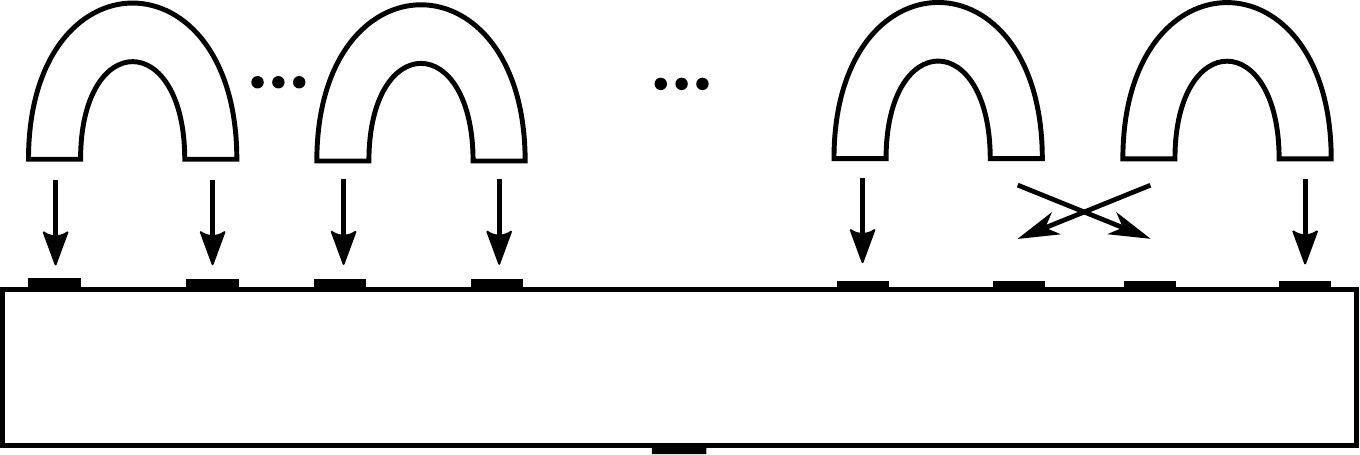}
		\put(2,7.5){\tiny $P(1)$}
		\put(12,7.5){\tiny $P(1')$}
		\put(22.5,7.5){\tiny $P(r)$}
		\put(32.5,7.5){\tiny $P(r')$}
		\put(60,7.5){$\cdots$}
		\put(70,7.5){\tiny $P(n)$}
		\put(84.5,6.75){$|$}
		\put(77.5,2.5){\tiny $P((n-1)')$}
		\put(90,7.5){\tiny $P(n')$}
		\put(-1.5,30){\footnotesize $\g_1$}
		\put(20,30){\footnotesize $\g_r$}
		\put(52.5,30){\footnotesize $\g_{n-1}$}
		\put(78.5,30){\footnotesize $\g_n$}
	\end{overpic}  
	\vspace{0.5cm}
  \caption{Gluing a surface from a decorated gluing pattern.}
  \label{fig:gluingpattern}
\end{subfigure}
\caption{ }
\end{figure}
 
For a $\G$-labeled gluing pattern $(P,\g_1\dots \g_n)$ we are going to define an algebra $a_P^{\g_1,\dots, \g_n} \in \mathcal{A}$. As an object in $\Aa$, it is defined by the tensor product 
\begin{align}
a_P^{\g_1,\dots , \g_n} \coloneqq \bigotimes_{i=1}^n \mathcal{F}_{\mathcal{A}}^{\g_i} \ \ ,
\end{align}  
where the $\mathcal{F}_{\mathcal{A}}^{\g_i}$ are defined by the coend in Equation \eqref{eq:twistedcoend}. The gluing pattern can be used to define an algebra structure on this object in complete analogy with \cite{bzbj}. To that end, we will use the following terminology: Two labeled discs $\mathbb{D}_{\bullet \bullet}^{\g_i}$ and $\mathbb{D}_{\bullet\bullet}^{\g_j}$ with $i < j$ are called 
\begin{itemize}
    \item \emph{positively (negatively) linked} if $P(i) < P(j) < P(i') < P (j')$ ($P(j) < P(i) < P (j') < P (i')$)
    \item \emph{positively (negatively) nested} if $P(i) < P(j) < P(j') < P(i')$ ($P(j) < P(i) < P(i') < P(j')$)
    \item \emph{positively (negatively) unlinked} if $P(i) < P(i') < P(j) < P (j')$ ($P(j) < P(j') < P(i) < P (i')$)
\end{itemize}
To each of the above cases, we assign a crossing-morphism as depicted in Figure \ref{Fig:crossingmorphisms} below. Notice that the crossing-morphism in the nested case differs from the one given in \cite[Definition 5.8]{bzbj}.

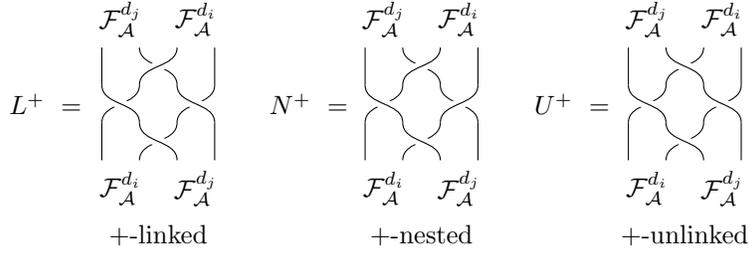
\begin{figure}[h!]
\begin{equation}
	\begin{tikzpicture}
	\begin{pgfonlayer}{nodelayer}
		\node [style=none] (0) at (-7, 0) {};
		\node [style=none] (1) at (-6, 0) {};
		\node [style=none] (2) at (-5, 0) {};
		\node [style=none] (3) at (-4, 0) {};
		\node [style=none] (4) at (-6, 1) {};
		\node [style=none] (5) at (-5, 1) {};
		\node [style=none] (6) at (-7, 1) {};
		\node [style=none] (7) at (-4, 1) {};
		\node [style=none] (8) at (-6, 2) {};
		\node [style=none] (9) at (-5, 2) {};
		\node [style=none] (10) at (-7, 2) {};
		\node [style=none] (11) at (-4, 2) {};
		\node [style=none] (12) at (-6, 3) {};
		\node [style=none] (13) at (-5, 3) {};
		\node [style=none] (14) at (-7, 3) {};
		\node [style=none] (15) at (-4, 3) {};
		\node [style=none] (49) at (-5.5, -2) {$+\text{-linked}$};
		\node [style=none] (50) at (0, 0) {};
		\node [style=none] (51) at (1, 0) {};
		\node [style=none] (52) at (2, 0) {};
		\node [style=none] (53) at (3, 0) {};
		\node [style=none] (54) at (1, 1) {};
		\node [style=none] (55) at (2, 1) {};
		\node [style=none] (56) at (0, 1) {};
		\node [style=none] (57) at (3, 1) {};
		\node [style=none] (58) at (1, 2) {};
		\node [style=none] (59) at (2, 2) {};
		\node [style=none] (60) at (0, 2) {};
		\node [style=none] (61) at (3, 2) {};
		\node [style=none] (62) at (1, 3) {};
		\node [style=none] (63) at (2, 3) {};
		\node [style=none] (64) at (0, 3) {};
		\node [style=none] (65) at (3, 3) {};
		\node [style=none] (67) at (1.5, -2) {$+\text{-nested}$};
		\node [style=none] (68) at (7, 0) {};
		\node [style=none] (69) at (8, 0) {};
		\node [style=none] (70) at (9, 0) {};
		\node [style=none] (71) at (10, 0) {};
		\node [style=none] (72) at (8, 1) {};
		\node [style=none] (73) at (9, 1) {};
		\node [style=none] (74) at (7, 1) {};
		\node [style=none] (75) at (10, 1) {};
		\node [style=none] (76) at (8, 2) {};
		\node [style=none] (77) at (9, 2) {};
		\node [style=none] (78) at (7, 2) {};
		\node [style=none] (79) at (10, 2) {};
		\node [style=none] (80) at (8, 3) {};
		\node [style=none] (81) at (9, 3) {};
		\node [style=none] (82) at (7, 3) {};
		\node [style=none] (83) at (10, 3) {};
		\node [style=none] (85) at (8.5, -2) {$+\text{-unlinked}$};
		\node [style=none] (90) at (-6.5, -0.75) {$\mathcal{F}_{\mathcal{A}}^{d_i}$};
		\node [style=none] (91) at (-4.5, -0.75) {$\mathcal{F}_{\mathcal{A}}^{d_j}$};
		\node [style=none] (92) at (-4.5, 3.75) {$\mathcal{F}_{\mathcal{A}}^{d_i}$};
		\node [style=none] (93) at (-6.5, 3.75) {$\mathcal{F}_{\mathcal{A}}^{d_j}$};
		\node [style=none] (94) at (0.5, -0.75) {$\mathcal{F}_{\mathcal{A}}^{d_i}$};
		\node [style=none] (95) at (7.5, -0.75) {$\mathcal{F}_{\mathcal{A}}^{d_i}$};
		\node [style=none] (96) at (2.5, 3.75) {$\mathcal{F}_{\mathcal{A}}^{d_i}$};
		\node [style=none] (97) at (9.5, 3.75) {$\mathcal{F}_{\mathcal{A}}^{d_i}$};
		\node [style=none] (98) at (2.5, -0.75) {$\mathcal{F}_{\mathcal{A}}^{d_j}$};
		\node [style=none] (99) at (9.5, -0.75) {$\mathcal{F}_{\mathcal{A}}^{d_j}$};
		\node [style=none] (100) at (0.5, 3.75) {$\mathcal{F}_{\mathcal{A}}^{d_j}$};
		\node [style=none] (101) at (7.5, 3.75) {$\mathcal{F}_{\mathcal{A}}^{d_j}$};
		\node [style=none] (103) at (-8.5, 1.5) {$L^+~=$};
		\node [style=none] (104) at (-1.5, 1.5) {$N^+~=$};
		\node [style=none] (105) at (5.5, 1.5) {$U^+~=$};
	\end{pgfonlayer}
	\begin{pgfonlayer}{edgelayer}
		\draw (0.center) to (6.center);
		\draw (3.center) to (7.center);
		\draw [in=-90, out=90] (1.center) to (5.center);
		\draw [style=over, in=-90, out=90] (2.center) to (4.center);
		\draw [in=-90, out=90] (6.center) to (8.center);
		\draw [in=-90, out=90] (5.center) to (11.center);
		\draw [in=-90, out=90] (9.center) to (12.center);
		\draw (11.center) to (15.center);
		\draw (10.center) to (14.center);
		\draw [style=over, in=-90, out=90] (7.center) to (9.center);
		\draw [style=over, in=-90, out=90] (4.center) to (10.center);
		\draw [style=over, in=-90, out=90] (8.center) to (13.center);
		\draw (50.center) to (56.center);
		\draw (53.center) to (57.center);
		\draw [in=-90, out=90] (51.center) to (55.center);
		\draw [style=over, in=-90, out=90] (52.center) to (54.center);
		\draw [in=-90, out=90] (56.center) to (58.center);
		\draw [in=-90, out=90] (59.center) to (62.center);
		\draw (61.center) to (65.center);
		\draw (60.center) to (64.center);
		\draw [style=over, in=-90, out=90] (54.center) to (60.center);
		\draw [style=over, in=-90, out=90] (58.center) to (63.center);
		\draw (68.center) to (74.center);
		\draw (71.center) to (75.center);
		\draw [in=-90, out=90] (69.center) to (73.center);
		\draw [style=over, in=-90, out=90] (70.center) to (72.center);
		\draw [in=-90, out=90] (74.center) to (76.center);
		\draw [in=-90, out=90] (73.center) to (79.center);
		\draw (79.center) to (83.center);
		\draw (78.center) to (82.center);
		\draw [style=over, in=-90, out=90] (75.center) to (77.center);
		\draw [style=over, in=-90, out=90] (72.center) to (78.center);
		\draw [in=-90, out=90] (57.center) to (59.center);
		\draw [style=over, in=-90, out=90] (55.center) to (61.center);
		\draw [in=-90, out=90] (76.center) to (81.center);
		\draw [style=over, in=-90, out=90] (77.center) to (80.center);
	\end{pgfonlayer}
\end{tikzpicture}
\end{equation}
 \caption{Definition of crossing-morphisms $L^+,N^+,U^+ \colon \Fa_\Aa^{\g_i} \otimes \Fa_\Aa^{\g_j} \to \Fa_\Aa^{\g_j} \otimes \Fa_\Aa^{\g_i}$ for positively linked, nested and unlinked decorated discs. Notice that we read the diagrams from bottom to top.}
 \label{Fig:crossingmorphisms}
 \end{figure}
 
Now, for each pair of indices $1 \leq i < j \leq n$, the restriction of the multiplication to $\Fa^{\g_i}_\Aa \otimes \Fa^{\g_j}_\Aa \subset a_P^{\g_1,\dots,\g_n}$ is defined by
$$
\Fa^{\g_i}_\Aa \otimes \Fa^{\g_j}_\Aa \otimes \Fa^{\g_i}_\Aa \otimes \Fa^{\g_j}_\Aa \xrightarrow{\id \otimes C \otimes \id} \Fa^{\g_i}_\Aa \otimes \Fa^{\g_i}_\Aa \otimes \Fa^{\g_j}_\Aa \otimes \Fa^{\g_j}_\Aa \xrightarrow{m \otimes m} \Fa^{\g_i}_\Aa \otimes \Fa^{\g_j}_\Aa \ \ ,
$$
where $C$ is either $L^\pm$, $N^\pm$ or $U^\pm$, depending on whether the decorated discs $\mathbb{D}^{\g_i}_{\bullet \bullet}$ and $\mathbb{D}^{\g_j}_{\bullet \bullet}$ are $\pm$-linked, $\pm$-nested or $\pm$-unlinked.\par 
Finally, given a $\G$-labeled gluing pattern, we wish to describe the module structure induced by gluing the marked disks $\mathbb{D}_{\bullet \bullet}^{\g_i}$ to the disk $_{\bullet^{2n}}\mathbb{D}_\bullet$ as sketched in Figure \ref{fig:gluingpattern}. To that end, we look at the example of a sphere with three punctures $(\mathbb{S}^2)_3$ and a $\G$-bundle described by the map $\varphi \colon \pi_1((\mathbb{S}^2)_3) \to \G$ sending the two generators of the fundamental group to $\g_1$ and $\g_2$, respectively. The corresponding gluing pattern is $P(1,1',2,2') = (1,2,3,4)$, decorated by the tuple $(\g_1,\g_2) \in \G^2$. We then choose a collar-gluing $(\mathbb{S}^2)_3 \cong \Sigma_- \cup_{\Sigma_0} \Sigma_+$ for the punctured sphere, as sketched on the right hand side of Figure \ref{fig:sphere}, and an equivalence in $\G\text{-}\mathsf{Man}_2$, so that the maps to $B\G$ are constant on $\Sigma_- \setminus \Sigma_0$ and $\Sigma_+ \setminus \Sigma_0$ and are given by the loops $\gamma_{\g_1}$ and $\gamma_{\g_2}$ on fixed open intervals in $\Sigma_0$, which are depicted by the red and blue intervals in Figure \ref{fig:sphere}. We immediately see that we are in a situation similar to Example \ref{Ex:Excision D-mfd}: The right $\Aa \boxtimes \Aa$-module structure on $\int_{\mathbb{D}^{\g_i}_{\bullet \bullet}} \Aa$, for $i=1,2$, is the twisted regular action $\reg^{\g_i}$ from \eqref{eq:twsited regular action}. The module structure for more general decorated gluing patterns can be worked out analogously.
 
\begin{figure}[H]
 \centering
  \vspace{0.25cm}
  \begin{overpic}[scale=0.4
		,tics=10]
		{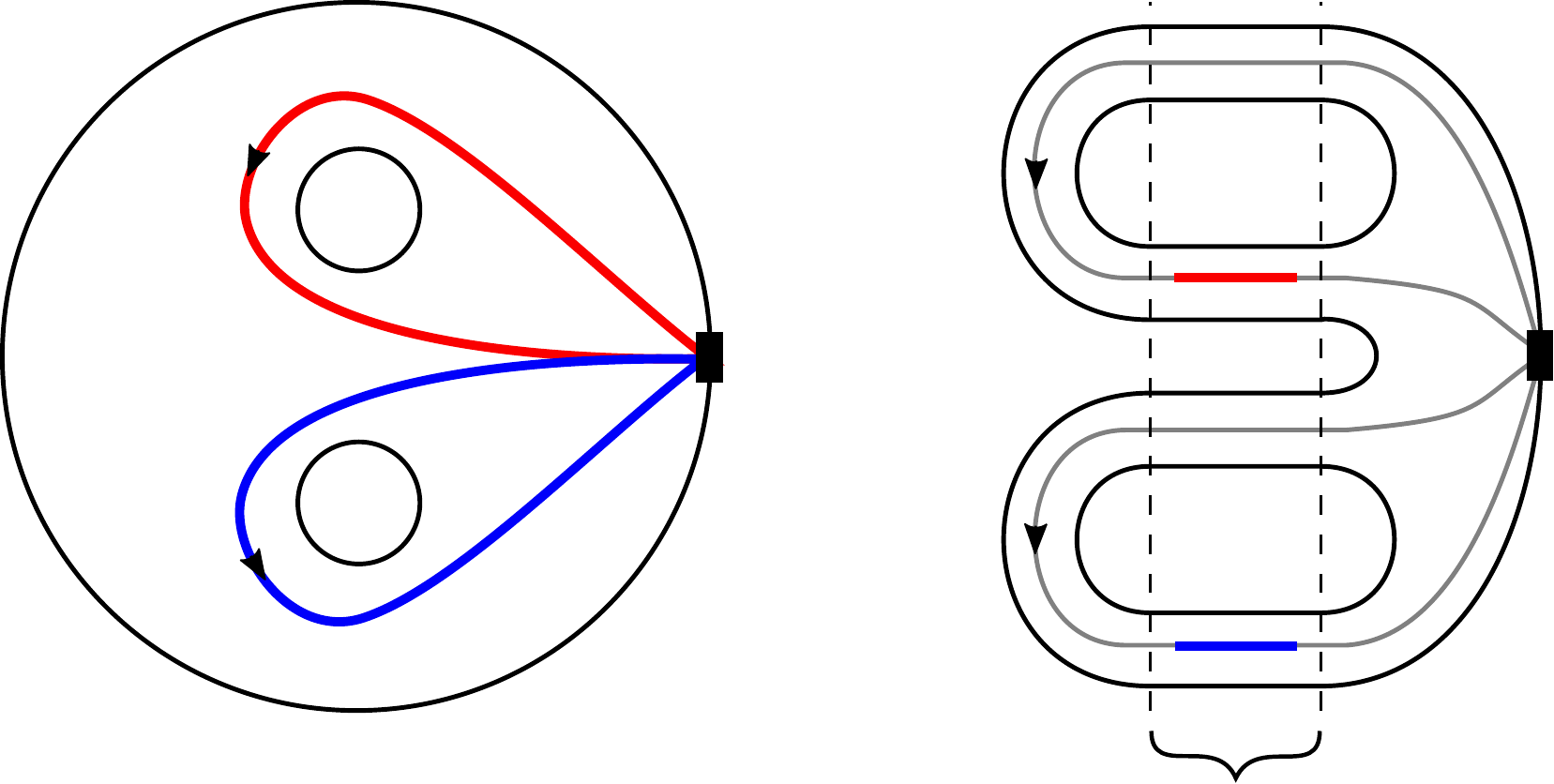}
		\put(52.5,25){$\cong$}
		\put(77.5,-7.5){$\Sigma_0$}
		\put(7.5,40){\textcolor{red}{$\gamma_{\g_2}$}}
		\put(7.5,15){\textcolor{blue}{$\gamma_{\g_1}$}}
	\end{overpic}  
	\vspace{0.5cm}
  \caption{Example: sphere with three punctures.}
  \label{fig:sphere}
\end{figure}
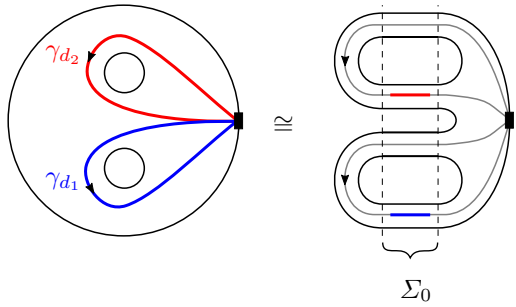

\begin{theorem}\label{Thm: value surface}
Let $\Sigma$ be a surfaces with at least one boundary component. Fix a principal $\G$-bundle $\varphi\colon \Sigma \longrightarrow B\G$ on $\Sigma$ and a corresponding $\G$-labeled gluing pattern $(P,\g_1, \dots , \g_n)$. There is an equivalence of categories
\begin{align}
\int_{(\Sigma,\varphi)} \mathcal{A} \cong a_P^{\g_1,\dots, \g_n} \operatorname{-mod}_{\mathcal{A}}
\end{align} 
\end{theorem}   

\begin{proof}
The following is an extension of the proof given in \cite[Theorem 5,14]{bzbj} to surfaces with $\G$-bundles. We have seen that for a $\g$-labeled disk $\mathbb{D}^\g_{\bullet\bullet}$ with two marked intervals we have $\int_{\mathbb{D}^\g_{\bullet\bullet}}\Aa \cong \Aa$ as plain categories, with the markings inducing the structure of a right $\Aa^{\boxtimes 2}$-module category with module structure given by the twisted regular action $\reg^{\g}$. Now, $\int_{\sqcup_i \mathbb{D}^{\g_i}_{\bullet\bullet}} \Aa \cong \Aa^{\boxtimes n}$ has the structure of a right $\Aa^{\boxtimes 2n}$-module category. Indeed, using the decorated gluing pattern $(P,\g_1,\dots,\g_n)$ we have an action:
$$
\reg_P^{\g_1,\dots,\g_n} \colon (x_1 \boxtimes \dots \boxtimes x_n) \boxtimes (y_1 \boxtimes \dots \boxtimes y_{2n}) \mapsto (x_1 \otimes y_{P(1)} \otimes \vartheta(\g_1).y_{P(1')}) \boxtimes \dots \boxtimes (x_n \otimes y_{P(n)} \otimes \vartheta(\g_n).y_{P(n')})
$$
We denote the resulting right module category by $\Ma_{P}^{\g_1,\dots,\g_n}$.\par  
On the other hand, we have the disk ${_{\bullet^{2n}}}\mathbb{D}_\bullet$ with $2n$ marked intervals to the left and one marked interval to the right. This turns $\int_{{_{\bullet^{2n}}}\mathbb{D}_\bullet} \mathcal{A} \cong \mathcal{A}$ into a $(\Aa^{\boxtimes 2n},\Aa)$-bimodule via the iterated tensor product 
$$
(x_1 \boxtimes \dots \boxtimes x_{2n}) \boxtimes y \boxtimes z \mapsto x_1 \otimes \dots \otimes x_{2n} \otimes y \otimes z.
$$
We denote the resulting bimodule category by ${_{\Aa^{\boxtimes 2n}}}\Aa_\Aa$. Using excision, we then have 
$$
\int_{(\Sigma, \varphi)}\Aa \cong \Ma_P^{\g_1,\dots,\g_n} \underset{\Aa^{\boxtimes 2n}}{\boxtimes} {_{\Aa^{\boxtimes 2n}}}\Aa_\Aa \ \ .
$$
Let $\tau_P \colon \{1, \dots ,2n\} \to \{1, \dots, 2n\}$ be the bijection given by postcomposing the map defined by $2k-1 \mapsto k$, $2k \mapsto k'$ with $P$. Notice that the inverse of this map is part of the action $\reg^{\g_1,\dots,\g_2}_P$. Applying monadic reconstruction as in Theorem \ref{thm:reconstruction}, together with Lemma \ref{lma:pg}, we can identify $\Ma_P^{\g_1,\dots,\g_n}$ with modules over an algebra $\underline{\End}_{\Aa^{\boxtimes 2n}}(1_\Aa)_P^{\g_1,\dots,\g_n} \in \Aa^{\boxtimes 2n}$, obtained from $\underline{\End}_{\Aa^{\boxtimes 2}}(1_\Aa)^{\g_1} \boxtimes \dots \boxtimes \underline{\End}_{\Aa^{\boxtimes 2}}(1_\Aa)^{\g_n}$ by acting with $\tau_P$. Applying Corollary \ref{crl:basechange} to the dominant tensor functor $T^{2n} \colon \Aa^{2n} \to \Aa$, we thus get 
$$
\int_{\Sigma}\Aa \cong T^{2n}(\underline{\End}_{\Aa^{\boxtimes 2n}}(1_\Aa)_P^{\g_1,\dots,\g_n})\operatorname{-mod}_\Aa \ \,
$$
as right $\Aa$-module categories. \par 
Let us write $T^{2n}(\underline{\End}_{\Aa^{\boxtimes 2n}}(1_\Aa)_P^{\g_1,\dots,\g_n}) = \widetilde{a}_P$ for brevity. To finish the proof, we want to show that there is an isomorphism of algebras $\widetilde{a}_P \cong a_P^{\g_1, \dots, \g_n}$. Consider the subalgebras $$\Fa^{(i,i')}_\Aa \coloneqq \underline{\End}_{\Aa_{P(i)} \boxtimes \Aa_{P(i')}}(1_\Aa)^{\g_i} \in \Aa^{\boxtimes 2n}$$ and their images under the tensor functor $\Fa_\Aa^{(i)} \coloneqq T^{2n}(\Fa_\Aa^{(i,i')}) \in \Aa$. By embedding each $\Fa_\Aa^{(i)}$ into $\widetilde{a}_P$ we get a map 
$$
\widetilde{m}_P \colon \Fa_\Aa^{(1)} \otimes \dots \otimes \Fa_\Aa^{(n)} \hookrightarrow \widetilde{a}_P^{\otimes n} \xrightarrow{\widetilde{m}} \widetilde{a}_P \ \ ,
$$
where $\widetilde{m}$ is the multiplication in $\widetilde{a}_P$. This map establishes the isomorphism on the level of objects in $\mathcal{A}$. The restriction of the multiplication to the image of one of the $\Fa_\Aa^{(i)}$ agrees with the multiplication $m$ in $\Fa_\Aa^{\g_i}$. So it is left to show that for each pair of indices $1 \leq i < j \leq n$ the composition
$$
\Fa_\Aa^{(i)} \otimes \Fa_\Aa^{(j)} \otimes \Fa_\Aa^{(i)} \otimes \Fa_\Aa^{(j)} \xrightarrow{\id \otimes C \otimes \id} \Fa_\Aa^{(i)} \otimes \Fa_\Aa^{(i)} \otimes \Fa_\Aa^{(j)} \otimes \Fa_\Aa^{(j)} \xrightarrow{m \otimes m} \Fa_\Aa^{(i)} \otimes \Fa_\Aa^{(j)} \xrightarrow{\widetilde{m}_P} \widetilde{a}_P,
$$
for $C$ being $L^\pm, N^\pm$ or $U^\pm$, agrees with $\widetilde{m}_P|_{(\mathcal{F}_{\Aa}^{(i)} \otimes \mathcal{F}_{\Aa}^{(j)})^{\otimes 2}}$. To that end, consider the following diagram
\begin{center}
\begin{tikzcd}
& T^4(\Fa_\Aa^{(i,i')} \otimes \Fa_\Aa^{(j,j')}) = T^4(\Fa_\Aa^{(j,j')} \otimes \Fa_\Aa^{(i,i')}) \arrow[dd,"T^4(m)"] \\
\Fa_\Aa^{(i)} \otimes \Fa_\Aa^{(j)} \arrow[ur,dashed,"J_{i,j}"] \arrow[dr,"\widetilde{m}_P",swap] && \Fa_\Aa^{(j)} \otimes \Fa_\Aa^{(i)} \arrow[dl,"\widetilde{m}_P"] \arrow[ul,dashed,"J_{j,i}",swap] \\
& \widetilde{a}_P
\end{tikzcd}
\end{center}
where the label $T^4(m)$ on the vertical arrow means applying the tensor functor to the multiplication in $\underline{\End}_{\Aa^{\boxtimes 2n}}(1_\Aa)_P^{\g_1, \dots, \g_n}$. The dashed arrows, making the above diagram commute, can be described by exhibiting the tensor structure of the iterated tensor product functor
$$
J_{i,j} \colon \Fa_\Aa^{(i)} \otimes \Fa_\Aa^{(j)} = T^4(\Fa_\Aa^{(i,i')}) \otimes T^4(\Fa_\Aa^{(j,j')}) \xrightarrow{\cong} T^4(\Fa_\Aa^{(i,i')} \otimes \Fa_\Aa^{(j,j')})
$$
given by the shuffle braiding\footnote{The shuffle braiding $
J \colon a_1 \otimes \dots \otimes a_n \otimes b_1 \otimes \dots \otimes b_n \xrightarrow{\cong} a_1 \otimes b_1 \otimes \dots \otimes a_n \otimes b_n$
is given by $J = \sigma_{a_n,b_{n-1}} \circ \dots \circ \sigma_{a_3 \otimes \dots \otimes a_n,b_2} \circ \sigma_{a_2 \otimes \dots \otimes a_n,b_1}$, where $\sigma$ is the braiding of $\Aa$.}. As an example, consider the gluing pattern $P(1,1',2,2') = (1,3,4,2)$ describing positively nested handles. The corresponding shuffle braiding is 
$$
J_{1,2} = (1 \otimes 1 \otimes \sigma) \circ (1 \otimes \sigma \otimes 1), \quad J_{2,1} = (\sigma \otimes 1 \otimes 1) \circ (1 \otimes \sigma \otimes 1),
$$
and we observe that the composition $J^{-1}_{1,2} \circ J_{2,1}$ agrees with the nested crossing morphism $N_{1,2}^{+} \colon \Fa_\Aa^{\g_2} \otimes \Fa_\Aa^{\g_1} \to \Fa_\Aa^{\g_1} \otimes \Fa_\Aa^{\g_2}$. From commutativity of the above diagram, we then get that $\widetilde{m}_P|_{\Fa^{\g_2}_\Aa \otimes \Fa_\Aa^{\g_1}} = \widetilde{m}_P|_{\Fa_\Aa^{\g_1} \otimes \Fa_\Aa^{\g_2}} \circ N_{1,2}^+$, which finishes the proof for the positively nested case. The other five cases can be worked out analogously.
\end{proof}

\subsection{Little bundles algebras and braided $\G$-crossed categories}\label{Sec: LB}

The value of oriented factorisation homology of a rigid balanced braided category $\Aa$ on $\mathbb{S}^1 \times \R $ is given by the Drinfeld centre $\mathcal{Z}(\Aa)$ of $\Aa$. In \cite[Remark 3.2]{bzbj2} it is observed that $\int_{\mathbb{S}^1\times \R} \Aa $ carries two natural monoidal structures induced from the
topology of genus zero surfaces; one is induced by stacking annuli in the $\R$-direction, which we will denote $\otimes_{\R}$, and the other one is induced by embedding annuli into the pair of pants and will be denoted $\otimes_{\text{Pants}}$. The monoidal structure coming from the pair of pants requires some explanation: Evaluating factorisation homology on the pair of embeddings sketched in Figure \ref{fig:pairofpants} gives rise to the cospan 
\begin{align}
\int_{\mathbb{S}^1\times \R} \Aa \boxtimes \int_{\mathbb{S}^1\times \R} \Aa \xrightarrow{(\iota_1 \sqcup \iota_2)_*} \int_{\text{Pants}} \Aa \xleftarrow{{\iota_{\text{out}}}_*} \int_{\mathbb{S}^1 \times \R} \Aa
\end{align}  
in $\Prc$. Using the right adjoint\footnote{Note that the right adjoint $\iota_{\text{out}}^*$ is again in $\Prc$ since $\iota_{\text{out}}$ is given by acting on the distinguished object in $\int_{\text{Pants}} \Aa$ which is a progenerator.} $\iota_{\text{out}}^*$ to ${\iota_{\text{out}}}_*$ we get 
an induced tensor product $\otimes_{\text{Pants}}$, which agrees with the usual tensor product on the Drinfeld centre. 
We refer to~\cite{Wasserman} for a detailed 
algebraic discussion of the type of interaction we expect between these two monoidal structures in the case of fusion categories.

\begin{figure}[h!]
\begin{center}
\vspace{0.25cm}
	\begin{overpic}[scale=0.3
		,tics=10]
		{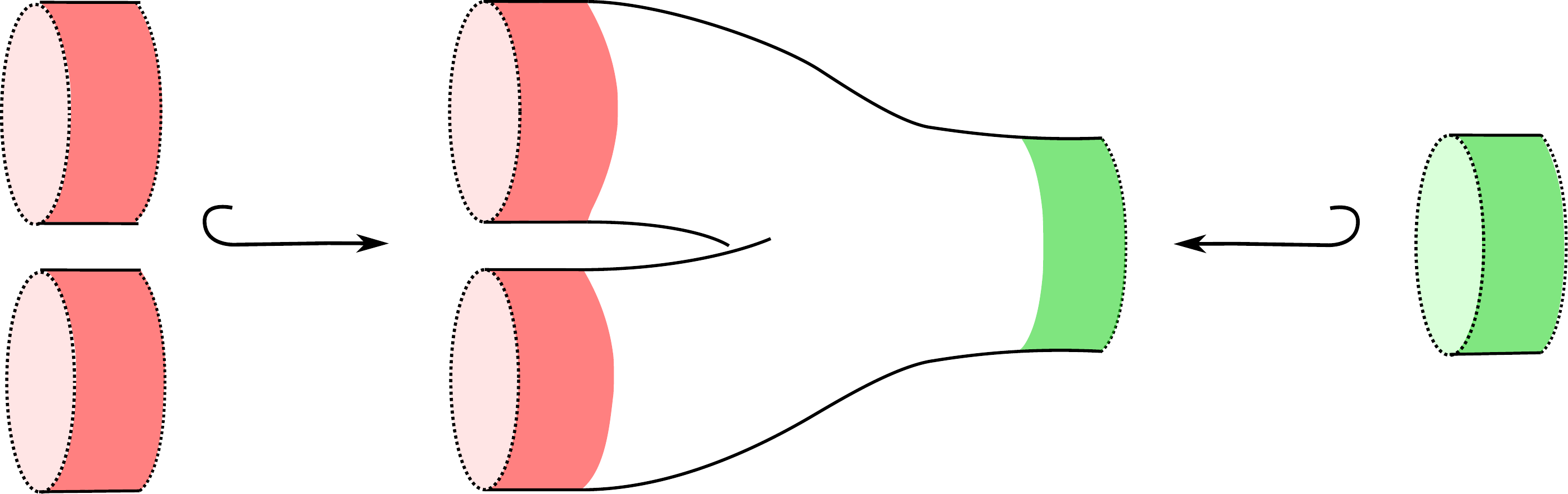}
		\put(13,10){$\iota_1 \sqcup \iota_2$}
		\put(77.5,10){$\iota_{\text{out}}$}
	\end{overpic}  
	\vspace{0.5cm}
	\caption{The maps inducing the monoidal structure $\otimes_{\text{Pants}}$.}
\label{fig:pairofpants}
\end{center}
\end{figure}

For the case of interest in the present work, i.e.~in the case that $\Aa$ is equipped with a $\G$-action, the situation is slightly different since the annulus $\mathbb{S}^1\times \R$ can be endowed with different maps into $B\G$. We can assume that, up to homotopy, every map $\varphi \colon \mathbb{S}^1\times \R \longrightarrow B\G$ is constant in the $\R$-direction and hence we still find an $\mathsf{E}_1$-algebra structure $\otimes_\R$ on $\int_{(\mathbb{S}^1\times \R, \varphi)} \Aa $. On the other hand, the pair of pants only induces an $\mathsf{E}_2$-algebra structure in the case that all maps into $B\G$ are chosen to be constant. The non-constant maps into $B\G$ induce instead another interesting algebraic structure on the collection of values taken by factorisation homology on all possible maps $\varphi \colon \mathbb{S}^1\times \R \longrightarrow B\G$, a little $\G$-bundles algebra~\cite{littlebundles}. \par
The operad $\mathsf{E}_2^{\G}$ of little $\G$-bundles is coloured over the space of maps from $\mathbb{S}^1$ to $B\G$. To describe the space of operations we need to introduce some notation: For a disk embedding $f\in \mathsf{E}_2(r)$ we denote by $\comp(f)$ the complement of the interior of all embedded disks.
Let $\underline{\varphi}=(\varphi_1,\dots,\varphi_r)$ be an $r$-tuple of maps $\varphi_i \colon \mathbb{S}^{1} \to B\G$ and $\psi \colon \mathbb{S}^{1} \to B\G$ another map. The space of operations $\mathsf{E}_2^{\G} \binom{\psi}{ \underline{\varphi}}$ consists of pairs of an element $f \in \mathsf{E}_2(r)$ together with a map $\xi \colon \comp(f) \to B\G$ whose restriction to $\partial \comp(f)$ is given by $(\underline{\varphi},\psi)$. By construction we have the following:

\begin{proposition}
The value of factorisation homology on $\mathbb{S}^1$ equipped with varying $\G$-bundle decorations has the structure of a little $\G$-bundles algebra. 
\end{proposition} 

The main result of~\cite[Theorem 4.13]{littlebundles} identifies algebras over $\mathsf{E}^{\G}_2$ inside the 2-category $\Cat$ of categories with braided $\G$-crossed categories as defined by Turaev~\cite{turaevgcrossed,turaevhqft} and recalled below. The proof directly carries over to $\Prc$.

\begin{definition}
A \emph{braided $\G$-crossed category} is a $\G$-graded monoidal category 
$$ 
\Aa^{\G} = \bigoplus_{\g \in \G} \Aa_{\g}, \quad \text{such that }  \otimes \colon \mathcal{A}_{\g} \boxtimes \mathcal{A}_{\g'} \to \mathcal{A}_{\g\g'}
$$
together with a $\G$-action on $\Aa^{\G}$, which is such that the image of the action by an element $h \in \G$ on $\Aa_{\g}$ is contained in $\Aa_{h \g h^{-1}}$, and natural isomorphisms $c_{X,Y}\colon X\otimes Y \longrightarrow \g.Y\otimes X$ for $X \in \Aa_{\g}$, satisfying natural coherence conditions~\cite{Coherence}.
\end{definition} 

We call the braided $\G$-crossed category assigned to $\Aa$ by factorisation homology the \emph{$\G$-centre $\mathcal{Z}^{\G}(\Aa)$ of $\Aa$}. The $\g$-components $\mathcal{Z}^{\G}_{\g}(\Aa)$ are given by factorisation homology on $\varphi_{\g} \colon \mathbb{S}^1\times \R \longrightarrow B\G$, where $\varphi_{\g}$ corresponds to the loop $\g \in \pi_{1}(B\G)=\G$ and is constant in the $\R$-direction. To compute the $\G$-centre explicitly, we recall the concept of bimodule traces and twisted centres from~\cite{dsps,fss}. Let $\mathcal{A}\in \Prc$ be a monoidal category and $\mathcal{M}$ be an $\mathcal{A}$-bimodule category. The bimodule trace of $\mathcal{M}$ is 
\begin{align}
\Tr_{\Aa}(\Ma) \coloneqq \mathcal{M} \boxtimes_{\Aa \boxtimes \Aa^{\operatorname{rev}}} \Aa \ \ ,
\end{align}
where $\Aa^{\operatorname{rev}}$ denotes the category $\Aa$ with the reverse multiplication. Assume now that $\mathcal{F} \colon \Aa \to \Aa$ is a monoidal functor and denote by ${_{\langle \mathcal{F} \rangle}}\mathcal{M}$ the $(\Aa, \Aa)$-bimodule whose left action is pulled back along $\mathcal{F}$. Similarly, we will denote $\mathcal{M}_{\langle \mathcal{F} \rangle}$ the bimodule whose right action is pulled back along $\mathcal{F}$. The $\mathcal{F}$-twisted centre $\mathcal{Z}^\mathcal{F}(\Ma)$ is then the Drinfeld centre of the bimodule category $\Ma_{\langle \mathcal{F} \rangle}$.

\begin{proposition}
Let $\g$ be an element of $\G$. There is a natural isomorphism 
\begin{align}
\mathcal{Z}_{\g}^{\G}(\Aa) \cong  \Tr_{\Aa}(\Ma_{\g}) \ \ ,
\end{align}
where $\Ma_{\g}$ is the bimodule constructed in Section~\ref{Sec: G-Reconstruction} via the twisted regular action. Moreover, one can identify the bimodule trace $\Tr_{\Aa}(\Ma_d)$ with the twisted Drinfeld centre $\mathcal{Z}^{\vartheta(\g^{-1})}(\Aa)$.
\end{proposition}

\begin{proof}
The first statement follows directly from applying excision to the cover sketched in Figure~\ref{fig:annulus} combined with the results of Section~\ref{sec: excision D-mfd}. Note that 
here excision is not used as in the proof of Theorem~\ref{Thm: value surface}.
\begin{figure}[H]
 \centering
  \vspace{0.25cm}
  \begin{overpic}[scale=0.4
		,tics=10]
		{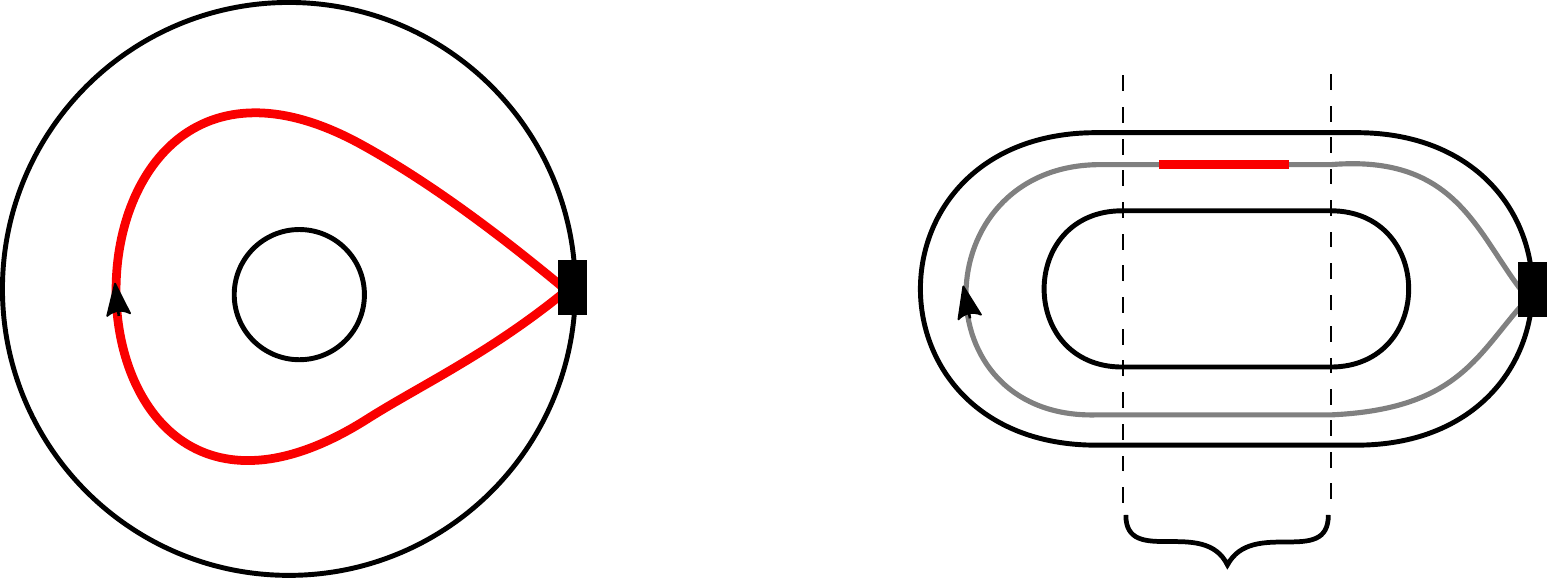}
		\put(47.5,17.5){$\cong$}
		\put(77.5,-7.5){$\Sigma_0$}
	\end{overpic}  
	\vspace{0.5cm}
  \caption{Collar-gluing for the annulus with a map to $B\G$.}
  \label{fig:annulus}
\end{figure}
For the second statement, recall that since $\Aa$ is rigid we can apply Theorem \ref{thm:reconstruction} to identify $\Ma_{\g} \cong \underline{\End}(1_\Aa)^{\vartheta(\g)}\operatorname{-mod}_{\Aa \boxtimes \Aa^{\rev}}$, where $\underline{\End}(1_\Aa)^{\vartheta(\g)}$ is the endomorphism algebra of $1_\Aa$ in $\Aa \boxtimes \Aa^{\rev}$ with respect to the $\vartheta(\g)$-twisted regular action. A categorical version of the Eilenberg-Watts theorem \cite[Lemma 5.7]{dualizabilityBrTens} then gives an equivalence 
$$
\Tr_\Aa(\Ma_{\g}) \underset{\text{Thm. }\eqref{thm:reconstructionreltensorprod}}{\cong} \underline{\End}(1_\Aa)^{\vartheta(\g)}\text{-mod}_{\Aa} \cong \Hom_{\Aa \boxtimes \Aa^{\rev}}( {_{\langle \vartheta(\g^{-1})\rangle}}\Aa, \Aa) \ \ .
$$
But, by \cite[Lemma 2.13]{fss} this is precisely the $\vartheta(\g)$-twisted Drinfeld centre of $\Aa$ as claimed.
\end{proof}

Let us introduce the following bimodule category $\Aa \rtimes \G \coloneqq \bigoplus_{\g \in \G} \mathcal{M}_{\g}$. This category has the structure of a $\G$-graded monoidal category via 
\begin{align}
\otimes_{\Aa \rtimes \G}\colon \mathcal{M}_{\g}\boxtimes \mathcal{M}_{\g'} & \longrightarrow \mathcal{M}_{\g\g'} \\ 
x \boxtimes x' & \longmapsto x \otimes \vartheta(\g).x'
\end{align}
as indicated by the notation. 

\begin{corollary}
The trace $\Tr_{\Aa}(\Aa \rtimes \G)$ of the bimodule $\Aa \rtimes \G$ agrees with the $\G$-centre $\mathcal{Z}^{\G}(\Aa)$ and is a braided $\G$-crossed category. 
\end{corollary}

\begin{remark}
In~\cite{centerofgradedfusioncategories}, the \emph{graded centre} of a $\G$-graded fusion category $\mathcal{C} = \bigoplus_{\g \in \G} \mathcal{C}_{\g}$ is defined to be $\mathcal{Z}_{\mathcal{C}_e} (\mathcal{C}) \cong \Tr_{\mathcal{C}_e}(\mathcal{C})$ and equipped with the structure of a braided $\G$-crossed category. In the case that $\Aa$ is a braided fusion category with $\G$-action, the $\G$-centre $\mathcal{Z}^{\G}(\Aa)$ agrees with the graded centre of $\Aa \rtimes \G$. A careful comparison of the two little bundles algebra structures would take us too far from the content of the paper.
\end{remark}

\begin{remark}
We also leave a detailed study of the interaction of the monoidal structure $\otimes_{\R}$ induced by stacking annuli in the $\R$-direction with the $\G$-crossed braided structure as an interesting open question for further research.  
\end{remark}     

\subsection{Algebraic description of boundary conditions and point defects}\label{Sec: Defects}

In Section~\ref{Sec: PD and BC} we explained that boundary conditions and point defects for $\G \times SO(2)$-structured 
factorisation homology with values in $\Prc$ are classified by symmetric monoidal functors from the categories of stratified 
disks $D\text{-}\Disk_{2,\partial}$ and $D\text{-}\Disk_{2, *}$ to $\Prc$, respectively. In this section we will describe the algebraic structure classifying these functors. Our strategy will be the following: The 
source categories can naturally be identified with the envelope of the coloured operads $\G\text{-}\mathsf{fSC}$, a framed and $\G$-equivariant version of the Swiss-cheese operad~\cite{Swiss},
and $\G\text{-}\mathsf{fE_2^1}$, a framed and $\G$-equivariant $\mathsf{E}_2$-operad with a frozen strand~\cite{MopCA}, respectively. Hence, defect data corresponds to algebras over them. Both operads are aspherical, meaning that all the homotopy groups of the 
operation spaces vanish in degree higher than 1. For this reason we can work equivalently with the groupoid 
valued operads $\Pi_1(\G\text{-}\mathsf{fSC})$ and $\Pi_1(\G\text{-}\mathsf{fE_2^1})$, instead of
topological operads. We extend existing 
combinatorial models~\cite{Swiss-CheeseNI,MopCA} in terms of generators and relations to the situation at 
hand. The results will be 
combinatorially described groupoid valued coloured operads $\G\text{-}\mathsf{fPeBr}$ and $\G\text{-}\mathsf{fBr^1}$ equivalent to $\Pi_1(\G\text{-}\mathsf{fSC})$ and $\Pi_1(\G\text{-}\mathsf{fE_2^1})$. \par  
We will work within the 2-categorical framework for operads, see 
for example Section 2 of~\cite{cfE2}. The advantage of this is that all structures will automatically be coherent in the appropriate sense. 
Alternatively, one could work with $\Sigma$-cofibrant models, similar to the parenthesised braid model for the $\mathsf{E}_2$-operad~\cite[Chapter 6]{FresseI}, at the categorical level~\cite{MopCA,Swiss-CheeseNI}.  

\subsubsection{Boundary conditions}

We briefly recall the situation without principal bundles~\cite{bzbj2}. The category $\Disk^{\operatorname{fr}}_{2,\partial}$ is 
equivalent to the envelope of the topological Swiss-cheese operad $\SC$ with its two colours $\mathbb{D}$ and $\mathbb{D}_\partial$, corresponding to the standard disk and the half-disk. The spaces of operations are given by rectilinear embeddings. In particular, one has that  
$$
\SC(\underbrace{\mathbb{D},\dots, \mathbb{D}}_{n}; \mathbb{D}) = \mathsf{E}_2(n), \quad \SC(\underbrace{\mathbb{D}_\partial,\dots,\mathbb{D}_\partial}_{n}; \mathbb{D}_\partial) = \mathsf{E}_1(n) \ \ .
$$ 
In Figure~\ref{fig: EG SC} we sketch an operation with different 
colours and in Figure~\ref{fig: Generators SC} we list the generators\footnote{We refer \cite[Section 4.1]{littlebundles} for more details on generators and relations for groupoid valued operads.} 
for the corresponding combinatorial model $\mathsf{PeBr}$ of permutations and braids, constructed in $\cite{Swiss-CheeseNI}$, together with the respective topological operations. 

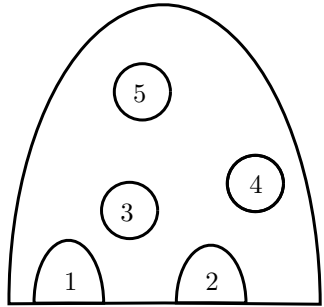
\begin{figure}[h]
	\begin{center}
		\vspace{0.25cm}
		\begin{overpic}[scale=0.5
			,tics=10]
			{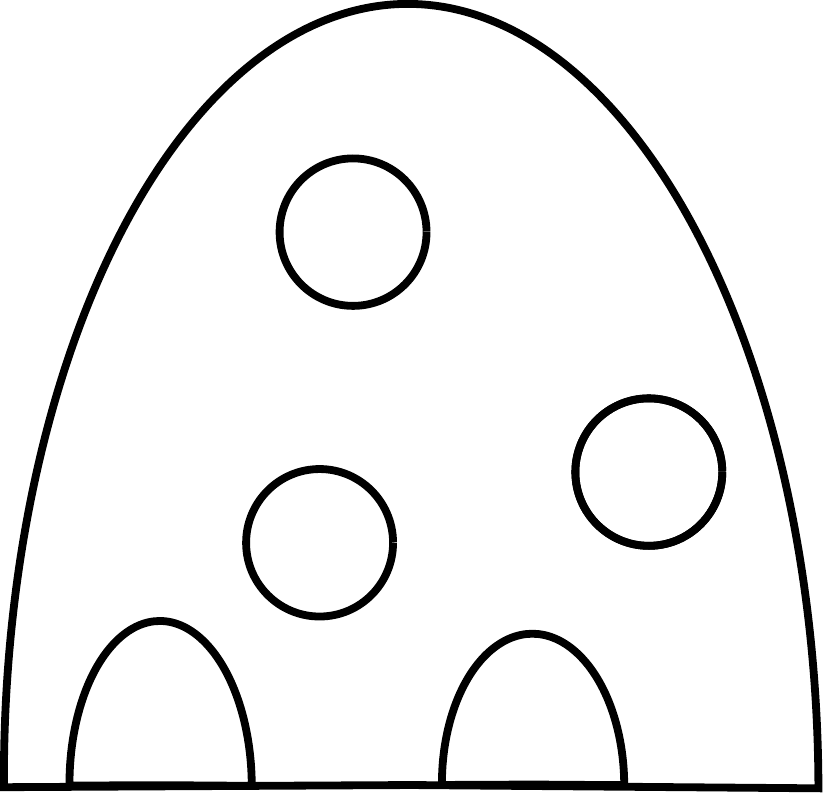}
			\put(18,5){$1$}
			\put(63,5){$2$}
			\put(36,27){$3$}
			\put(77,36){$4$}
			\put(40,65){$5$}
		\end{overpic}  
		\vspace{0.5cm}
		\caption{An example of an operation in $\SC(\mathbb{D}_\partial,\mathbb{D}_\partial,\mathbb{D}, \mathbb{D},\mathbb{D}; \mathbb{D}_\partial)$.}
		\label{fig: EG SC}
	\end{center}
\end{figure}

The relations for $\mathsf{PeBr}$ are such that an algebra over $\SC$ corresponds to a braided monoidal category $\mathcal{A}$, a monoidal category $\mathcal{C}$ and a braided functor $\mathcal{A}\longrightarrow \mathcal{Z}(\mathcal{C})$ into the Drinfeld centre of $\mathcal{C}$.
For a complementary physical perspective on the correspondence between boundary conditions and 
maps into $\mathcal{Z}(\mathcal{C})$ we refer the reader to~\cite{FSV}.   
\begin{figure}[t]
	\begin{center}
		\vspace{0.25cm}
		\begin{overpic}[scale=0.5
			,tics=10]
			{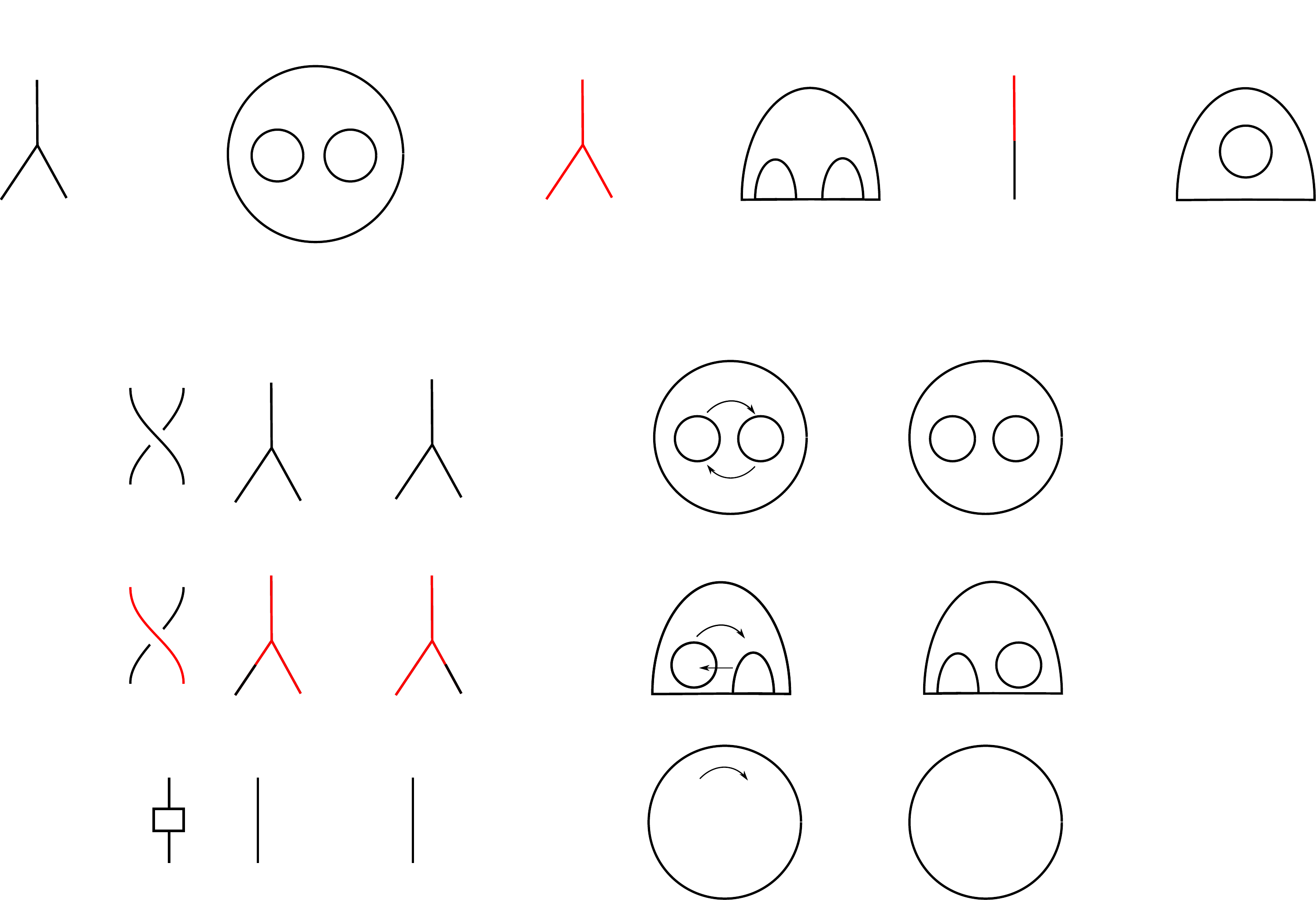}
			\put(2,67.5){$\text{Generating objects:}$}
			\put(2,43){$\text{Generating morphisms:}$}
			\put(9,56){\large$\longmapsto$}
			\put(33,52){\large$,$}
			\put(49,56){\large$\longmapsto$}
			\put(68,52){\large$,$}
			\put(17,28){$1$}
			\put(22.5,28){$2$}
			\put(29.5,28){$2$}
			\put(35,28){$1$}
			\put(17,13){$1$}
			\put(22.5,13){$2$}
			\put(29.5,13){$2$}
			\put(35,13){$1$}
			\put(81,56){\large$\longmapsto$}
			\put(15,34){\large$\colon$}
			\put(24,34){\large$\longrightarrow$}
			\put(40,34){\large$\longmapsto$}
			\put(63,34){\large$\longrightarrow$}
			\put(15,20){\large$\colon$}
			\put(24,20){\large$\longrightarrow$}
			\put(40,20){\large$\longmapsto$}
			\put(63,20){\large$\longrightarrow$}
			\put(15,5){\large$\colon$}
			\put(24,5){\large$\longrightarrow$}			
			\put(40,5){\large$\longmapsto$}
			\put(63,5){\large$\longrightarrow$}
		\end{overpic}  
		\vspace{0.5cm}
		\caption{Generating operations for $\mathsf{fPeBr}$ and their image under the equivalence $\mathsf{fPeBr}\xrightarrow{\cong} \Pi_1(\mathsf{fSC})$. The arrows indicate the paths in the space of embeddings. If we ignore the last generating morphism, we recover the generators of $\mathsf{PeBr}$.}
		\label{fig: Generators SC}
	\end{center}
\end{figure}

To study boundary conditions for oriented manifolds, one works with the framed Swiss-cheese operad $\fSC$ where embeddings are allowed to rotate the disks $\mathbb{D}$. In the respective combinatorial model $\mathsf{fPeBr}$ for the framed Swiss-cheese operad this is incorporated by introducing one additional generator in $\mathsf{fPeBr}(\mathbb{D};\mathbb{D})$, the balancing, and imposing the relation corresponding to Equation~\eqref{Eq: balancing} inside $\mathsf{fPeBr}(\mathbb{D},\mathbb{D};\mathbb{D})$, see also Figure~\ref{fig: Generators SC}. Hence, we see that in order to extend an algebra $(\mathcal{A},\mathcal{C})$ over $\SC$ to an algebra over $\fSC$, we need to equip $\mathcal{A}$ with a balancing. \par 
Finally, we turn our attention to the $\G$-equivariant version $\G\text{-}\fSC$ of the framed Swiss-Cheese operad, together with its combinatorial model $\G\text{-}\mathsf{fPeBr}$, whose envelope is equivalent to $\G\text{-}\Disk_{2,\partial}$. We can assume without loss of generality that all bundles are trivial and hence the colours of the operads do not change. However, for every group element $\g \in \G$, we get an additional arity one operation in both $\G\text{-}\mathsf{fPeBr}(\mathbb{D};\mathbb{D})$ and $\G\text{-}\mathsf{PeBr}(\mathbb{D}_\partial;\mathbb{D}_\partial)$ corresponding to gauge transformations of the trivial bundle, which `commute' with all the other generators. Hence, we can identify $\G\text{-}\mathsf{fPeBr}$ with the Boardman-Vogt tensor product $\mathsf{fPeBr}\otimes_{\text{BV}} \G $, where we consider the group $\G$ as an operad concentrated in arity
one. On the level of algebras this implies $\Alg(\G\text{-}\fSC; \Prc) \cong \Alg(\mathsf{fPeBr}; \Alg(\G;\Prc)) $. But a $\G$-algebra is just an object of $\Prc$ equipped with a $D$-action, and so we can summarise our discussion in the following proposition. 

\begin{proposition}\label{prp:D-bdrycondition}
Let $\mathcal{A}$ be a balanced braided category with $\G$-action. Boundary conditions for $\mathcal{A}$ in $\G \times SO(2)$-structured factorisation homology are given by a monoidal category $\mathcal{C} \in \Prc$ with $\G$-action and a $\G$-equivariant braided functor $\mathcal{A} \longrightarrow \mathcal{Z}(\mathcal{C})$ into the Drinfeld centre of $\mathcal{C}$ with its induced $\G$-action.   
\end{proposition} 

\begin{example}
The trivial boundary condition, corresponding to simply removing the boundary and computing factorisation homology on the resulting
manifold without boundary, is given by taking $\mathcal{C} = \mathcal{A}$ together with the canonical embedding $\mathcal{A} \longrightarrow \mathcal{Z}(\mathcal{A})$ induced by the braiding on $\mathcal{A}$. 
\end{example}

\begin{example}
The sources for boundary conditions from~\cite[Section 2.3]{bzbj2} have natural generalisations to the equivariant setting:

\begin{enumerate}
\item Let $\mathcal{A}$ be a balanced braided category with $\G$-action and denote by $E_2(\mathcal{A})$ the 
category of commutative algebras in $\mathcal{A}$, which comes with an induced $\G$-action. For every homotopy
fixed point\footnote{Here a \emph{homotopy fixed point} is a commutative algebra $a$ together with algebra isomorphisms $\tau_d \colon d.a \xrightarrow{\cong} a$ for all $d\in \G$ such that $\tau_{d'} \circ d'.\tau_d= \tau_{d'd}$.} $a \in E_2(\mathcal{A})^\G$, the category $a\operatorname{-mod}$ inherits a natural $\G$-action and
provides an example for boundary conditions of the bulk theory described by $\mathcal{A}$.
\item Consider the quantum Borel algebra $U_q(\mathfrak{b}) \hookrightarrow U_q(\mathfrak{g})$, which is the subalgebra generated by the elements $\{ K^{\pm}_{\alpha_i}, X^+_{\alpha_i} \}_{\alpha_i \in \Pi}$, following conventions from \cite[Section 9.1.B]{quantumgroups}. We get a forgetful tensor functor $\Rep_q(G) \to \Rep_q(B)$. Moreover, as noted in \cite[Section 2.3]{bzbj2}, the R-matrix provides a central structure on this forgetful functor. We observe that we have an $\Out(G)$-action on $U_q(\mathfrak{b})$, given on generators by $K^{\pm}_{\alpha_i} \mapsto K^{\pm}_{\kappa(\alpha_i)}$ and $H^+_{\alpha_i} \mapsto H^+_{\kappa(\alpha_i)}$ for any $\kappa \in \Out(G)$. We conclude that we get an $\Out(G)$-equivariant functor $\Rep_q(G) \to \mathcal{Z}(\Rep_q(B))$.
\end{enumerate}
\end{example}

\begin{remark}
There is another generalisation of the Swiss-cheese operad to the equivariant setting with operations consisting of an element in $\SC$ equipped with a map to $B\G$ on the complement
of the embedding. This is similar to the generalisation of the little disks operad given by the 
little bundles operad. We also expect this operad to play an important role in the description of boundary 
conditions for equivariant field theories. 
\end{remark}

\subsubsection{Point defects}

We again start by recalling the framed result from~\cite{bzbj2} in the language of coloured operads and then gradually build up to the oriented and $\G$-equivariant setting. The disk category $\Disk_{2, *}^{\operatorname{fr}}$ can be described
as the envelope of a topological operad with two colours, $\mathbb{D}$ and $\mathbb{D}_*$, corresponding to a disk and a marked disk, respectively. The spaces of operations are given by rectilinear embeddings which map marked points bijectively to marked points. The concrete structure of this coloured operad makes it into a \emph{moperad} as defined in~\cite[Definition 9]{Willwacher}. A combinatorial model for this topological operad is given  in~\cite{MopCA} in terms of parenthesised braids with a frozen strand. In Figure~\ref{fig: Generators fBr1}, we give a strict version of this combinatorial model, which will be denoted $\mathsf{Br}^1$. The description in terms of generators and relations allows us to read off the corresponding algebraic structure which was introduced in~\cite{BMod3, BMod1, BMod2}.

\begin{definition}\label{Def: Braided module}
Let $\mathcal{A}$ be a braided category. A \emph{braided module over $\mathcal{A}$} is a right module category 
$\triangleleft \colon \mathcal{M} \boxtimes \mathcal{A} \longrightarrow \mathcal{M}$
equipped with a natural isomorphism $E\colon \triangleleft \Longrightarrow \triangleleft$ satisfying (suppressing coherence isomorphisms)
\begin{align}\label{Eq: braided module}
E_{m\triangleleft x, y} &= (\id_{m} \triangleleft \sigma_{y,x}^{-1}) \circ (E_{m,y}\triangleleft \id_x ) \circ (\id_{m} \triangleleft \sigma_{x,y}^{-1}) \\ 
E_{m,x\otimes y}  &=  (E_{m,x}\triangleleft \id_y ) \circ E_{m\triangleleft x,y} \circ ( \id_m \triangleleft (\sigma_{y,x}\circ \sigma_{x,y}))
\end{align}
for all $m\in \mathcal{M}$ and $x,y\in \mathcal{A}$. 
\end{definition}

The framed version $\mathsf{fBr}^1$, giving a combinatorial model for the envelope of $\Disk_{2, *}$, can be described by an extension of $\mathsf{Br}^1$ obtained by adding two additional generating morphisms $\theta \in \mathsf{fBr}^1(\mathbb{D};\mathbb{D})$ and $\theta_* \in \mathsf{fBr}^1(\mathbb{D}_*,\mathbb{D}_*)$, corresponding to rotating the disks by $2\pi$. Furthermore, we need to include Relation~\eqref{Eq: balancing} for $\theta$ and Relation $(R4)$ from Figure~\ref{fig: Generators fBr1} for $\theta_*$. 
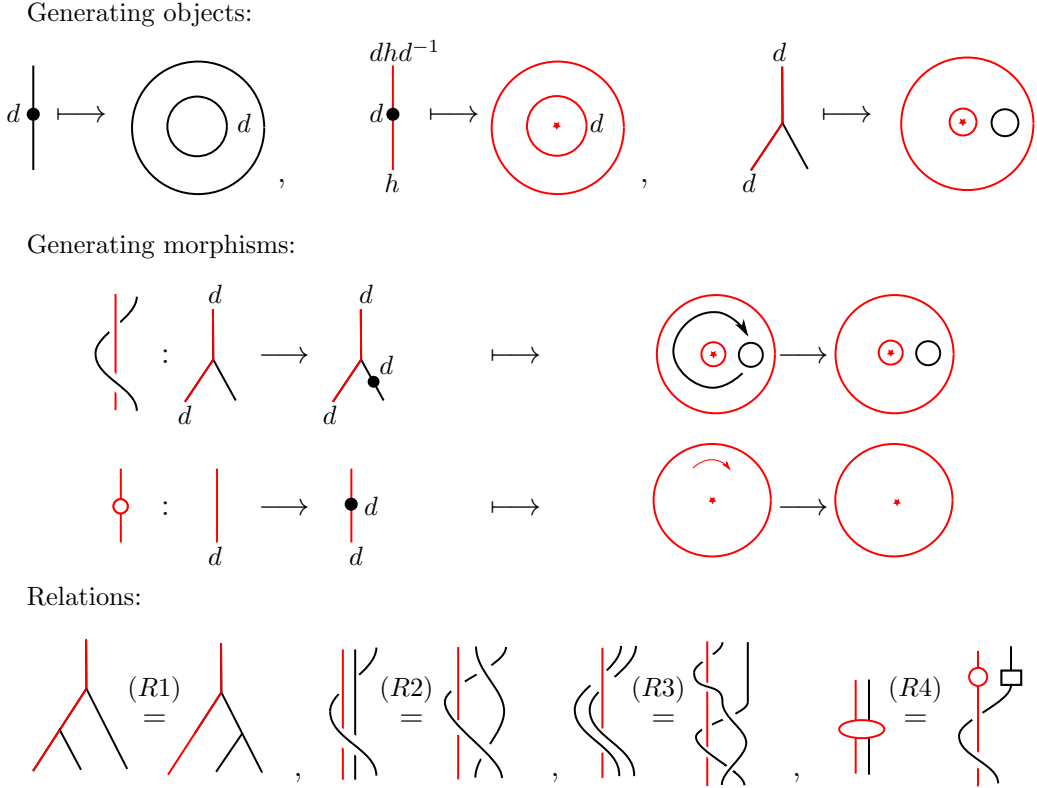
\begin{figure}[!t]
	\begin{center}
		\vspace{0.25cm}
		\begin{overpic}[scale=0.5, 
			,tics=10]
			{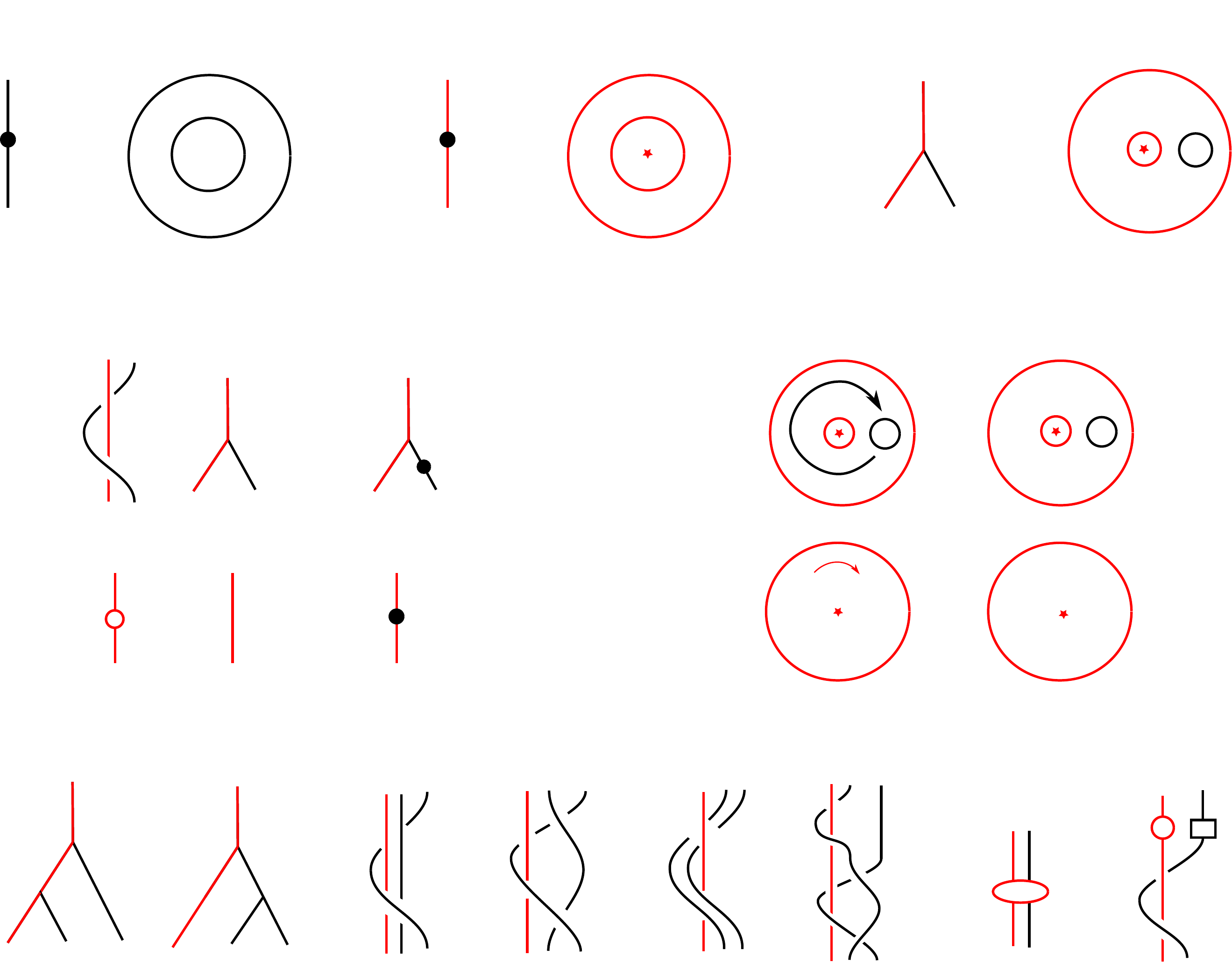}
			\put(0,76){$\text{Generating objects:}$}
			\put(0,53){$\text{Generating morphisms:}$}
			\put(0,18){$\text{Relations:}$}
			\put(-2,66){$d$}
			\put(2,66){ \large $\longmapsto$}
			\put(20,65){ $d$}
			\put(24,60){ \large $,$}
			\put(34,66){$d$}
			\put(34,72){$dhd^{-1}$}
			\put(35.5,59){$h$}
			\put(39,66){ \large $\longmapsto$}
			\put(55,65){ $d$}
			\put(60,60){ \large $,$}
			\put(71,59){$d$}
			\put(74,72){$d$}
			\put(78,66){ \large $\longmapsto$}
			\put(13,42){\large$\colon$}
			\put(13,27){\large$\colon$}
			\put(23,27){\large$\longrightarrow$}
			\put(23,42){\large$\longrightarrow$}
			\put(46,27){\large$\longmapsto$}
			\put(46,42){\large$\longmapsto$}
			\put(74.5,27){\large$\longrightarrow$}
			\put(74.5,42){\large$\longrightarrow$}
			\put(15,36){$d$}
			\put(29.4,36){$d$}
			\put(18,48){$d$}
			\put(33,48){$d$}
			\put(35,41){$d$}
			\put(18,22){$d$}
			\put(32,22){$d$}
			\put(33.5,27){$d$}
			\put(11.5,6){\large$=$}
			\put(37,6){\large$=$}
			\put(62,6){\large$=$}
			\put(87,6){\large$=$}
			\put(26.5,1){\large$,$}
			\put(52,1){\large$,$}
			\put(76,1){\large$,$}
			\put(10,9){$(R1)$}
			\put(35,9){$(R2)$}
			\put(60,9){$(R3)$}
			\put(85.5,9){$(R4)$}
		\end{overpic}  
		\vspace{0.5cm}
		\caption{Generating operations and relations for $\G\text{-}\mathsf{fBr^1}$ and their
image under the equivalence $\G\text{-}\mathsf{fBr^1} \xrightarrow{\cong} \Pi_1(\G\text{-}\mathsf{fE_2^1})$. Notice that we did not depict the relations related to the $\G$-action. The $\g$-labels on the disk for the first two generating objects mean that the map to $B\G$ is the loop $\g$ in radial direction. In $\DMan$ this embedding is isomorphic to the identity embedding equipped with the homotopy corresponding to $\g$. If we ignore the $\G$-labels, we get generators and relations of $\mathsf{fBr}^1$. If we furthermore drop the second generating morphism as well as relation $(R4)$, we get a combinatorial model for $\mathsf{E}^1_2$.}
		\label{fig: Generators fBr1}
	\end{center}
\end{figure}

We note that the system of relations is over-determined: To see this, note that Relation $(R4)$ allows one to rewrite $E_{\mathbb{D}_*,\mathbb{D}}$ in terms of the balancing $\theta$ and $\theta_*$. Inserting this into Relation~$(R3)$ in Figure~\ref{fig: Generators fBr1}, we find that this relation is automatically satisfied and hence obsolete. To show that the combinatorial description is correct, it is enough to note that the operation spaces in $\mathsf{fE}_2^1$ can be identified with the ones of $\mathsf{fE}_2$. 
Reading off the corresponding algebraic structure from the combinatorial model, one finds an equivalent reformulation of the braided balanced modules introduced in~\cite[Theorem 3.12]{bzbj2}. The only 
additional structure to the one described in Definition~\ref{Def: Braided module} is that of a balancing $\theta_{\mathcal{M}}
\colon \id_{\mathcal{M}} \Longrightarrow \id_{\mathcal{M}}$ on $\mathcal{M}$ compatible with $E$. \par 
Finally, we move on to describe point defects in the $D$-equivariant setting, which is slightly more subtle than the boundary conditions described in the previous section. The reason for this is that the disk with one marked point $\mathbb{D}_*$ is replaced by a collection of marked disks $\mathbb{D}_*^\g$ equipped with a map to $B\G$ with 
holonomy $d$. The combinatorial model for $\G\text{-}\mathsf{fE}_2^1$ can be derived from the model for the framed version of the 
little bundles operad given in~\cite[Section 5.4.2]{ThesisLW} similar to the derivation of the model for $\mathsf{fE}_2^1$
from the one for $\mathsf{fE}_2$. It is important to note here that we only consider configurations where the map to $B\G$ has 
non-trivial holonomy around the frozen strand. We list the generators and relations for the combinatorial model $\G\text{-}\mathsf{fBr}^1$ in Figure~\ref{fig: Generators fBr1}.
The corresponding algebraic notion is:
\begin{definition}\label{df:eqbalancedmodule}
Let $\mathcal{A}$ be a balanced braided category with $D$-action. An \emph{equivariant balanced right module over $\mathcal{A}$} is a $D$-graded
category $\mathcal{M}= \bigoplus_{\g \in \G} \mathcal{M}_{\g}$ equipped with 
\begin{itemize}
\item
a $D$-action $\act^\mathcal{M} \colon *\DS D \longrightarrow *\DS \Aut(\mathcal{M}) $ such that the image of $\mathcal{M}_{\g}$ under the action of $d'\in D$ is contained in $\mathcal{M}_{d' \g {d'}^{-1}}$,

\item an equivariant right $\mathcal{A}$-action $ \triangleleft \colon \mathcal{M} \boxtimes \mathcal{A} \longrightarrow \mathcal{M} \ \ ,$ 

\item natural isomorphisms $\theta_{\mathcal{M}}^d \colon \id_{\mathcal{M}_d} \longrightarrow \act^\mathcal{M}_d $ and 
$E^d \colon \triangleleft \longrightarrow \triangleleft \circ \left( \id_{\mathcal{M}_d}\boxtimes \act^\mathcal{A}_d \right) $ for all $d\in D \ \ .$

\end{itemize} 
such that (suppressing coherence isomorphisms)
\begin{itemize}
\item for all $m \in \mathcal{M}_d $ and $x,y \in \mathcal{A}$
\begin{align}
E^d_{m \triangleleft x, y} = \left( \id_{m} \triangleleft \sigma_{\act^\mathcal{A}_d(y),x}^{-1} \right) \circ \left( E^d_{m,y}\triangleleft \id_x \right) \circ \left(\id_{m} \triangleleft \sigma_{x,y}^{-1} \right) \ \ ,
\end{align} 
\item and for all $m \in \mathcal{M}_d $ and $x \in \mathcal{A}$
\begin{align}
\left(\theta_{\mathcal{M}}^d\right)_{m \triangleleft x} =  E^d_{\act^\mathcal{M}_d(m),x} \circ  \left( \left(\theta_{\mathcal{M}}^d\right)_{m}  \triangleleft (\theta_\mathcal{A})_x \right) \ \ .
\end{align}
\end{itemize}
\end{definition}  
 
We can summarise our discussion in the following proposition. 
\begin{proposition}\label{prp:D-mkdpoints}
Point defects in $D\times SO(2)$-structured factorisation homology are equivalent to equivariant balanced modules.
\end{proposition} 

\begin{example}
Let $\mathcal{C}$ be a boundary condition for a bulk theory $\mathcal{A}$. We can form a point defect from this boundary condition by removing a small 
circle around every marked point and inserting $\mathcal{C}$. On the algebraic level, the map from 
boundary conditions to point defects sends $\mathcal{C}$ to the $\G$-centre $\mathcal{Z}^\G(\mathcal{C})$ with the $\mathcal{A}$-action
induced by the functor $\mathcal{A} \longrightarrow \mathcal{Z}(\mathcal{C}) \subset \mathcal{Z}^\G(\mathcal{C})$. 
\end{example}    

\begin{remark}\label{Rem: braided modules}
In~\cite{bzbj2} a different approach to the description of point defects is taken: They are identified with modules over the value assigned to the annulus by factorisation homology equipped with the stacking tensor product. The same approach should work in the situation considered in this section, hence we expect that equivariant balanced modules 
over $\mathcal{A}$ can equivalently be described by graded modules over the graded centre $\mathcal{Z}^D(\mathcal{A})$ equipped with the stacking tensor product.   
\end{remark}

\begin{example}
Here we set $\G = \Out(G)$. For each element $\kappa \in \Out(G)$, let $h \in G$ act via $\kappa$-twisted conjugation $\Ad^\kappa_h(g) = hg\kappa(h^{-1})$ on $G$. Denote $C^\kappa \subset G$ the orbits of this action, i.e.~the $\kappa$-twisted conjugacy classes of $G$. For each $\kappa$-component of the $\Out(G)$-centre of $\Rep(G)$, we thus get a tensor functor 
$$
\int_{(\mathbb{S}^1,\kappa)} \Rep(G) \cong \QCoh(G/G) \to \QCoh(C^\kappa/G)
$$
where the $G$ acts by $\kappa$-twisted conjugation.
\end{example}

\subsubsection{Closed surfaces and marked points}\label{Sec:ClsdMfd}

We first compute the value of factorisation homology on a closed, unmarked surface $\Sigma$ equipped with a map $\varphi \colon \Sigma \longrightarrow B\G$. We use a decomposition of $\Sigma$ into a surface 
$\Sigma_o$ with one boundary component and a disk $\mathbb{D}$, see Figure~\ref{fig:clSurf}.

\begin{figure}[h!]
\begin{center}
	\begin{overpic}[scale=0.5
		,tics=10]
		{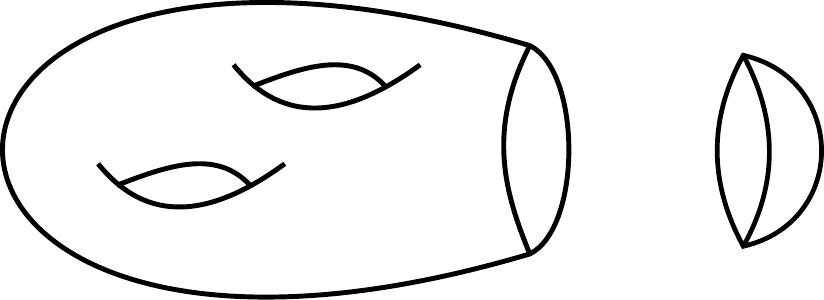}
		\put(40,10){$\Sigma_o$}
	\end{overpic}  
	\caption{The surface $\Sigma_o$ obtained from $\Sigma$ by removing a disk $\mathbb{D}$.}
\label{fig:clSurf}
\end{center}
\end{figure}
 
We denote by $\varphi_o$ the restriction of $\varphi$ to $\Sigma_o$ which has trivial holonomy around the boundary $\partial \Sigma_o$ since the bundle extends to $\Sigma$. Excision now implies that the value of factorisation homology on $\Sigma$ is given by  
the relative tensor product
\begin{align}\label{fhclosedsurface}
\int_{(\Sigma,\varphi)} \mathcal{A} \cong \int_{(\Sigma_o,\varphi_o)} \mathcal{A} \underset{{\int_{(\mathbb{S}^1 \times \R, \ast)} \mathcal{A}}}{\boxtimes}   \mathcal{A}  \ \ ,
\end{align}
where $\ast \colon \mathbb{S}^1 \times \R \longrightarrow B\G$ is the constant map at the base point. Given a decorated gluing pattern for $\Sigma_o$, we showed in Theorem~\ref{Thm: value surface} that one obtains identifications
$$
\int_{(\Sigma_o,\varphi_o)} \mathcal{A} \cong a_P^{\g_1, \dots \g_n}\text{-mod}_\mathcal{A}, \quad \int_{(\mathbb{S}^1\times \R, \ast)} \mathcal{A} \cong \mathcal{F}_{\mathcal{A}}^e\text{-mod}_\mathcal{A} \ \ ,
$$
via monadic reconstruction in $\mathcal{A}$. Now in order to compute the relative tensor product \eqref{fhclosedsurface}, we have to describe the categorical factorisation homology internal to the annulus category $\int_{\mathbb{S}^1 \times \mathbb{R}} \mathcal{A}$. The techniques to do so were developed in~\cite[Section 4]{bzbj2}, and we will briefly review the main results that will be used to compute factorisation homology on closed surfaces with $\G$-bundles.\par 
We first recall the notion of a quantum moment map, see \cite[Section 3]{qmomentmaps} for more details. For every $V \in \mathcal{A}$ we have a natural isomorphism, the so-called \emph{``field goal" isomorphism}~\cite[Corollary 4.6]{bzbj2} :
\begin{equation}\label{Eq: fieldgoal}
	\tau_V \colon \mathcal{F}_\mathcal{A} \otimes V \to V \otimes \mathcal{F}_\mathcal{A} , \quad \tau_V \coloneqq \begin{tikzpicture}
	\begin{pgfonlayer}{nodelayer}
		\node [style=none] (102) at (-7, -6) {};
		\node [style=none] (103) at (-6, -6) {};
		\node [style=none] (104) at (-7, -5) {};
		\node [style=none] (105) at (-6, -5) {};
		\node [style=none] (106) at (-6, -7) {};
		\node [style=none] (107) at (-5, -7) {};
		\node [style=none] (108) at (-6, -6) {};
		\node [style=none] (109) at (-5, -6) {};
		\node [style=none] (110) at (-7, -7) {};
		\node [style=none] (111) at (-7, -6) {};
		\node [style=none] (112) at (-5, -6) {};
		\node [style=none] (113) at (-5, -5) {};
		\node [style=none] (115) at (-6.5, -7.75) {$\mathcal{F}_{\mathcal{A}}$};
		\node [style=none] (117) at (-5, -7.75) {$V$};
		\node [style=none] (119) at (-7, -4.25) {$V$};
		\node [style=none] (120) at (-5.5, -4.25) {$\mathcal{F}_{\mathcal{A}}$};
	\end{pgfonlayer}
	\begin{pgfonlayer}{edgelayer}
		\draw [in=-90, out=90] (103.center) to (104.center);
		\draw [style=over, in=-90, out=90] (102.center) to (105.center);
		\draw [in=-90, out=90] (106.center) to (109.center);
		\draw [style=over, in=-90, out=90] (107.center) to (108.center);
		\draw (110.center) to (111.center);
		\draw (112.center) to (113.center);
	\end{pgfonlayer}
\end{tikzpicture}
 \ \ .
\end{equation}
Now let $A$ be an algebra in $\mathcal{A}$. A \emph{quantum moment map} is an algebra map $\mu \colon A \to \mathcal{F}_\mathcal{A}$ in $\mathcal{A}$ such that it fits into the following commutative diagram 
\begin{center}
\begin{tikzcd}
A \otimes \mathcal{F}_\mathcal{A} \arrow[dd,"\tau_A^{-1}",swap] \arrow[r,"\id \otimes \mu"] & A \otimes A \arrow[dr,"m"] & \\
 & & A \\
\mathcal{F}_\mathcal{A} \otimes A  \arrow[r,"\mu \otimes \id"]& A \otimes A \arrow[ur,"m",swap]
\end{tikzcd} 
\end{center}
It is shown in \cite[Corollary 4.7]{bzbj2} that algebras $A \in \int_{\mathbb{S}^1 \times \mathbb{R}} \mathcal{A}$ amount to the data of a quantum moment map $\mu \colon \mathcal{F}^e_\mathcal{A} \to A$. \par  
As mentioned in Remark~\ref{Rem: braided modules}, braided modules are identified in~\cite{bzbj2} with module categories over $\mathcal{F}_\mathcal{A}^e\text{-mod}_{\mathcal{A}}$, where the latter is equipped with the tensor product $\otimes_{\mathbb{R}}$ induced by stacking annuli in the radial direction. Let now $\mathcal{M}$ be a braided module category and assume there is a progenerator $m \in \mathcal{M}$ for the induced $\mathcal{A}$-action. In the situation at hand, $\mathcal{M} = \int_{(\Sigma_o,\varphi_o)} \mathcal{A}$ and the progenerator is the distinguished object given by the pointing via the inclusion of the empty manifold. The following reconstruction result for $\mathcal{M}$ is proven in~\cite[Theorem 1.1]{bzbj2}: There is an equivalence 
$$
\mathcal{M} \cong A\text{-mod}_{\int_{\mathbb{S}^1 \times \mathbb{R}}\mathcal{A}}, \quad A = \underline{\text{End}}_\mathcal{A}(m) \ \ ,
$$
where the endomorphism algebra comes with a canonical quantum moment map $\mu_{\Sigma_o} \colon \mathcal{F}_{\mathcal{A}} \longrightarrow A$. The right action of $\mathcal{F}_{\mathcal{A}}\text{-mod}_\mathcal{A}$ on $\mathcal{M}$ is then given by~\cite[Corollary 4.7]{bzbj2} :
\begin{align}
A\text{-mod}\boxtimes \mathcal{F}_{\mathcal{A}}\text{-mod} & \longrightarrow  A\text{-mod} \\
V \boxtimes X &\longmapsto V \otimes_{\mathcal{F}_{\mathcal{A}}} X \ \ ,
\end{align}          
where the algebra homomorphism $\mu_{\Sigma_o}$ and the field goal transformations are used to form the relative tensor product. 

\begin{remark}
Conversely, given an algebra $A \in \mathcal{A}$ and a quantum moment map $\mu \colon \mathcal{F}_\mathcal{A} \to A$, the category $\mathcal{M} = A\text{-mod}_\mathcal{A}$ is equipped with the structure of a braided module category. We refer to~\cite[Section 4.3]{bzbj2} for an explicit description of the braided module structure that one obtains from the given quantum moment map $\mu$. 
\end{remark}

Applying the above reconstruction result to the situation at hand, we get quantum moment maps 
\begin{align}
\mu_{\Sigma_o} \colon \mathcal{F}_{\mathcal{A}} \longrightarrow a_P^{\g_1,\dots \g_n} \ \ \text{and } \ \ \mu_{\mathbb{D}} \colon \mathcal{F}_{\mathcal{A}} \longrightarrow 1_\mathcal{A} \ \ ,
\end{align}
which endow $a_P^{d_1,\dots d_n}$ and $1_\mathcal{A}$ with the structure of algebras in $\mathcal{F}_{\mathcal{A}}\text{-mod}_\mathcal{A}$. Finally, by~\cite[Corollary 4.8]{bzbj2}, we get:
\begin{proposition}
The factorisation homology on a closed decorated surface $(\Sigma,\varphi)$ is given by
\begin{align}
\int\displaylimits_{(\Sigma,\varphi)} \mathcal{A} \cong (a_P^{\g_1,\dots , \g_n}\text{-}\mathrm{mod}\text{-}1_\mathcal{A})_{\mathfrak{F}_{\mathcal{A}}\text{-}\mathrm{mod}_\mathcal{A} } \ \ ,
\end{align}     
the category of $a_P^{\g_1,\dots , \g_n}$-$1_\mathcal{A}$-bimodules inside $\mathcal{F}_{\mathcal{A}}\text{-}\mathrm{mod}_\mathcal{A}$.
\end{proposition}
 
\begin{remark}
Let $\{x_1, \dots, x_r\} \subset \Sigma$ be a collection of marked points on the surface and $\varphi \colon \Sigma \setminus \{x_1, \dots, x_r\} \to B\G$ a continuous map. Let $\Sigma_o$ be the surface obtained from $\Sigma$ by removing a small disk $\mathbb{D}^{\g_i}$ around each point $x_i$, where the label $\g_i$ indicates that the holonomy of $\varphi$ around the $i$-th boundary component $\partial_i \Sigma_o$ is given by the group element $\g_i \in \G$. Let $\mathcal{M} = \bigoplus_{\g \in \G} \mathcal{M}_d$ be an equivariant balanced right module over $\mathcal{A}$. Applying excision, we can express factorisation homology over the marked surface $\Sigma$ via the following relative tensor product: 
$$
\int_{((\Sigma,\varphi),\{x_1, \dots, x_r\})} (\mathcal{A},\mathcal{M}) \cong
\int_{(\Sigma_o, \varphi)} \mathcal{A} \underset{\big( \int_{(\mathbb{S}^1,\g_1)}\mathcal{A} \boxtimes \dots \boxtimes \int_{(\mathbb{S}^1,\g_r)} \mathcal{A} \big) }{\boxtimes} \Big( \mathcal{M}_{\g_1} \boxtimes \dots \boxtimes \mathcal{M}_{\g_r} \Big) \ \ .
$$
\end{remark}

\section{Quantisation of flat twisted bundles}\label{Sec: QCV}

In this section we describe the Poisson algebra of functions on the moduli space of flat $\Out(G)$-twisted $G$-bundles on an oriented surface $\Sigma$ and its quantisation via factorisation homology over $\Sigma$ with coefficients in the ribbon category $\Rep_q(G)$ equipped with the $\Out(G)$-action defined in Section~\ref{Ex: Aut-extension}. 

\subsection{The moduli space of flat twisted bundles}\label{Sec:ModuliBunTwisted}

We first recollect some background about twisted bundles in the differential geometric setting, see for example \cite{Meinrenken:Convexity} and \cite{Zerouali} for more details and~\cite{2DYM} for the non-flat version. We refer to~\cite{WildChar} for the original algebraic geometric definition and extension to wild character varieties. Let $\Sigma$ be an oriented surface equipped with a principal $\Out(G)$-bundle $\mathcal{P} \longrightarrow \Sigma$. The group homomorphism $G \rtimes \Out(G) \longrightarrow \Out(G)$, given by projection onto the second factor, induces a morphism of smooth groupoids\footnote{Here smooth groupoids can, for example, be modelled as sheaves of groupoids on the site of Cartisan space as in~\cite[Section 5.1]{2Group}. We will not go into the details here because they will not be important for what follows.} $\Bun^{\text{flat}}_{\Out(G)\rtimes G}(\Sigma) \longrightarrow \Bun_{\Out(G)}(\Sigma)$. The \emph{groupoid of flat $\mathcal{P}$-twisted $G$-bundles} is defined as the homotopy pullback
\begin{equation}
\begin{tikzcd}
\Bun_{G\downarrow \mathcal{P}}^{\text{flat}} (\Sigma) \ar[d] \ar[r] & \Bun^{\text{flat}}_{G\rtimes \Out(G)} (\Sigma) \ar[d] \\ 
\star \ar[r,"\mathcal{P}"]  & \Bun_{\Out(G)}(\Sigma) \ \ .
\end{tikzcd}
\end{equation} 
The \emph{trivial $\mathcal{P}$-twisted $G$-bundle} is the bundle $\mathcal{P}\times_{\Out(G)} (G\rtimes \Out(G))$ associated to $\mathcal{P}$ using the group homomorphism $\Out(G)\hookrightarrow G\rtimes \Out(G)$, $\kappa \longmapsto 1\rtimes \kappa$. 
Note that the automorphisms of the trivial flat $\mathcal{P}$-twisted $G$-bundle are $G^{\pi_0(\Sigma)}$ and not $(G\rtimes \Out(G))^{\pi_0(\Sigma)}$ as one might naively expect.

\begin{remark}
The moduli space of flat $\Out(G)$-twisted bundles on a closed surface $\Sigma$ was studied in the differential geometric setting in~\cite{Meinrenken:Convexity,Zerouali}. In particular, it is shown in \emph{loc.~cit.~}that the moduli space of $\Out(G)$-twisted flat bundles for a compact Lie group $G$ carries a canonical Atiyah-Bott like symplectic structure. Similar symplectic structures have been constructed in the algebraic geometric setting on of (wild) twisted character varieties in~\cite{WildChar}.  
\end{remark} 

Since in this paper we obtain our results in the algebraic setting, we will now give another description of flat twisted bundles that is more suitable for us, namely the holonomy description of twisted $G$-bundles. We will only consider surfaces $\Sigma$ with at least one boundary component and a marked point $v \in \partial \Sigma$ on one of the boundary circles. For brevity we write simply $\pi_1(\Sigma)$ for $\pi_1(\Sigma,v)$. For any group $G$, we call the space of group homomorphisms $\Hom(\pi_1(\Sigma),G)$ the $G$-representation variety. It comes with a natural action of $G$ via conjugation: $g.\varphi(\gamma) = g\varphi(\gamma)g^{-1}$ for all $g \in G$, $\gamma \in \pi_1(\Sigma)$ and $\varphi \in \Hom(\pi_1(\Sigma),G)$. As before, we fix a principal $\Out(G)$-bundle, here described by a group homomorphism $\rho \colon \pi_1(\Sigma)\longrightarrow \Out(G)$. Such a map $\rho$ is given by picking an element $\kappa \in \Out(G)$ for every generator in $\pi_1(\Sigma)$. Then, an element in the \emph{$\rho$-twisted $G$-representation variety} is a lift
\begin{equation}
\begin{tikzcd}
 & G\rtimes \Out(G) \ar[d] \\ 
 \pi_1(\Sigma) \ar[r, "\rho", swap] \ar[ru, dashed] & \Out(G) \ \ .
\end{tikzcd}
\end{equation}  
We write $\Hom_\rho(\pi_1(\Sigma),G)$ to denote the space of lifts. 
Concretely, elements in $\Hom_\rho(\pi_1(\Sigma),G)$ can be described by maps $\varphi \colon \pi_1(\Sigma) \longrightarrow G$, which are such that $\varphi(\gamma_1 \circ \gamma_2) = 
\varphi(\gamma_1) \rho(\gamma_1).\varphi(\gamma_2)$. The group $G$ acts via twisted conjugation, i.e.~the action of an element $g \in G$ is given by $\varphi(\gamma) \longmapsto g \varphi(\gamma) \rho(\gamma).g^{-1}$. Given a set $E$ of free generators of $\pi_1(\Sigma)$, we get an identification $\Hom_\rho(\pi_1(\Sigma),G) \cong G^E$. \par  
There is a bijective correspondence between elements in the twisted representation variety $\Hom_\rho(\pi_1(\Sigma),G)$ and elements in 
$$
\mathcal{M}^\circ_\rho(\Sigma) \coloneqq \{\text{Isomorphism classes of flat twisted }G\text{-bundles with trivialisation over }v \in \Sigma\} \ \ ,
$$
which is established via the holonomy map. The group $G$ acts on $\mathcal{M}_\rho^\circ(\Sigma)$ by changing the trivialisation. The moduli space of flat twisted bundles is then given by the quotient stack
$$
\mathcal{M}_\rho(\Sigma) = \mathcal{M}_\rho^\circ(\Sigma) /^\rho G  \ \ ,
$$
where the notation $/^\rho$ indicates that $G$ acts via twisted conjugation.

\subsubsection{The twisted Fock-Rosly Poisson structure}

For the remainder of this section, $\Sigma$ is a connected surface with at least one boundary component. We will give an explicit description of the Poisson structure on  $\mathcal{M}^\circ_\rho(\Sigma)$, following the strategy of Fock and Rosly~\cite{FRPoisson} using lattice gauge theory. \par 
We choose a ciliated fat graph model for $\Sigma$ with one vertex and edges $E= \{e_1,\dots , e_n\}$, constructed from a gluing
pattern for $\Sigma$ as defined in Section \ref{sec:compFH}. Furthermore, we choose an $\Out(G)$-labeling $\{\kappa_1,\dots , \kappa_n\}$ of the gluing pattern describing the twisting principal $\Out(G)$-bundle $\rho$. The fundamental group of $\Sigma$ is freely generated by the edges $E$ of the graph model, as depicted in Figure \ref{fig:graphmodel}. Using the holonomy description from the previous section, we can characterise a $\rho$-twisted bundle on $\Sigma$ by a graph connection, that is a labeling of every edge $e_i \in E$ with a group element $g_i\in G$:
$$
\text{hol} \colon \mathcal{M}_{\rho}^\circ(\Sigma) \xrightarrow{\cong } \Hom_{\rho}(\pi_1(\Sigma, v),G) = G^E \ \ .
$$
This identification chooses an orientation for every edge in the fat graph model which we choose to agree with the natural orientation coming from the gluing pattern. Hence, we get an identification 
$$
\mathcal{M}_{\rho}(\Sigma)\cong G^E /^{\rho} G \ \ ,
$$
where $h\in G$ acts via twisted conjugation 
\begin{equation}
(g_{e_1},\dots g_{e_n}) \longmapsto (h g_{e_1} \kappa_1(h)^{-1},\dots, h g_{e_n} \kappa_n(h)^{-1}) \ \ .
\end{equation} 

In this way, we consider the algebraic functions on $G^E$ as an element of $\Rep(G)$ and we denote this algebra by $\mathcal{O}^\rho(G^E)$. Quasi-coherent sheaves on $\mathcal{M}_\rho(\Sigma)$ can now be identified with modules over $\mathcal{O}^\rho(G^E)$ in $\Rep(G)$. 

\begin{proposition}
Let $\Sigma$ be a surface of genus $g$ and with $r \geq 1$ boundary 
components. Given a principal $\Out(G)$-bundle $\rho\colon \pi_1(\Sigma) \longrightarrow \Out(G)$, described by the elements $\kappa_1,\dots,\kappa_{2g+r-1}\in \Out(G)$, and a gluing pattern $P$ for $\Sigma$, there is an isomorphism $\mathcal{O}^\rho(G^{2g+r-1}) \cong a_P^{\kappa_1, \dots , \kappa_{2g+r-1}}$ of algebras in $\Rep(G)$. 
\end{proposition}  
   
\begin{proof}
To establish the isomorphism on the level of vector spaces, we use the algebraic Peter-Weyl theorem:
$$
\mathcal{O}(G) \cong \bigoplus_{V} V^\vee \otimes V \ \ ,
$$
where the sum on the right hand side is over all irreducible representations of $G$ and $\mathcal{O}(G)$ is the Hopf algebra of matrix coefficients of irreducible $G$-representations. Next we take into account the twist by a given automorphism $\kappa \in \Out(G)$: a group element $h \in G$ acts on $\phi \in \mathcal{O}^\kappa(G)$ via $h \triangleright \phi = \phi(h^{-1}(-)\kappa(h))$. As explained in Example \ref{ex:CoendRep(H)}, we thus get an isomorphism $\mathcal{O}^\kappa(G) \cong \bigoplus_{V} V^\vee \otimes \kappa^*V = \mathcal{F}_{\Rep(G)}^\kappa$ compatible with the $G$-action. 
\end{proof}

In combination with Theorem \ref{Thm: value surface}, the above result shows that $\int_{(\Sigma, \rho)} \Rep(G)$ agrees with the category of quasi-coherent sheaves on the moduli space $\mathcal{M}_{\rho}(\Sigma)$ of twisted bundles. Note that $G^{E}$ is a finite dimensional smooth algebraic variety and independent of the concrete form of the gluing pattern or topology of $\Sigma$. However, we will see shortly that the Poisson structure is sensitive to the topology. 

\begin{figure}[h!]
\begin{center}
\vspace{0.25cm}
	\begin{overpic}[scale=0.4
		,tics=10]
		{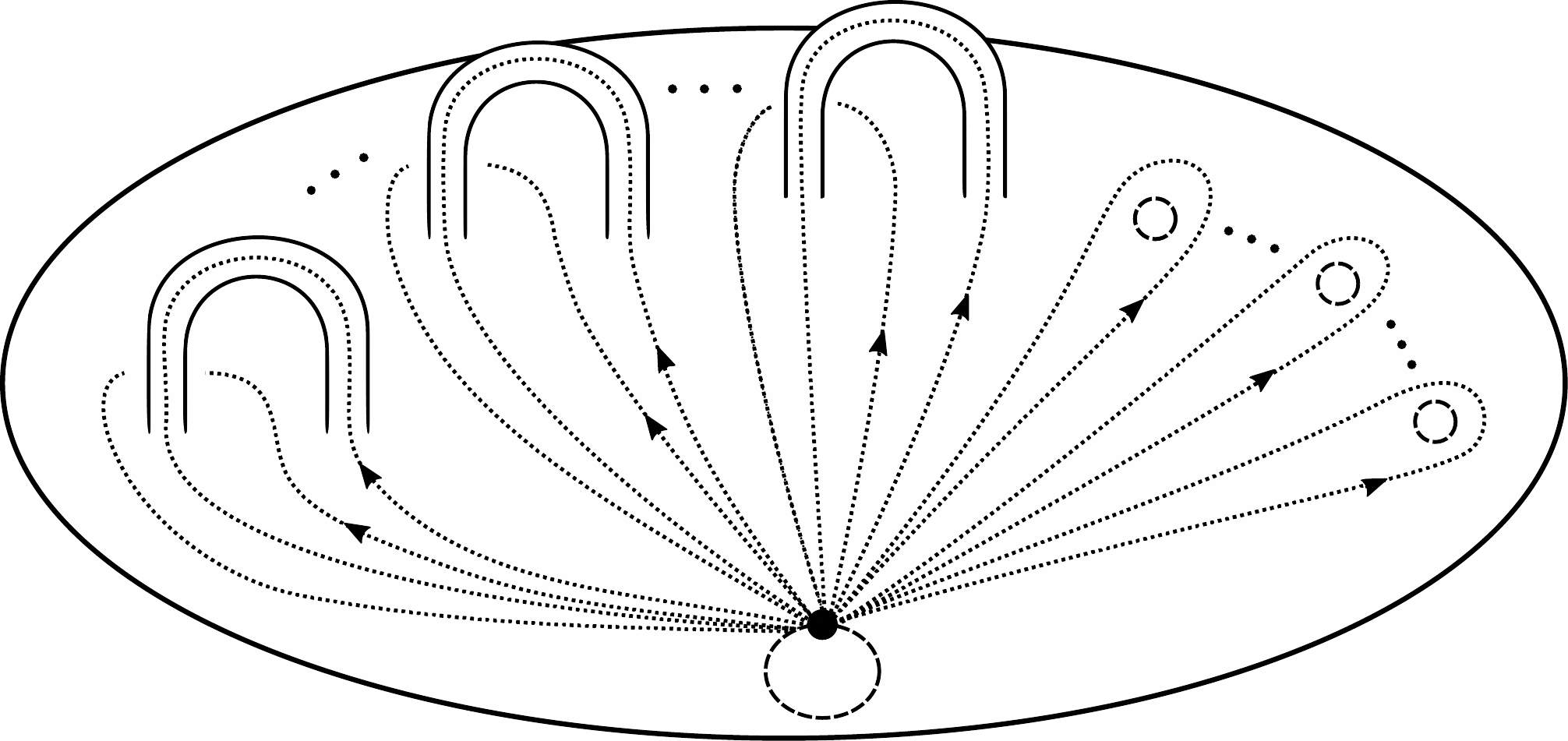}
	\end{overpic}  
	\vspace{0.5cm}
	\caption{Generators of the fundamental group for an $r$-punctured genus $g$ surface.}
\label{fig:graphmodel}
\end{center}
\end{figure}

In order to describe the Poisson structure on the representation variety $\mathcal{M}_\rho^\circ(\Sigma)$, we notice that there is an equivariant embedding 
\begin{align}
\iota \colon G^{E} & \longrightarrow (G\rtimes \Out(G))^E \\
(g_{e_1},\dots g_{e_n}) & \longmapsto (g_{e_1}\rtimes \kappa_1,\dots g_{e_n} \rtimes \kappa_n)
\end{align}
which identifies $G^E$ with a connected component of $(G\rtimes \Out(G))^E$ since $\Out(G)$ is discrete. The $G$-action on the right side is via the embedding $G\longrightarrow G\rtimes \Out(G)$ and conjugation inside $G\rtimes \Out(G)$. Using the gluing pattern for $\Sigma$, together with the choice of an $\Out(G)$-invariant classical r-matrix\footnote{For example, the semi-classical limit of the quantum R-matrix $\mathcal{R}$ of $U_\hbar(\mathfrak{g})$ is $\Out(G)$-invariant, see Proposition \ref{prp:OutGactionRephG}.} $r\in \left(\mathfrak{g} \otimes \mathfrak{g}\right)^{\Out(G)}$, Fock and Rosly's construction~\cite{FRPoisson} gives a Poisson structure $\pi_{\text{FR}}$ on $(G\rtimes \Out(G))^E$, such that the action of $G\rtimes \Out(G)$ is Poisson-Lie. Pulling back $\pi_{\text{FR}}$ along $\iota$, we get the desired Poisson structure on $\mathcal{M}_\rho^{\circ}(\Sigma)$, which is compatible with the twisted $G$-action. In Proposition \ref{prp:twistedFR} below we give an explicit formula for the Poisson structure $\pi_{\Ma_\rho^\circ(\Sigma)}$ we just described on $\mathcal{M}_{\rho}^\circ(\Sigma)$, which is a twisted version of the Fock-Rosly Poisson structure on $G^E$ given in \cite[Proposition 3]{FRPoisson}. 

\begin{proposition}
Let the surface $\Sigma$ be represented by a ciliated fat graph with one vertex $v$ and a set $E$ of edges. Let $(x_i)_{i = 1, \dots, \text{dim}(\mathfrak{g})}$ be a basis of $\mathfrak{g}$. Then for a given choice $r = r^{ij} x_i \otimes x_j \in \left(\mathfrak{g} \otimes \mathfrak{g}\right)^{\Out(G)}$ of $\Out(G)$-invariant classical $r$-matrix there is a Poisson structure on $\Ma_\rho^\circ(\Sigma)$ given by the bivector
$$
\pi_{\Ma_\rho(\Sigma)} = \sum_{\alpha \prec \beta}  r^{ij} x_i(\alpha) \wedge x_j(\beta) + \frac{1}{2} \sum_{\alpha} r^{ij} x_i(\alpha) \wedge x_j(\alpha)
$$
where $\alpha$ and $\beta$ run over the set of half-edges\footnote{We break up the edges of the graph, so that from each edge we get an incoming and an outgoing half-edge at the vertex $v$. Since the chosen graph is ciliated, we get an ordering $\prec$ on the set of half-edges.} and
$$
x_i(\alpha) \coloneqq \begin{cases} -x_i^R(\alpha), & \alpha \text{ is incoming at v} \\
(\kappa_\alpha)_*x_i^L(\alpha), & \alpha \text{ is outgoing at v}
\end{cases}
$$
where $x^{R/L}_i(\alpha)$ denotes the right/left-invariant vector field of $x_i$ acting on the $\alpha$-copy of $G^E$. Furthermore, the induced Poisson structure on the subalgebra of $G$-invariant functions is independent of the chosen fat graph model for $\Sigma$.
\end{proposition}\label{prp:twistedFR}
 
\subsection{Quantisation}

In Section \ref{sec:compFH} we constructed an algebra $a_P^{\kappa_1, \dots, \kappa_n}$, $n = 2g + r -1$, from a combinatorial presentation of the decorated surface $\Sigma$. We now explain how these algebras provide a deformation quantisation of the twisted Fock-Rosly Poisson structure on $\mathcal{M}_\rho^\circ(\Sigma)$. To that end, we consider $a_P^{\kappa_1, \dots, \kappa_{n}}$ as an object in the representation category $\Rep_\hbar(G)$ of the formal quantum group. It is the tensor product $\bigotimes_{i=1}^{n} \mathcal{O}_\hbar^{\kappa_i}(G)$, where each $\mathcal{O}_\hbar^{\kappa_i}(G)$ is a $\kappa_i$-twisted REA of quantised coordinate functions. The multiplication on the tensor product is defined in terms of the crossing morphisms depicted in Figure \ref{Fig:crossingmorphisms}. We will show in Theorem \ref{Thm:QuantisationTwistedFR} that for all elements $f_{\hbar}^{\kappa_i} \in \mathcal{O}_\hbar^{\kappa_i}$ and $g_\hbar^{\kappa_j} \in \mathcal{O}_\hbar^{\kappa_j}$ we have
$$
\frac{[f_\hbar^{\kappa_i},g_\hbar^{\kappa_j}]}{\hbar} \text{ mod}(\hbar) = \{ f^{\kappa_i},g^{\kappa_j} \} \ \ ,
$$
where $\{ \cdot , \cdot \}$ is the twisted Fock-Rosly Poisson structure from Proposition \ref{prp:twistedFR}, and $f^{\kappa_i} = f_\hbar^{\kappa_i} \text{ mod}(\hbar) \in \mathcal{O}^{\kappa_i}(G)$, and similarly for $g^{\kappa_j}$.
\par 
We present a reformulation of the Poisson structure on $\mathcal{M}_\rho^\circ(\Sigma)$ that will prove useful for what follows. Let $r = \omega + t$ be the decomposition of the classical r-matrix into an anti-symmetric part $\omega$ and an invariant symmetric element $t$. For a given automorphism $\kappa \in \Out(G)$, define the bivector field 
\begin{align}
\pi^\kappa_{\text{STS}} \coloneqq \omega^{\ad(\kappa),\ad(\kappa)}+t^{R,L(\kappa)}-t^{L(\kappa),R} \ \ ,
\end{align}
where the superscripts indicate that the action by left-invariant vector fields is twisted by the automorphism $\kappa$, and we used the notation $x^{\ad(\kappa)} = x^R - \kappa_*x^L$ for the vector field
generated by the element $x \in \mathfrak{g}$ via the twisted adjoint action $h \mapsto gh\kappa(g^{-1})$ of $G$ on itself. In the case $\kappa = e$, the bivector field $\pi^e_{\text{STS}}$ agrees with the Semenov-Tian-Shansky (STS) Poisson structure on $G$, see~\cite{STSPoisson}. Using the decorated gluing pattern $(P, \{\kappa_1, \dots, \kappa_{2g+n-1}\})$ for $\Sigma$, we define the bivector
\begin{align}
	\pi = \sum_{\alpha \in E} \pi_{\text{STS}}^{\kappa_\alpha} + \sum_{\alpha < \beta} \pi_{\alpha, \beta}- \pi_{\beta, \alpha} \ \ ,
\end{align}  
where $\pi_{\alpha, \beta}$ is a 2-tensor, acting on the $\alpha$-component of the first factor and on the $\beta$-component of the second factor of $G^E \times G^E$, and is defined by
\begin{align}
	\pi_{\alpha ,\beta} \coloneqq \begin{cases} 
		- r_{2,1}^{\ad(\kappa_\alpha),\ad(\kappa_\beta)} & \text{, if $\alpha$ and $\beta$ are positively unlinked} \\
		- r_{2,1}^{\ad(\kappa_\alpha),\ad(\kappa_\beta)} - 2t^{L(\kappa_\alpha),R} & \text{, if $\alpha$ and $\beta$ are positively linked} \\
		-r_{2,1}^{\ad(\kappa_\alpha),\ad(\kappa_\beta)} - 2t^{L(\kappa_\alpha),R} + 2t^{L(\kappa_\alpha),L(\kappa_\beta)} & \text{, if $\alpha$ and $\beta$ are positively nested}
	\end{cases}
\end{align}
And similarly, the 2-tensor $\pi_{\beta, \alpha}$ acts on the $\beta$-component of the first factor and on the $\alpha$-component of the second factor of $G^E \times G^E$ and is defined as $\pi_{\beta, \alpha} = \tau(\pi_{\alpha, \beta})$, where $\tau$ swaps the two tensor factors. Similarly, for the remaining three cases, we define 
\begin{align}
	\pi_{\alpha ,\beta} \coloneqq \begin{cases} 
		r_{1,2}^{\ad(\kappa_\alpha),\ad(\kappa_\beta)} & \text{, if $\alpha$ and $\beta$ are negatively unlinked} \\
		r_{1,2}^{\ad(\kappa_\alpha),\ad(\kappa_\beta)} + 2t^{R,L(\kappa_\beta)} & \text{, if $\alpha$ and $\beta$ are negatively linked} \\
	r_{1,2}^{\ad(\kappa_\alpha),\ad(\kappa_\beta)} + 2t^{R,L(\kappa_\beta)} - 2t^{L(\kappa_\alpha),L(\kappa_\beta)} & \text{, if $\alpha$ and $\beta$ are negatively nested}
	\end{cases}
\end{align}
and set again $\pi_{\beta,\alpha} = \tau(\pi_{\alpha,\beta})$. A direct computation shows that $\pi$ agrees with the twisted Fock-Rosly Poisson structure defined in Proposition~\ref{prp:twistedFR}.

\begin{theorem}\label{Thm:QuantisationTwistedFR}
The algebra $a_P^{\kappa_1, \dots, \kappa_{2g+r-1}}$ is a quantisation of the twisted Fock-Rosly Poisson structure on $\mathcal{M}^\circ_\rho(\Sigma) \cong G^{2g + r -1}$. Its subalgebra of $U_\hbar(\mathfrak{g})$-invariants does not depend on the choice of the gluing pattern $P$ and is a quantisation of the Poisson structure on the affine quotient $\mathcal{M}_\rho^\circ(\Sigma) \DS G$.
\end{theorem}

\begin{proof}
First, we show that the quasi-classical limit of the commutator of two quantised functions in $\mathcal{O}_\hbar^{\kappa}(G)$ agrees with the $\kappa$-twisted STS Poisson structure $\pi_{\text{STS}}^\kappa$. We recall from Example \ref{ex:CoendRep(H)} that the multiplication in the $\kappa$-twisted REA $\mathcal{O}_\hbar^{\kappa}(G)$ is related to the multiplication in the FRT-algebra via a twisting cocycle given in terms of R-matrices. The commutator in the (untwisted) FRT-algebra $H^\circ$, $H = U_\hbar(\mathfrak{g})$, can be computed by acting with 
$$
(1 \otimes^{\rev} 1) \boxtimes (1\otimes 1) - (\mathcal{R}_2^{-1} \otimes^{\rev} \mathcal{R}^{-1}_1 )  \boxtimes (\mathcal{R}'_2 \otimes \mathcal{R}'_1) 
$$
on the components  $V^\vee \otimes^{\rev} W^\vee \boxtimes V \otimes W$, for $V, W \in \Rep_\hbar(G)$, since the multiplication in the FRT-algebra is given by the Hopf pairing $\langle - , - \rangle$ between $H^\circ$ and $H$:
$$
\langle m_{\text{FRT}}(\phi \psi), h \rangle = \langle \phi \otimes \psi, \Delta(h) \rangle , \quad \phi, \psi \in H^\circ, h \in H
$$
and $\Delta(-) = \mathcal{R}^{-1} \Delta^{\text{op}}(-) \mathcal{R}$. Now we take into account the twist by $\kappa$, as well as the twisting cocycle $\mathcal{R}'_1 \otimes \kappa.\mathcal{R}_1 \otimes \mathcal{R}'_2 \mathcal{R}_2 \otimes 1$, to compute the commutator in $\mathcal{O}_\hbar^\kappa(G)$ component-wise by acting with 
\begin{align}\label{eq:commutator}
& (\mathcal{R}'_1 \otimes^{\rev} \mathcal{R}'_2 \mathcal{R}_2) \boxtimes  ( \kappa.\mathcal{R}_1 \otimes 1) - C \circ (\mathcal{R}'_2 \mathcal{R}_2 \otimes^{\rev} \mathcal{R}'_1) \boxtimes (1 \otimes \kappa.\mathcal{R}_1)  \\
& \text{where }C = (\mathcal{R}_2^{-1} \otimes^{\rev} \mathcal{R}^{-1}_1 )  \boxtimes (\kappa.\mathcal{R}'_2 \otimes \kappa.\mathcal{R}'_1) 
\end{align}
on $V^\vee \otimes^{\rev} W^\vee \boxtimes V \otimes W$. To compute the quasi-classical limit of the action \eqref{eq:commutator}, we use that in the limit $\exp(\hbar) \to 1$, the R-matrix has the following expansion: $\mathcal{R} = 1 + \hbar r + \mathcal{O}(\hbar^2)$, where $r = r_1 \otimes r_2 \in \mathfrak{g}^{\otimes 2}$ is the classical r-matrix. Explicitly, the quasi-classical limit of \eqref{eq:commutator} is
$$
r^{3(\kappa),2} + r^{1,2} - r^{4(\kappa),1} - r^{2,1} + r^{2,1} - r^{4,3}  \in U(\mathfrak{g})^{\otimes 4}  \ \ ,
$$ 
where for instance $r^{3(\kappa),2} = 1 \otimes r_2 \otimes r_1^{\kappa} \otimes 1 \in U(\mathfrak{g})^{\otimes 4}$ and the superscript $\kappa$ means that the respective action will be twisted by $\kappa$. More explicitly, the first two copies of $U(\mathfrak{g})^{\otimes 4}$ act on $\mathcal{O}^\kappa(G)$ via $x \mapsto x^r$, for $x \in \mathfrak{g}$, and the last two copies act via $x \mapsto -\kappa_*x^L$. Thus, we find that the quasi-classical limit of the commutator is the bivector field on $G$ given by
\begin{align*}
-r^{L(\kappa),R} + r^{R,R} + r_{2,1}^{R,L(\kappa)} - r_{2,1}^{L,L}
& = \omega^{\ad(\kappa),\ad(\kappa)} + t^{R,L(\kappa)} - t^{L(\kappa),R} \\
& = \pi_{\text{STS}}^\kappa \ \ ,
\end{align*}
where we used that $r^{R,R} - r_{2,1}^{L,L} = \omega^{R,R} + \omega^{L,L}$.\par
Next, we prove the claim for two positively unlinked edges $\alpha < \beta$. We recall that the crossing morphism for two unlinked edges $\alpha < \beta$ is given by acting on $\mathcal{O}^{\kappa_\beta}_\hbar(G) \otimes \mathcal{O}^{\kappa_\alpha}_\hbar(G)$ with 
\begin{align*}
U^+ & = \tau_{12,34} \circ (\mathcal{R}_1 \otimes 1 \otimes 1 \otimes \kappa_\alpha.\mathcal{R}_2)(1 \otimes \kappa_\beta.\mathcal{R}_1 \otimes 1 \otimes \kappa_\alpha.\mathcal{R}_2)(\mathcal{R}_1 \otimes 1 \otimes \mathcal{R}_2 \otimes 1)(1 \otimes \kappa_\beta.\mathcal{R}_1 \otimes \mathcal{R}_2 \otimes 1) \\
& \coloneqq \tau_{12,34} \circ \widetilde{U}^+
\end{align*}
Hence, the commutator on components $\phi \otimes \kappa_\alpha^*v \in \mathcal{O}^{\kappa_\alpha}_\hbar(G)$ and $\psi \otimes \kappa_\beta^*w \in \mathcal{O}^{\kappa_\beta}_\hbar(G)$ can be computed via
$$
(m_{\mathcal{O}^{\kappa_\alpha}_\hbar(G)} \otimes m_{\mathcal{O}^{\kappa_\beta}_\hbar(G)}) \circ (1 - (U^+)^{7,8,1,2}) (\phi \otimes \kappa_\alpha^*v \otimes 1^{\otimes 4} \otimes \psi \otimes \kappa_\beta^*w ) \ \ .
$$
Taking the quasi-classical limit of this action thus amounts to 
\begin{equation}\label{eq:commutator2}
\frac{1-\tau(\widetilde{U}^+)}{\hbar}~\text{mod}(\hbar) = -r^{3,2(\kappa_\alpha)} - r^{4(\kappa_\beta),2(\kappa_\alpha)} - r^{3,1} - r^{4(\kappa_\beta),1} \in U(\mathfrak{g})^{\otimes 4} \ \,
\end{equation}
where this time the first and third copy in $U(\mathfrak{g})^{\otimes 4}$ act via $x \mapsto x^R$ and the second and the forth copy via $x \mapsto -\kappa_*x^L$, so that the right hand side of \eqref{eq:commutator2} acts on $\mathcal{O}^{\kappa_\alpha}(G) \otimes \mathcal{O}^{\kappa_\beta}(G)$ via $-r_{2,1}^{\ad(\kappa_\alpha),\ad(\kappa_\beta)}$, which agrees with $\pi_{\alpha,\beta}$ as claimed. Similarly, for two positively linked edges we have
$$
\frac{1-\tau(\widetilde{L}^+)}{\hbar} \text{ mod}(\hbar)  = r^{2(\kappa_\alpha),3} - r^{4(\kappa_\beta),2(\kappa_\alpha)} - r^{3,1} - r^{4(\kappa_\beta),1} \ \,
$$
and we see that in the positively linked case the 2-tensor $\pi_{\alpha, \beta}$ differs from the unlinked case by adding a term $-2 t^{L(\kappa_\alpha),R}$. Lastly, for two positively nested edges we find  
$$
\frac{1-\tau(\widetilde{N}^+)}{\hbar} \text{ mod}(\hbar)  = r^{2(\kappa_\alpha),3} + r^{2(\kappa_\alpha),4(\kappa_\beta)} - r^{3,1} - r^{4(\kappa_\beta),1} \ \,
$$
which differs from the linked case by adding the term $2t^{L(\kappa_\alpha),L(\kappa_\beta)}$, which ends the proof for the positively unlinked, linked and nested case. The remaining three cases can be worked out analogously.
\end{proof}

\end{document}